\RequirePackage{fix-cm}
\documentclass[smallextended,envcountsect,envcountsame]{svjour3}
\smartqed
\usepackage{needspace}
\usepackage[T1]{fontenc}
\usepackage[utf8]{inputenc}
\usepackage{amsmath,amssymb,mathtools,graphicx}
\numberwithin{equation}{section}
\usepackage[dvipsnames]{xcolor}
\usepackage{enumitem,booktabs,tabularx,array,arydshln,placeins,multicol,comment,etoolbox,stackrel}
\AtEndEnvironment{proof}{\qed}
\usepackage{tikz-cd}
\usepackage{xurl}
\usepackage[hidelinks]{hyperref}
\usepackage{cite}
\usepackage{aliascnt}

\newaliascnt{lemma}{theorem}
\spnewtheorem{lemma}[lemma]{Lemma}{\bfseries}{\itshape}
\aliascntresetthe{lemma}

\newaliascnt{proposition}{theorem}
\spnewtheorem{proposition}[proposition]{Proposition}{\bfseries}{\itshape}
\aliascntresetthe{proposition}

\newaliascnt{remark}{theorem}
\spnewtheorem{remark}[remark]{Remark}{\bfseries}{\rmfamily}
\aliascntresetthe{remark}

\newaliascnt{example}{theorem}
\spnewtheorem{example}[example]{Example}{\bfseries}{\rmfamily}
\aliascntresetthe{example}

\newaliascnt{definition}{theorem}
\spnewtheorem{definition}[definition]{Definition}{\bfseries}{\rmfamily}
\aliascntresetthe{definition}

\newaliascnt{corollary}{theorem}
\spnewtheorem{corollary}[corollary]{Corollary}{\bfseries}{\itshape}
\aliascntresetthe{corollary}

\newaliascnt{question}{theorem}
\spnewtheorem{question}[question]{Question}{\bfseries}{\itshape}
\aliascntresetthe{question}

\newaliascnt{conjecture}{theorem}
\spnewtheorem{conjecture}[conjecture]{Conjecture}{\bfseries}{\itshape}
\aliascntresetthe{conjecture}
\usepackage{microtype}
\newcommand{\doi}[1]{\url{https://doi.org/#1}}
\spnewtheorem*{theorem*}{Theorem}{\bfseries}{\itshape}
\spnewtheorem*{proposition*}{Proposition}{\bfseries}{\itshape}
\spnewtheorem*{conjecture*}{Conjecture}{\bfseries}{\itshape}

\makeatletter
\def\makeheadbox{}
\makeatother

\newcommand*{\tmin}{\dot{\otimes}}
\newcommand*{\tmax}{\hat{\otimes}}
\DeclareMathOperator{\intr}{int}
\DeclareMathOperator{\id}{id}

\newcommand*{\lrbracket}[1]{\left(#1\right)}
\newcommand*{\lrbrace}[1]{\left\lbrace #1 \right\rbrace}
\newcommand*{\lrvert}[1]{\left\lvert #1 \right\rvert}
\newcommand*{\lrrec}[1]{\left[#1\right]}
\newcommand*{\MM}{\mathcal{M}}
\newcommand*{\Acal}{\mathcal{A}}
\newcommand*{\Bcal}{\mathcal{B}}
\newcommand*{\Xcal}{\mathcal{X}}
\newcommand*{\Ycal}{\mathcal{Y}}

\begin{document}
\title{Finite de Finetti for convex bodies and polynomial optimization\thanks{\textbf{Acknowledgements.} GK thanks Mario Berta for pointing out \cite{Berta_2011}. GK thanks Ludovico Lami for a discussion about the paper. GK and JZ acknowledge support from the Excellence Cluster - Matter and Light for Quantum Computing (ML4Q-2) and by the European Research Council (ERC Grant Agreement No. 948139).
MP acknowledges support from the Niedersächsisches Ministerium für Wissenschaft und Kultur.
RS is supported by the Deutsche Forschungsgemeinschaft (DFG) under Germany's Excellence Strategy - EXC-2123 QuantumFrontiers - 390837967 and  SFB 1227 (DQ-mat), the Quantum Valley Lower Saxony, and the Federal Ministry of Education and Research (BMBF) projects ATIQ, SEQUIN, Quics and CBQD.}}
\titlerunning{Finite de Finetti for convex bodies and polynomial optimization}
\author{Julius A. Zeiss \and Gereon Kossmann \and Ren\'e Schwonnek \and Martin Pl\'avala}
\authorrunning{J. A. Zeiss, G. Kossmann, R. Schwonnek, M. Pl\'avala}
\institute{J. A. Zeiss \and G. Kossmann \at
Institute for Quantum Information, RWTH Aachen University, Aachen, Germany\\
\email{jzeiss@physik.rwth-aachen.de} (J. A. Zeiss, corresponding author)\\
\email{kossmann@physik.rwth-aachen.de} (G. Kossmann)\\
J. A. Zeiss and G. Kossmann contributed equally to this work.
\and R. Schwonnek \and M. Pl\'avala \at
Leibniz Universit\"at Hannover, Hannover, Germany}
\date{}
\maketitle
\enlargethispage{3\baselineskip}
\begin{abstract}
We study polynomial optimization over convex bodies with separately affine objectives and local affine equality and inequality constraints. We construct a convergent hierarchy of conic relaxations with explicit finite-level error bounds of order $n^{-1/2}$ at level $n$, complemented by a constructive rounding scheme that produces feasible points. The relaxation bounds and rounded objective values bracket the optimum at the same rate. The analytic ingredient is a finite de Finetti representation theorem for arbitrary proper cones: bipartite marginals of symmetric extensions in maximal tensor products are $O(n^{-1/2})$-close, in base norm, to points of the minimal tensor product. Its proof combines a uniform, monogamy-type bound on a variational mutual information based on Frenkel's relative entropy with classical conditioning and informationally complete measurements. We apply the framework to the assemblage relaxation of two-player non-local games in general probabilistic theories. For polytopic state spaces, the hierarchy consists of linear programs, and experiments on the Clauser--Horne--Shimony--Holt (CHSH) game provide a numerical proof of concept for the outer approximations and feasible rounding procedure.
\end{abstract}
\section{Introduction}

This paper concerns the approximation of multivariate polynomial optimization problems over convex bodies, in the lineage of the Lasserre, Sherali--Adams, and NPA (Navascu\'es--Pironio--Ac\'in) moment hierarchies and of their conic relatives \cite{Lasserre2001,Sherali1990,Sherali1992,Berta2021,plavala2025polarization,Pironio2010}. Quantum information and the more general framework of general probabilistic theories are a rich source of such problems \cite{Berta2021,plavala2025polarization,zeiss2025approximatingfixedsizequantum,kossmann2025symmetric}, and we return to them below as both a motivation and a toolbox; but the objects we manipulate, and the guarantees we prove, are convex-geometric.

Concretely, a convex body is a compact convex subset $K$ of a finite dimensional real vector space $V$. After the canonical evaluation embedding into $A(K)^\star$ (equivalently, homogenization $x\mapsto(x,1)$), each nonempty body $K$ is identified with the base of the proper cone $C=\operatorname{cone}(K)$, singled out by an order unit (unit functional) $1_K$ via $K=\{x\in C : 1_K(x)=1\}$; we use the symbol $1_K$ for this order unit throughout (cf.\ \autoref{sec:notation_and_preliminaries}). Thus a convex body and the proper cone it generates carry the same information, and we pass freely between the body $K$ and its cone $C$. For $i\in [m]$, let $K_i$ be convex bodies such that a multivariate polynomial objective function
\begin{align}
    p\,:\, K_1\times \ldots\times K_m\,\rightarrow\, \mathbb{R}
\end{align}
is optimized subject to collections of $q$ and $s$ multivariate polynomial constraints,
\begin{align}
 f\,:\, K_1\times \ldots K_m \rightarrow \mathbb{R}^{q},\quad g\,:\, K_1\times \ldots K_m \rightarrow \mathbb{R}^{s},  
\end{align}
respectively, resulting in 
\begin{align}\label{eq:multivariate_problem}
\begin{split}
    p^{*}=\min_{x_i\in K_i}\quad & p\lrbracket{\left\{ x_i\right\}_{i=1}^m}\\
    \text{s.t.}\quad & f\lrbracket{\left\{ x_i\right\}_{i=1}^m}=0,\\
    & g\lrbracket{\left\{ x_i\right\}_{i=1}^m}\geq 0.
\end{split}
\end{align}
Following \cite{plavala2025polarization}, and for simplicity, we restrict our attention to the setting of two convex bodies $K_A$ and $K_B$ together with separately affine objectives and constraints. Our quantitative constrained hierarchy will further require the constraints to be local, as specified in \autoref{sec:comparison_DPS_PLG}. Concretely, we consider the following problem: given $K_A$, $K_B$, and separately affine maps $f\,:\, K_A\times K_B\rightarrow V$ and $g\,:\, K_A\times K_B\rightarrow V$, where $V$ is a finite-dimensional vector space equipped with a specified closed positive cone for the inequality, find $x_A\in K_A$ and $x_B\in K_B$ such that $f(x_A, x_B)=0$ and $g(x_A, x_B)\geq 0$.

As one application, and to illustrate how such problems arise, we study the value of two-player non-local games over general probabilistic theories; there the relevant convex body is the \emph{state space} of a system (so ``state space'' should throughout be read simply as ``convex body'', and ``state'' as ``point of the body''). The extension of this framework to the multivariate setting \autoref{eq:multivariate_problem}, involving more than two sets or higher-order polynomial constraints, requires certain adaptations of the techniques developed here, in the spirit of \cite{jee2020quasi}. As problems of this type are generally intractable to solve outright (cf., e.g.\ \cite{ioannou2007computationalcomplexityquantumseparability} and references therein), our attention turns to the study of suitable approximation methods.

\subsection{Contributions}

In this work we develop and apply information-theoretic and convex-analytic techniques to polynomial optimization problems.

Our main results are an answer to the optimization problem \autoref{eq:multivariate_problem} (for separately affine objectives and local affine constraints in the bivariate case studied here): a convergent hierarchy of relaxations for polynomial optimization over convex bodies, with finite-level rate guarantees (informal version of \autoref{thm:convergence_hierarchy} below), together with a certified rounding scheme that turns each level into a feasible point with controlled error (\autoref{thm:inner_bounds}). Both rest on our main technical contribution, a finite de Finetti representation theorem for convex bodies (\autoref{thm:main_theorem}), which supplies the quantitative engine behind the rate.

Although our tools and terminology come from quantum information and general probabilistic theories, this paper is about convex relaxations of nonconvex optimization problems, of the same kind as the Lasserre and Sherali--Adams hierarchies. The problem is nonconvex in its original variables but equivalent to a linear optimization over a convex set that in general has no tractable description, and our hierarchy relaxes it to a convergent sequence of convex conic programs. The two registers translate exactly, and the dictionary below (\autoref{tab:dictionary}) is the single reference for this correspondence; we prove every statement for arbitrary proper cones, not only the positive-semidefinite cone of quantum theory. The right-hand column gives the corresponding convex-geometric terminology.
\renewcommand{\arraystretch}{1.25}
\begin{table}[ht]
\caption{Correspondence between the information-theoretic vocabulary in which our tools are stated and the convex-geometric objects they denote. All notions are defined for arbitrary proper cones; see \autoref{sec:notation_and_preliminaries} for the formal definitions.}
\label{tab:dictionary}

\centering
\begin{tabularx}{\linewidth}{@{}>{\raggedright\arraybackslash}p{0.35\linewidth}>{\raggedright\arraybackslash}X@{}}
\toprule
\textbf{Information-theoretic term} & \textbf{Convex-geometric meaning} \\
\midrule
state space & convex body $K$ (its points are ``states'') \\
effect & functional in the dual cone taking values in $[0,1]$ \\
channel & affine map between bodies \\
measurement & affine map to a simplex \\
separable / entangled point & inside / outside the minimal tensor product $K_A\tmin K_B$ \\
relative entropy & integral divergence over effects (Kullback--Leibler on simplices) \\
mutual information & divergence from the closest product point with the same $B$-marginal \\
monogamy & that divergence is at most a constant depending only on $K_A$, uniformly in the number of copies of $K_B$ \\
finite de Finetti & symmetric extendible points approach $K_A\tmin K_B$ \\
\bottomrule
\end{tabularx}

\end{table}
\renewcommand{\arraystretch}{1.0}

Two tensor products of state spaces appear throughout (formal definitions in \autoref{sec:notation_and_preliminaries}). Each is itself a state space (a convex body), the base of a corresponding tensor product of the cones $C_A=\operatorname{cone}(K_A)$ and $C_B=\operatorname{cone}(K_B)$. The \emph{minimal tensor product} $K_A\tmin K_B$ (written with $\tmin$) is the convex hull of the product points $x_A\otimes x_B$; its cone $C_A\tmin C_B$ is generated by these products. The \emph{maximal tensor product} $K_A\tmax K_B$ is the base of the cone $C_A\tmax C_B = (C_A^* \tmin C_B^*)^*$, where $C^*$ is the dual cone of $C$. That is, $C_A \tmax C_B$ is the largest cone on which all product functionals $f_A\otimes f_B$ (with $f_A,f_B$ positive) stay nonnegative. The two state spaces coincide if and only if one of the bodies is a simplex \cite{Aubrun_2021}, and the gap between them is exactly the room for ``entangled'' points; the elements in $C_A \tmax C_B \setminus C_A \tmin C_B$ are called entangled, while the elements of $C_A \tmin C_B$ are called separable.

We can now state the optimization deliverable. Write $p^*$ for the true optimum of \autoref{eq:multivariate_problem} for the separately affine objective and local affine constraints considered here, assuming a nonempty feasible set and closed inequality cones. The level-$n$ relaxation \emph{relaxes} the product point $x_A\otimes x_B$ to a symmetric $n$-extendible point -- a point of $K_A\tmax K_B$ that extends to a permutation-invariant point of $K_A\tmax K_B^{\tmax n}$ (the conic analogue of a Lasserre/moment ``level-$n$ lifted variable'', i.e.\ a pseudo-moment) -- and imposes the (local) constraints on that extension; this is the conic analogue of a Lasserre/moment relaxation, where higher $n$ tightens the relaxation. Let $p^{(n)}$ denote the value of this level-$n$ relaxation, let $\lVert P\rVert$ denote the norm of the affine objective $P$, and let $c(A,B)$ be the de Finetti constant of \autoref{thm:main_theorem}, depending only on $K_A$ and $K_B$.
\begin{theorem*}[Convergent hierarchy with finite-level rate, informal]
    For every $n\in\mathbb{N}$, the level-$n$ relaxation value $p^{(n)}$ is a lower bound on $p^*$, and
    \begin{align}
        0\le p^*-p^{(n)}\le \frac{2\,\lVert P\rVert\,c(A,B)}{\sqrt n}.
    \end{align}
\end{theorem*}
The bridge from the de Finetti theorem to this rate is the following chain. A feasible point of the level-$n$ relaxation is a symmetric $n$-extendible point, hence by the de Finetti theorem below it is $O(1/\sqrt n)$-close to a convex combination of componentwise-feasible product points of $K_A\tmin K_B$. Because the constraints are \emph{local} -- each restricts $x_A$ alone or $x_B$ alone -- the conditioning construction preserves each constraint on every product component. The nearby separable point is thus feasible for the equivalent convex-hull formulation of the original problem. Evaluating the objective at this rounded point and using $\lvert P(x)-P(y)\rvert\le\lVert P\rVert\,\lVert x-y\rVert$ shows that the level-$n$ relaxation value is within $O(\lVert P\rVert/\sqrt n)$ of $p^*$, which is the displayed rate. The same rounding is constructive, and \autoref{thm:inner_bounds} turns it into a certified feasible point $w^{(n)}$ with $0\le P(w^{(n)})-p^*\le 2\,\lVert P\rVert\,c(A,B)/\sqrt n$, so that the relaxation and the rounding sandwich $p^*$ from below and above at the same $1/\sqrt n$ rate.
\begin{theorem*}[Finite de Finetti for convex bodies, informal]
    
    Let $x_{AB_1^n}$ be a point of the maximal tensor power $K_A\tmax K_B^{\tmax n}$ (itself a state space) that is invariant under permutations of the $n$ copies of $K_B$. Then there exists a separable state $\tilde{x}_{AB}$
 (a point of the minimal tensor product $K_A\tmin K_B$; cf.\ \autoref{tab:dictionary}) such that
    \begin{align}
        \left\lVert \tilde{x}_{AB} - x_{AB}\right\rVert\leq \frac{2c(A,B)}{\sqrt{n}},
    \end{align}
    where $\lVert \cdot\rVert$ is the base norm used in \autoref{thm:main_theorem} and $c(A,B)$ is a constant depending solely on the underlying state spaces $K_A$ and $K_B$.
\end{theorem*}
Although de Finetti theorems are well developed in quantum theory (see, e.g., \cite{Fuchs2004,christandl2007one,Brando2017}), to the best of the authors’ knowledge, comparatively little is known for de Finetti theorems formulated at the level of general convex bodies or cones; see \cite{Aubrun_2021,Aubrun_2024,plavala2025polarization}. Absent further assumptions, the available de Finetti theorems in this general setting are qualitative and pertain to the (infinitely) exchangeable regime \cite{Aubrun_2021,Aubrun_2024}, rather than giving quantitative, finite-$n$ approximation bounds. \cite[Thm.\ 2]{Aubrun_2024} provides a rigidity theorem: $K_B$ is affinely equivalent to a Cartesian product of at most $k$ simplices if and only if, for every $K_A$, every $k$-extendible point of $K_A\tmax K_B$ is separable. For general state spaces, however, such a finite-level collapse should not be expected, underscoring the need for quantitative approximation guarantees at finite-$n$. Quantitative, finite-level de Finetti and DPS-type rates of this kind are by now known for the positive-semidefinite cone, i.e.\ the quantum setting \cite{christandl2007one}; our proof adapts the information-theoretic argument to maximal tensor products of arbitrary proper cones. Finite results for classical conditional-probability state spaces also appear in \cite{Christandl2009}, and for general proper cones such finite-level guarantees are, to our knowledge, new.

Stated purely in convex-geometric terms, the theorem shows that $n$-extendible points (points of the maximal tensor power $K_A\tmax K_B^{\tmax n}$ symmetric in the copies of $K_B$) converge, at rate $1/\sqrt{n}$, to the minimal tensor product $K_A\tmin K_B$ in distance. Here \emph{entanglement} is precisely the max/min gap, the points of the state space $K_A\tmax K_B$ lying outside $K_A\tmin K_B$. The theorem also admits an information-theoretic reading: shared correlation is limited in how widely it can be distributed, a quantitative \emph{monogamy} statement, which we make precise below as a uniform, $n$-independent bound on the mutual information. The result is a consequence of several key ideas, which we detail below:
\begin{enumerate}[label=(\Alph*)]

    \item We use the cone-level relative entropy of Frenkel \cite{Frenkel_2023} and a related result of Jen\v{c}ová \cite{Jen_ov__2024}: the integral representation introduced there already defines a relative entropy on \emph{any} proper cone, built from the support functionals of the convex body. On the simplex it reduces to the classical Kullback--Leibler divergence and on the positive-semidefinite cone to Umegaki's relative entropy, both Bregman divergences of the negative entropy; the classical and quantum definitions live only on the orthant and the PSD cone, respectively, whereas this one is intrinsic to the cone and in general not of Bregman form.\footnote{For instance, on the square state space, $x\mapsto D(x\Vert y)-D(x\Vert z)$ is not affine, whereas it is affine for every Bregman divergence.} Our contribution is the \emph{mutual information} built from this divergence: concretely, it is the infimum over states $y_A\in K_A$ of $D(x_{AB}\,\Vert\,y_A\otimes x_B)$, measuring correlations relative to product points with fixed $B$-marginal. We define it in the general conic setting and establish the properties we need -- positivity, a data-processing inequality, and a local data-processing property -- which to our knowledge have not been recorded at this level of generality.
    \item Within this framework, we prove that the mutual information admits a uniform upper bound under extensions of one subsystem (see \autoref{prop:uniform_A_upper_bound_mutual_information}):
\begin{proposition*}[Conic mutual information bound, informal]
    For any state $x_{AB}$, we have
    \begin{align}
        I(A:B)_{x_{AB}} \leq \ln \lambda_A,
    \end{align}
    where $\lambda_A\geq1$ depends only on $K_A$ and satisfies $a\leq\lambda_A\tau_A$ for all $a\in K_A$, for a fixed $\tau_A\in\operatorname{relint}K_A$.
\end{proposition*}
\item As a consequence of the above, we obtain
\begin{align}
    I(A:B_1^n)_{x_{AB_1^n}} \leq c(A),
\end{align}
with a constant $c(A)$ that does not depend on the number of extensions $n \in \mathbb{N}$. 
This uniform, $n$-independent bound on the mutual information between $K_A$ and the $n$ copies of $K_B$ is the analytic engine of the de Finetti rate above, and it is exactly the \emph{monogamy} statement invoked earlier. Interpreted information-theoretically, it says that the correlation a point of $K_A$ can share is bounded independently of how many copies of $K_B$ it is shared with -- a strong monogamy relation.
\end{enumerate}

The direct incorporation of inequality constraints into the polarization hierarchy is left open in \cite{plavala2025polarization}. Local affine inequalities can instead be absorbed into the local convex bodies and treated asymptotically by the conic DPS hierarchy \cite{Aubrun_2024}; here we retain explicit lifted constraints and establish a finite-level rate with feasible rounding. We address the direct inequality question for certain \emph{local constraints} (cf.\ \autoref{sec:comparison_DPS_PLG}) by constructing a convergent conic hierarchy incorporating both local equality and inequality constraints (cf. \autoref{thm:convergence_hierarchy}). Our approach gives an explicit finite-level error bound in this general conic setting. In addition, we provide a certified rounding procedure yielding inner feasible approximations with explicit accuracy bounds (cf. \autoref{thm:inner_bounds}).

A word on the scope of ``local''. A local constraint restricts $x_A$ alone or $x_B$ alone, and so merely carves a sub-convex-body out of $K_A$ or $K_B$ -- when nonempty, itself a compact convex body and, after homogenization in its affine span, the base of a proper cone. Local constraints therefore place \emph{no} restriction on the local geometry; the only thing they forbid is constraints that \emph{couple} $A$ and $B$. A prominent class of problems is already of this form: copositive and completely-positive programs \cite{Grtner2012}, whose hardness lives entirely in the cone rather than in any cross-party coupling, and which our local-constraint hierarchy therefore covers when the problem admits the compact normalized formulation assumed here.

The genuinely uncovered case is that of coupled (global) constraints enforced at each point with a finite rate; averaged coupled constraints admit an asymptotic hierarchy, but retaining our finite rate requires an additional feasibility-repair argument. Such global coupling is the regime of the polarization hierarchy \cite{plavala2025polarization}, and of problems such as nonnegative matrix rank, which we mention only to mark the boundary of what global-constraint methods, rather than the present local-constraint approach, can address. Our results therefore extend the available guarantees and draw a link between information-theoretic de Finetti arguments and the theory of conic optimization. For a summary comparison highlighting how our contributions differ from the polarization hierarchy of \cite{plavala2025polarization}, we refer the reader to \autoref{tab:comparison_PLG_this_work}.

The most directly comparable work in the present journal is that of Berta, Borderi, Fawzi, and Scholz \cite{Berta2021}, who develop SDP hierarchies for constrained bilinear (polynomial) optimization with quantitative finite-level convergence guarantees. The primary distinction is one of \emph{generality}: their constrained bilinear results concern the positive-semidefinite cone, while finite de Finetti results also exist for classical conditional-probability spaces \cite{Christandl2009}; we prove a single theorem valid for \emph{all} proper cones, and hence for the much larger landscape of polynomial optimization over arbitrary convex bodies. Within that wider setting, the contribution is moreover threefold: a de Finetti rate at \emph{finite} level rather than only in the limit, the systematic inclusion of \emph{inequality} constraints alongside equalities, and a rounding scheme producing \emph{certified feasible points} with explicit error bounds. We do not claim improved practical performance; the comparison is one of guarantees and scope.

We connect our hierarchy to the theory of conic lifts and representability questions: if the underlying convex bodies admit suitable conic lifts, then the bilinear optimization problem can be reformulated in lifted variables; implementing a convergent hierarchy there additionally requires compact normalized lifted state spaces and tractable tensor-cone descriptions (cf.\ \autoref{prop:lifted_conic_program} and \cite{gouveia2025faces}).

For two-player non-local games, we show that the assemblage relaxation of the GPT value is an instance of our polynomial optimization problems with local constraints (\autoref{prop:game_as_47}). For polytopic state spaces and polyhedral inequality cones, every level of the hierarchy is an exact linear program, since maximal tensor products of polyhedral cones are again polyhedral. We solve these linear programs for the CHSH game with local state space a segment (the classical bit), a square (the gbit) or a regular pentagon, the pentagon also under an additional local inequality constraint. In all four experiments, at each of the computed levels $n\le 5$, the exact optimum (computed independently) lies between the value of the rounded feasible point and the relaxation value; to solver precision the two meet at $n=2$ for the segment and $n=4$ for the square, and for the pentagon the gap narrows steadily (\autoref{tab:numerics_results} and \autoref{fig:numerics_convergence}).

\renewcommand{\arraystretch}{1.8}
\begin{table}
\caption{Comparison of the polarization hierarchy of \cite{plavala2025polarization} with the identity polarization map with this work, emphasizing de Finetti-based convergence guarantees under equality and inequality constraints, locality of constraints, finite-level rates, and certified construction of feasible (rounded) points. We postpone to future research the development of an extension of our approach to more general constraint families, which may remove the current disparity in scope relative to the polarization hierarchy.}
\label{tab:comparison_PLG_this_work}

\centering
\begin{tabularx}{\linewidth}{@{}>{\raggedright\arraybackslash}p{0.21\linewidth}>{\raggedright\arraybackslash}X>{\raggedright\arraybackslash}X@{}}
\toprule
\textbf{Aspect} &
\textbf{Polarization hierarchy \cite{plavala2025polarization}} & 
\textbf{This work} \\
\midrule
Convergence guarantees &
For global eq. constraints &
For local eq. and ineq. constraints \\
\hdashline
Convergence rate &
Asymptotic convergence; no explicit finite-level error rate &
Quantified finite-level guarantees \\
\hdashline
Rounding / feasible points &
No finite-error rounding guarantee stated &
Certified feasible points with accuracy bounds \\
\bottomrule
\end{tabularx}

\end{table}
\renewcommand{\arraystretch}{1.0}

\subsection{Previous work}\label{sec:previous_work}

A substantial line of work on approximation techniques for multivariate polynomial optimization over convex bodies traces back to Lasserre’s moment hierarchy \cite{Lasserre2001,deKlerk2011,Nie2013}, which is itself based on a multivariate extension of the sum-of-squares decomposition for nonnegative polynomials \cite{nesterov2000squared}. Its non-commutative counterpart, the NPA hierarchy \cite{Pironio2010}, plays a central role in many approximation algorithms in quantum information theory (see e.g. \cite{Tavakoli_2024} and references therein) and there have been a number of extensions (see e.g.\ \cite{Klep_2023,Ligthart2023,Ligthart2023_inflation}). Alternatively, the Sherali-Adams hierarchy \cite{Sherali1990,Sherali1992} provides an approximation scheme that is conceptually quite different from the approaches mentioned above. In certain regimes, connections to the polarization techniques in \cite{Ligthart2023,Ligthart2023_inflation,plavala2025polarization} can be established; we refer the reader to \cite{plavala2025polarization} for a more detailed discussion.

In this work, we are particularly interested in approximation methods whose convergence is certified via a de Finetti-type argument.  De Finetti-based methods fundamentally rely on the interaction between exchange symmetries and appropriate monogamy relations. This interaction is captured by de Finetti representation theorems, which assert, in a precise sense, that permutation symmetry gives rise to (approximate) independence. While these methods offer analytical convergence guarantees, a recurring technical challenge -- addressed in works such as \cite{Berta2021,jee2020quasi,plavala2025polarization,kossmann2025symmetric,zeiss2025approximatingfixedsizequantum} -- is to understand how this approach remains compatible when the objects of interest are required to satisfy additional constraints, as is the case for the polynomial optimization problems considered here. The primary methods relevant to our setting are the DPS hierarchy \cite{Doherty_2004} and its conic generalization \cite{Aubrun_2024}, along with the polarization hierarchy for convex bodies developed in \cite{plavala2025polarization}. A comprehensive comparison is deferred to \autoref{sec:comparison_DPS_PLG}; here we merely indicate the principal ways in which our results differ.

The convergence of the DPS hierarchy rests on an underlying de Finetti representation theorem. For quantum systems, this is supplied by the theorem of \cite{christandl2007one}, whereas the infinite-exchangeability results applicable to general convex bodies are developed in \cite{Barrett2009,Aubrun_2024}. These unconstrained de Finetti statements approximate a state by a convex mixture. Imposing constraints only on its averaged marginal does not force every product component to satisfy them; the required componentwise feasibility needs additional lifted constraints, as in \cite{Berta2021} and in the hierarchy below. In contrast, the polarization hierarchy \cite{plavala2025polarization} remedies this limitation and provides a convergent approximation scheme that enforces the desired constraints individually. In the comparison below we use the identity polarization map, for which the required support argument is valid. However, for general state spaces, convergence is guaranteed only asymptotically, and no explicit finite-level error rate is established there.

The information-theoretic finite de Finetti theorem established here provides a framework valid for arbitrary proper cones, with analytical finite-level convergence guarantees for separately affine optimization and componentwise compatibility with local affine constraints. In addition, we address the direct inequality question posed in \cite{plavala2025polarization} for local affine constraints by incorporating them into the de Finetti-based convergence framework. Finally, following \cite{kossmann2025symmetric,zeiss2025approximatingfixedsizequantum}, we introduce a rounding scheme that produces certifiably good feasible points, with the resulting approximation error explicitly bounded by our de Finetti theorem. In particular, our method enables us to obtain matching upper and lower bounds, thereby sandwiching the original problem value.

The remainder of this manuscript is organized as follows. In \autoref{sec:polynomial_optimization_over_convex_bodies}, we formulate in detail the polynomial optimization problems over convex bodies that we study and introduce a reformulation with a more advantageous mathematical structure, which serves as the basis for our subsequent analysis. We also provide a technically refined comparison with the DPS and polarization hierarchies and discuss the main challenges encountered in deriving our results.
In \autoref{sec:information_theoretic_de_finetti_representation_theorem}, we develop the information-theoretic foundations underlying our main theorem. In particular, we define an appropriate notion of relative entropy on convex bodies and establish associated monogamy relations, which enable us to prove an information-theoretic finite de Finetti representation theorem for general states.
Subsequently, in \autoref{sec:convergent_hierarchy_for_polynomial_optimization}, we apply this de Finetti theorem to establish convergence guarantees for certain hierarchical approximation schemes in polynomial optimization. We further introduce a rounding procedure that certifies high-quality inner feasible approximations, and we show how these schemes can be formulated as convex optimization problems -- thereby identifying conditions under which numerical implementations are feasible.

In \autoref{sec:games}, we study two-player non-local games over general probabilistic theories. We relax the GPT value of such a game by merging Alice's state and measurements into an assemblage on Bob's system, show that this relaxation is an instance of our polynomial optimization problems with local constraints, so that the hierarchy and the rounding scheme apply to it directly, and discuss when the relaxation is exact.
In \autoref{sec:numerics}, we illustrate both schemes on the CHSH game over polytopic state spaces, where every level of the hierarchy is a linear program, and compare the resulting upper and lower bounds with the exact optimum.
Finally, we conclude in \autoref{sec:discussion_and_open_problems} with a discussion of our results and several directions for future research. Additional technical details and background material for the reader’s convenience are provided in the appendices.

\Needspace{8\baselineskip}
\section{Polynomial optimization over convex bodies}\label{sec:polynomial_optimization_over_convex_bodies}

\subsection{Notation and preliminaries}\label{sec:notation_and_preliminaries}

In this work, let $V=(V, \geq )$ be a finite dimensional ordered vector space over $\mathbb{R}$ with Euclidean topology. A subset $C$ of $V$ is called \emph{convex cone} if it satisfies $sx \in C$ for all $s \in \mathbb{R}_{\geq 0}$, $x \in C$ and $x,y \in C$ implies that $x+y \in C$. Moreover, a cone is called \emph{generating} if it spans $V$ as a vector space, $C$ is called \emph{pointed} if $C \cap (-C) = \{0\}$. A cone $C$ is called \emph{proper}\footnote{Definitions vary; in some sources the generating condition is replaced by requiring the cone to be solid.} provided it is closed, convex, generating and pointed. For any $X\subseteq V$, we denote by $\operatorname{conv}(X)$ the convex hull of $X$, by $\text{cone}(X)$ the smallest cone containing $X$ and by $\operatorname{aff}(X)$ the affine hull of $X$. We denote the \emph{relative interior} of $X$ by $\operatorname{relint}(X)$ and the \emph{interior} by $\operatorname{int}(X)$. If we denote with $V^\star$ the dual space of the vector space (recall that $V$ is finite-dimensional, so its dual is simply the space of all linear maps $V\to \mathbb{R}$), then the \emph{dual cone} $C^\star$ of a cone $C\subseteq V$ is defined as 
\begin{align}
    C^\star \coloneqq \{f \in V^\star \ : \ f(x) \geq 0, \ \forall \ x\in C\}.
\end{align}
For additional background on cones and ordered vector spaces, we refer the reader to \cite[App.\ A]{Pl_vala_2023}.

The central object of interest in this work is a nonempty compact convex subset $K$ of a finite-dimensional real vector space, called a \emph{state space} or \emph{convex body} \cite[Def.\ 2.1]{Pl_vala_2023}\cite[Sec.\ 3.1]{Hug2020}. We use the Euclidean topology throughout and do not require $K$ to have nonempty ambient interior. By \cite[Def.\ 4.1]{Pl_vala_2023}, classical state space is a simplex $\Delta_n\subseteq\mathbb{R}^{n}$ of probability distributions on $n$ outcomes, $n\in\mathbb{N}$. The simplex $\Delta_n$ is a base of the simplicial cone $\mathbb{R}_{\geq0}^n$. We next observe that each state space $K$ is associated with a proper cone and ordered vector space in a natural way. Let $A(K,V)$ denote the space of affine functions from $K$ to $V$, i.e.\ for any $f\in A(K,V)$, we have
\begin{align}
    f(\lambda x + (1-\lambda)y) = \lambda f(x) + (1-\lambda)f(y), \quad x,y\in K, \, \lambda \in [0,1], 
\end{align}
and write $A(K)$ when $V \cong \mathbb{R}$ as a vector space. Furthermore, we denote the proper cone \cite[Lem.\ 3.11]{Pl_vala_2023} of positive affine functions inside $A(K)$ by $A(K)^{+}$, i.e.\
\begin{align}
    A(K)^{+}=\left\{ f\in A(K)\, : \, f(x)\geq 0, \forall x\in K\right\}.
\end{align} 
By \cite[Lem.\ 3.12]{Pl_vala_2023}, $A(K)$ with the natural partial ordering "$\geq$" is an ordered vector space. 
We define the \emph{effect algebra} $E(K)\subseteq A(K)^{+}$ as the space of affine functions $f:K \to [0,1]$. In other words, by \cite[Prop.\ 3.13]{Pl_vala_2023},
\begin{align}
    E(K)= \left\{ f \in A(K)\, : \, 0\leq f \leq 1_K\right\},
\end{align}
where $1_K\in E(K)$ denotes the unit effect, i.e.\ the constant function yielding $1_K(x)=1$ for any $x\in K$. Note that $A(K)=\operatorname{span}_{\mathbb{R}}(E(K))$ and $A(K)^+=\operatorname{cone}(E(K))$ \cite[Prop.\ 3.10]{Pl_vala_2023}. Furthermore, $1_K$ is the (canonical) \emph{order unit} of $A(K)^+$ \cite[Def.\ 3.16, Lem.\ 3.17]{Pl_vala_2023}, i.e.\ 
\begin{align}
    \forall g\in A(K)\, \exists \mu\in\mathbb{R}_+,\,\,:\, g\leq \mu 1_K. 
\end{align}
In other words, $\lrbracket{A(K),\, A(K)^+,\, 1_K}$ is an order unit space \cite[Def.\ 3.27]{Pl_vala_2023}. Let $A(K)^\star$ denote the dual space to $A(K)$ and  $A(K)^{\star+}\coloneqq \lrbracket{A(K)^+}^\star$ the dual cone induced by $A(K)^{+}$. By \cite[Lem.\ 3.18]{Pl_vala_2023}, $A(K)^{\star+}$ is a proper cone, $\lrbracket{A(K)^{\star +}}^{\star}= A(K)^{+}$ \cite[Prop.\ B.11]{Pl_vala_2023}, and \cite[Prop.\ B.4]{Pl_vala_2023} establishes the canonical vector space isomorphism $A(K)^{\star\star}\simeq A(K)$. Consequently, $K$ admits a natural embedding into $A(K)^{\star +}$. In other words, each state $x\in K$ is a positive linear functional on affine functions. By \cite[Lem.\ 3.20]{Pl_vala_2023}, we have up to a natural isomorphism
\begin{align}
    K = \left\{\phi_x \in A(K)^{\star +}\, : \, \langle \phi_x, 1_K \rangle=1\right\}.
\end{align}
See also \autoref{sec:preliminaries_on_GPTs}. The state space arises as the intersection of a proper cone with an affine hyperplane. Importantly, by \cite[Lem.\ 3.34]{Pl_vala_2023}, $K$ is a \emph{base} of the proper cone $A(K)^{\star +}$. Because $A(K)^{\star +}$ generates $A(K)^\star$, \cite[Lem.\ 3.35]{Pl_vala_2023} ensures that for every $\phi\in A(K)^\star$ there exist $x,y \in K$ and $\lambda, \mu\in \mathbb{R}_+$ such that 
\begin{align}
    \phi= \lambda x - \mu y.
\end{align}
Thus, any affine map on the state space $K$ induces a linear map on $A(K)^\star$. Note that $\lrbracket{A(K)^{\star},\, A(K)^{\star +},\, e}$ with $e\in\operatorname{int}\lrbracket{A(K)^{\star +}}$ is an order unit space (see \autoref{prop:interior_points_of_proper_cones_are_order_units}). By \cite[Prop.\ 3.38]{Pl_vala_2023},
\begin{align}
    \lVert f\rVert_{1_K} \coloneqq  \inf \left\lbrace t \in \mathbb{R}_+\,:\, -t 1_K \leq f \leq t 1_K\right\rbrace,\quad f\in A(K),
\end{align}
is an \emph{order-unit norm} and its dual is the \emph{base norm}
\begin{align}
\lVert X\rVert_{\star,1_K}
=\inf\{1_K(a)+1_K(b):X=a-b,\ a,b\in A(K)^{\star+}\}.
\end{align}
For $1_K(X)=0$, the substitution $f=2e-1_K$ between the order-unit ball and the effect interval gives
\begin{align}
\lVert X\rVert_{\star,1_K}=2\max_{e\in E(K)}e(X).
\end{align}
In particular, this identity applies to differences of normalized states. More generally, in a finite-dimensional ordered vector space with a proper positive cone, every order unit defines an order-unit norm (cf.\ \cite[Sec.\ 3.6]{Pl_vala_2023}).

Let $K_A$, $K_B$ be two state spaces. A \emph{channel} \cite[Def.\ 6.1]{Pl_vala_2023} is an affine map $\Phi:K_A \to K_B$. The set of channels from $K_A$ to $K_B$ is denoted by $\mathcal{C}\lrbracket{K_A, K_B}$, and write $\mathcal{C}(K)$ whenever $K_A=K_B=K$. We denote by $\operatorname{id}_K \in\mathcal{C}(K)$ the identity channel, i.e.\ the channel given as $\operatorname{id}_K(x)=x$ for all $x\in K$. The data processing inequality (DPI) for channels yields
\begin{align}
   \frac{1}{2} \Vert \Phi(x) - \Phi(y) \Vert_{\star, 1_{K_B}} \leq \sup_{f \in E(K_A)} f(x-y),\qquad x,y\in K_A.
\end{align}
An $n$-\emph{outcome measurement} \cite[Def.\ 6.11]{Pl_vala_2023} is a channel $\mathcal{M}\in \mathcal{C}\lrbracket{K, \Delta_n}$. A measurement $\mathcal{M}\in \mathcal{C}\lrbracket{K, \Delta_n}$ is \emph{informationally complete} 
if for $x,y\in K$ with $\mathcal{M}(x) = \mathcal{M}(y)$ we have $x = y$.
Every state space has a finite informationally complete measurement. Indeed, let $d=\dim A(K)$ and choose a basis $1_K,h_1,\ldots,h_{d-1}$. For sufficiently small $\varepsilon>0$, the functions
\begin{align}
m_j=d^{-1}1_K+\varepsilon h_j\quad (j<d),\qquad
m_d=d^{-1}1_K-\varepsilon\sum_{j<d}h_j
\end{align}
are nonnegative on compact $K$. They sum to $1_K$ and span $A(K)$, so they define such a measurement with $d$ outcomes. When $d=1$, $K$ is a singleton and the one-outcome measurement suffices.

For an informationally complete measurement $\mathcal{M}$, there exists a constant $d(K)>0$ (also depending on the fixed choice of $\mathcal{M}$) such that
\begin{align}\label{eq:def_distortion} 
   \sup_{f \in E(K)} f(x-y) \leq d(K) \frac{1}{2} \Vert \mathcal{M}(x) - \mathcal{M}(y) \Vert_{\star, 1_{\Delta_n}}.
\end{align}
The existence of such a constant in the conic setting follows from a standard compactness argument \cite{Pl_vala_2023}. By \cite[Prop.\ 6.13]{Pl_vala_2023}, any $n$-outcome measurement $\mathcal{M}\in \mathcal{C}\lrbracket{K, \Delta_n}$ is uniquely defined by effects (i.e.\ affine positive functions)\footnote{For readers familiar with quantum theory, this notion corresponds to the resolution of a complete POVM in terms of POVM elements indexed in their outcomes.} $\left\{f_1,\ldots, f_n\right\}\subseteq E(K)$ such that $\sum_{i=1}^n f_i = 1_K$.

Throughout, “$\otimes$" denotes the algebraic tensor product of vector spaces, characterized by its universal property. In this work, we will frequently consider two natural tensor products for cones $C_1\subseteq V_1$, $C_2\subseteq V_2$ inside the algebraic tensor product $V_1\otimes V_2$, namely the minimal and maximal tensor products, corresponding to the two extremal constructions. The \emph{minimal tensor product} of $C_1$ and $C_2$ is defined as 
\begin{align}
    C_1\dot{\otimes}C_2 \coloneqq \operatorname{conv}\{x_1\otimes x_2 \ : \ x_1 \in C_1, \ x_2 \in C_2\}.
\end{align}
On the algebraic tensor product $V_1\otimes V_2$, any pair $f_1\in V_1^{\star}$, $f_2\in V_2^{\star}$ defines a linear functional
\begin{align}
\begin{split}
    (f_1\otimes f_2)\,:\,&V_1\otimes V_2 \rightarrow \mathbb{R}\\
    &x_1\otimes x_2 \mapsto f_1(x_1)f_2(x_2)
\end{split}
\end{align}
extended linearly. The \emph{maximal tensor product} of $C_1$ and $C_2$ is defined as 
\begin{align}
\begin{split}
    C_1 \hat{\otimes} C_2 &\coloneqq \lrbracket{C^\star_1\tmin C_2^\star}^\star \\
    &=\left\{z \in V_1 \otimes V_2 \ : \ (f_1\otimes f_2)(z)\geq 0 \ \text{for all} \ f_1 \in C_1^\star, \ f_2 \in C_2^\star \right\}.
\end{split}
\end{align}
In other words, it is the largest cone in $V_1\otimes V_2$ on which all functionals $f_1\otimes f_2$ (with $f_1, f_2$ positive) are themselves positive. Formally, if $C$ is any cone in $V_1\otimes V_2$ such that
\begin{align}
    (f_1\otimes f_2)(C)\subseteq [0, \infty),\quad \forall f_1\in C_1^{\star},\, f_2\in C_2^{\star} 
\end{align}
then necessarily $C\subseteq C_1 \tmax C_2$.
For state spaces, write $C_i=A(K_i)^{\star+}$ and $u_i=1_{K_i}$. The associated normalized tensor products are
\begin{align}
K_A\tmin K_B=\operatorname{conv}\{a\otimes b:a\in K_A,\ b\in K_B\},
\end{align}
\begin{align}
K_A\tmax K_B=\{x\in C_A\tmax C_B:(u_A\otimes u_B)(x)=1\}.
\end{align}
These are compact bases of the corresponding tensor cones. The positive affine functions on the maximal composite form the dual cone
$C_A^\star\tmin C_B^\star$, and its effects are the interval
$[0,u_A\otimes u_B]$ in that cone.

We will also use the canonical correspondence between tensors and linear maps:
\begin{align}\label{eqn:general_Choi}
z\in C_A\tmax C_B
\quad\longleftrightarrow\quad
\psi_z:C_A^\star\to C_B,\qquad
g(\psi_z(f))=(f\otimes g)(z).
\end{align}
Here $\psi_z$ extends linearly to the ambient spaces and is positive precisely when $z$ lies in the maximal cone. For any proper cone $C\subseteq V$, a basis $(e_i)_i$ of $V$ and its dual basis $(e^i)_i$, the basis-independent tensor
\begin{align}\label{eqn:Choi_identity}
\chi_V=\sum_i e^i\otimes e_i\in C^\star\tmax C
\end{align}
represents the identity map of $V$, since $\sum_i e^i(x)e_i=x$.
Furthermore, using associativity, we can iterate the tensor products to obtain the $n$-fold constructions $C^{\tmin n}$ and $C^{\tmax n}$, defined for every proper cone $C$. Note that $K_A\tmin \Delta_n = K_A\tmax \Delta_n$ by \cite[Prop.\ 5.20]{Pl_vala_2023}, because the simplicial cone generated by $\Delta_n$ is nuclear. Within the framework of \emph{bipartite states} \cite[Def.\ 5.1]{Pl_vala_2023}, elements from $K_A\tmin K_B$ are referred to as \emph{separable states}  while states from $K_A\tmax K_B$ which are not separable are called \emph{entangled} \cite[Def.\ 5.8]{Pl_vala_2023}. By Carathéodory’s theorem \cite[Thm.~3.4]{Pl_vala_2023}, any state $x_{AB}\in K_A\tmin K_B$ admits a finite \emph{product state} decomposition; that is, one can find finitely many $x_A^{\lambda}\otimes x_B^\lambda$ and coefficients $p_{\lambda}\geq 0$ with $\sum_{\lambda} p_\lambda =1$ satisfying
\begin{align}
    x_{AB}=\sum_{\lambda}p_{\lambda} x_A^{\lambda}\otimes x_B^\lambda.
\end{align}
Let $\tilde{\otimes}$ denote any tensor product, such that $K_A \tmin K_B \subseteq K_A \tilde{\otimes} K_B \subseteq K_A \tmax K_B$. A state from $K_A\tilde{\otimes}K_B^{\tilde{\otimes} n}$ is denoted by $x_{AB_1^n}$. We identify the affine map 
\begin{align}
    \operatorname{tr}_{B}(\cdot)\coloneqq \lrbracket{\operatorname{id}_{K_A}\otimes 1_{K_B}}(\cdot)\in\mathcal{C}\lrbracket{K_A\tilde{\otimes}K_B, K_A}
\end{align}
as the partial trace map or channel \cite[Ex.\ 6.4]{Pl_vala_2023}. While the notion of a \emph{completely positive channel} (cf.\ \cite[Def.\ 6.21, Prop.\ 6.27, Prop.\ 6.28]{Pl_vala_2023}) is dependent on the choice of tensor product, by \cite[Prop.\ 6.23]{Pl_vala_2023}, a measurement is completely positive with respect to any tensor product $\tilde{\otimes}$. Concretely, a measurement $\mathcal{M}_B\in\mathcal{C}\lrbracket{K_B, \Delta_n}$ is completely positive with respect to any $K_A\tilde{{\otimes}}K_B \rightarrow K_A\tilde{\otimes} \Delta_n$, i.e.\ $\forall x_{AB}\in K_A\tilde{\otimes} K_B$ we have $\lrbracket{\operatorname{id}_A\otimes \mathcal{M}_B}\lrbracket{x_{AB}}\in K_A\tilde{\otimes}\Delta_n$.

\subsection{Polynomial optimization over convex bodies}\label{subsec:polynomial_optimization_over_convex_bodies}

In this work, we consider two state spaces $K_A\subseteq V_A$ and $K_B\subseteq V_B$ and study the class of optimization problems of the form 
\begin{align}\label{eq:bivariate_problem}
\begin{split}
    p^{*}= \min_{(x_A, x_B)\in K_{A}\times K_{B}} \quad & p\lrbracket{x_A, x_B}\\
    \text{s.t.} \quad & f(x_A, x_B)=0, \\
    &g(x_A,x_B) \geq 0,
\end{split}
\end{align}
where $p\,:\, K_A\times K_B \rightarrow \mathbb{R}$ and $f,g\,:\, K_A\times K_B \rightarrow V$ are separately affine: each map is affine in each argument with the other fixed (see \autoref{sec:affine_maps}). The inequality is interpreted in a specified closed proper cone $D\subseteq V$. This class of optimization problems includes many well-studied instances, such as those considered in \cite{Berta2021,plavala2025polarization,Sherali1990,Sherali1992,Lasserre2001,ohst2024characterising,zeiss2025approximatingfixedsizequantum,kossmann2025symmetric}.

For any separately affine map $f:K_A \times K_B \to V$ we can construct a unique linear map on $A(K_A)^\star\otimes A(K_B)^\star$ whose restriction to $K_A\tmin K_B$ is denoted by $\hat{F}_{\hookrightarrow}$, as sketched in \autoref{fig:affine_maps_to_linear_maps} and rigorously worked out in \autoref{sec:affine_maps}. We thus obtain an equivalent formulation of the optimization problem in \autoref{eq:bivariate_problem}, wherein the affine maps $p$, $f$, and $g$ are replaced by their associated linear maps $\hat{P}_{\hookrightarrow}$, $\hat{F}_{\hookrightarrow}$, and $\hat{G}_{\hookrightarrow}$, as described in \autoref{fig:affine_maps_to_linear_maps}:
\begin{align}\label{eq:general_linear_constraint_problem}
\begin{split}
    p^{*}= \min_{(x_A,x_B)\in K_A\times K_B} \quad & \hat{P}_{\hookrightarrow}\lrbracket{x_A\otimes x_B} \\
    \text{s.t.} \quad & \hat{F}_{\hookrightarrow}(x_A\otimes x_B)= \ 0, \\
    &\hat{G}_{\hookrightarrow}(x_A \otimes x_B) \geq \  0.
\end{split}
\end{align}

\begin{figure}
    \centering
    \begin{tikzpicture}[>=stealth]

  \node (A) at (0,3) {$A(K_A)^\star \times A(K_B)^\star$};
  \node (V) at (4,3) {$V$};
  \node (Z) at (8,3) {$K_A \times K_B$};

  \node (B) at (0,0) {$A(K_A)^\star \otimes A(K_B)^\star$};
  \node (C) at (4,0) {$K_A' \dot{\otimes}K_B'$};

  \draw[->] (A) -- node[above] {$F$} (V);
  \draw[->] (A) -- node[left] {$\otimes$} (B);
  \draw[->] (C) -- node[above] {$\iota$} (B);
  \draw[->] (C) -- node[right] {$\hat{F}_{\hookrightarrow} \coloneqq  \hat{F}\circ \iota $} (V);
  \draw[->] (Z) -- node[above] {$f$} (V);

  \draw[->, dashed] (B) -- node[above] {$\exists ! \ \hat{F} \quad$} (V);

\end{tikzpicture}
\caption{The figure schematically shows how the correspondence between affine maps $f$ and linear maps $\hat{F}_{\hookrightarrow}$ works. We start with a separately affine map $f$ and build out of it a bilinear map $F$ via an image identification construction. From this we use the universal property of the tensor product to get a unique linear map $\hat{F}$. Using the embedding $\iota$ of $K_A' \dot{\otimes} K_B'$ into $A(K_A)^\star\otimes A(K_B)^\star$ yields a map $\hat{F}_{\hookrightarrow} = \hat{F} \circ \iota$.}\label{fig:affine_maps_to_linear_maps}
\end{figure}
 
\FloatBarrier
\subsection{Challenges and comparison to DPS- and Polarization-hierarchy}\label{sec:comparison_DPS_PLG}
To complement the overview in \autoref{sec:previous_work}, the next subsection develops a more detailed mathematical comparison between the DPS and polarization hierarchies. We focus on the polynomial optimization problems depicted in \autoref{fig:four_problems} and discuss how the approximation schemes of \cite{Aubrun_2024,plavala2025polarization} converge to the respective problems represented there. In the following, let $V$ be a finite-dimensional real vector space, $C\subseteq V$ a proper cone with base $K$ and $n\in\mathbb{N}$.

\begin{figure}
    \centering
\resizebox{\linewidth}{!}{\begin{tikzpicture}[
  box/.style={draw, rectangle, rounded corners, inner sep=4pt, align=left}
]

\node[box] (p1) at (0,0) {$
\begin{aligned}
p_1 \coloneqq \  \displaystyle\min_{x_{A}\otimes x_B\in K_{A}{\color{PineGreen}\tmin} K_{B}}
          \ &\hat{P}_{\hookrightarrow}\lrbracket{x_{A}\otimes x_B} \\
 \text{s.t. } \ &\hat{F}_{\hookrightarrow}\lrbracket{x_{A}\otimes x_B}=0,\\
& \hat{G}_{\hookrightarrow}\lrbracket{x_{A}\otimes x_B}\geq 0
\end{aligned}
$};

\node[box] (p2) at (0,-3) {$
\begin{aligned}
p_2 \coloneqq \  \displaystyle\min_{x_{A}\otimes x_{B}\in K_{A}{\color{PineGreen}\tmax} K_{B}}
          \ &\hat{P}_{\hookrightarrow}\lrbracket{x_{A}\otimes x_B}\\
 \text{s.t. } \ & \hat{F}_{\hookrightarrow}\lrbracket{x_{A}\otimes x_B}=0,\\
& \hat{G}_{\hookrightarrow}\lrbracket{x_{A}\otimes x_B}\geq 0
\end{aligned}
$};

\node[box] (p3) at (7,0) {$
\begin{aligned}
p_3 \coloneqq \  \displaystyle\min_{x_{AB}\in K_{A}{\color{PineGreen}\tmin} K_{B}}
       \ &   \hat{P}_{\hookrightarrow}\lrbracket{x_{AB}}\\
 \text{s.t. } \ & \hat{F}_{\hookrightarrow}\lrbracket{x_{AB}}=0,\\
& \hat{G}_{\hookrightarrow}\lrbracket{x_{AB}}\geq 0
\end{aligned}
$};

\node[box] (p4) at (7,-3) {$
\begin{aligned}
p_4 \coloneqq \  \displaystyle\min_{x_{AB}\in K_{A}{\color{PineGreen}\tmax} K_{B}}
        \ &\hat{P}_{\hookrightarrow}\lrbracket{x_{AB}} \\
\text{s.t. } \ & \hat{F}_{\hookrightarrow}\lrbracket{x_{AB}}=0,\\
& \hat{G}_{\hookrightarrow}\lrbracket{x_{AB}}\geq 0
\end{aligned}
$};
\end{tikzpicture}}
\caption{We compare the four optimization problems $p_1, p_2, p_3, p_4$. Because the optimization variable in $p_1$ is restricted to product states, we obtain $p_1=p_2$ immediately. Moreover, every state feasible for $p_1$ is also feasible for $p_3$, whereas the converse does not hold; hence $p_3 \leq p_1$. Concretely, any state feasible for $p_3$ can be expressed as $x_{AB}=\sum_{\lambda}p_{\lambda} x_A^{\lambda}\otimes x_B^{\lambda}$. Hence, the condition $\hat{F}_{\hookrightarrow}\lrbracket{x_{AB}}=0$ implies $\sum_{\lambda}p_{\lambda} \hat{F}_{\hookrightarrow}\lrbracket{x_A^{\lambda}\otimes x_B^{\lambda}}=0$. However, in contrast to the case $p_1$, this does not generally imply that  $\hat{F}_{\hookrightarrow}\lrbracket{x_A^{\lambda}\otimes x_B^{\lambda}}=0$ for all $\lambda$ as in $p_1$. As it turns out, this subtle but essential difference poses a significant challenge for approximation schemes. Since $K_A\tmin K_B\subseteq K_A\tmax K_B$, we also have $p_4\leq p_3$.}
\label{fig:four_problems}
\end{figure}

\subsubsection{The DPS hierarchy for cones}
Our discussion follows the cone-theoretic formulation of the DPS hierarchy in \cite{Aubrun_2024}. For its original quantum-mechanical version, we refer the reader to \cite{Doherty_2004}. The \emph{symmetric group} $S_n$ admits a natural permutation representation on $V^{\otimes n}$ via linear maps
\begin{align}
    S_{\sigma}\lrbracket{x_1 \otimes \cdots \otimes x_n} = x_{\sigma^{-1}(1)} \otimes \cdots \otimes x_{\sigma^{-1}(n)}, \quad \forall \sigma \in S_n.
\end{align}
The \emph{symmetric subspace} $\operatorname{Sym}_n(V)\subseteq V^{\otimes n}$ is the space invariant under this $S_n$-action, i.e.\ 
\begin{align}
    \operatorname{Sym}_n(V) = \left\{x\in V^{\otimes n}\, : \, S_{\sigma}(x)=x,\ \forall \sigma\in S_n\right\}.
\end{align}
Accordingly, the \emph{symmetric projection} is the operator $P_{\operatorname{Sym}_n(V)}\,:\, V^{\otimes n} \rightarrow V^{\otimes n}$ defined as 
\begin{align}
    P_{\operatorname{Sym}_n(V)} = \frac{1}{n!}\sum_{\sigma \in S_n} S_{\sigma}.
\end{align}
See \autoref{sec:symmetry_action_on_convex_bodies} for further discussions on the $S_n$-action on cones and state spaces. Let $V_A$ and $V_B$ be finite-dimensional real vector spaces, and let $C_A\subseteq V_A$ and $C_B\subseteq V_B$ be proper cones. Denote by $K_A$ and $K_B$ bases of the cones $C_A$ and $C_B$, respectively. Let $\phi\in\operatorname{int}\lrbracket{C_B^\star}$. The \emph{cone of $n$-extendibile elements} is given by
\begin{align}
    \operatorname{Ext}_n\lrbracket{C_A, C_B, \phi}=\lrbracket{\id_{V_A}\otimes \gamma_n^{\phi}}\lrbracket{C_A\tmax C_B^{\tmax n}},
\end{align}
where the $n$-th \emph{reduction map} $\gamma_n^{\phi}\,:\, V_B^{\otimes n}\rightarrow V_B$ \cite[Thm.\ 1]{Aubrun_2024} is given by
\begin{align}
\begin{split}
     \gamma_n^{\phi} &= \frac{1}{n}\sum_{j=1}^n \phi^{\otimes (j-1)}\otimes\id_{V_B} \otimes \phi^{\otimes (n-j)} \\
     &= \lrbracket{\id_{V_B}\otimes \phi^{\otimes (n-1)}}\circ P_{\operatorname{Sym}_n\lrbracket{V_B}}.
\end{split}
\end{align}
In other words, $\gamma_n^{\phi}$ symmetrically traces out all but one copy of $V_B$. For any $x_{AB}\in V_A\otimes V_B$, we have $x_{AB}\in \operatorname{Ext}_n\lrbracket{C_A, C_B, \phi}$ if and only if there exists an $n$-\emph{extension} $y_{AB_1^n} \in \lrbracket{\id_{V_A}\otimes P_{\operatorname{Sym}_n(V_B)}}\lrbracket{C_A\tmax C_B^{\tmax n}}$ of $x_{AB}$, i.e.\
\begin{align}
    x_{AB}= \lrbracket{\id_{V_A}\otimes \id_{V_B}\otimes \phi^{\otimes (n-1)}}\lrbracket{y_{AB_1^n}}. 
\end{align}
Since any $n$-extension can be reduced via $\phi$ to an $(n-1)$-extension, we obtain a decreasing sequence of cones and further define
\begin{align}
    \operatorname{Ext}_{\infty}\lrbracket{C_A, C_B, \phi} \coloneqq  \bigcap_{n\geq 1} \operatorname{Ext}_n\lrbracket{C_A, C_B, \phi}.
\end{align}
The notions introduced above readily translate to states. Concretely, consider the subset of $\operatorname{Ext}_n\lrbracket{C_A, C_B, \phi}$ which are states, i.e.\
\begin{align}
    \operatorname{Ext}_n\lrbracket{K_A, K_B} \coloneqq  \operatorname{Ext}_n\lrbracket{C_A, C_B, 1_{K_B}} \cap \lrbracket{1_{K_A}\otimes 1_{K_B}}^{-1}(1).
\end{align} 
Then, by \cite[Lem.\ 6]{plavala2025polarization}, this is a compact subset of $K_A\tmax K_B$. Then, we refer to 
\begin{align}
\begin{split}
 \operatorname{DPS}^{(n)}\coloneqq \, \displaystyle \min_{x_{AB}\in \operatorname{Ext}_n(K_A, K_B)}
          \ &\hat{P}_{\hookrightarrow}\lrbracket{x_{AB}}\\
 \text{s.t. } \ &\hat{F}_{\hookrightarrow}\lrbracket{x_{AB}}=0,\\
& \hat{G}_{\hookrightarrow}\lrbracket{x_{AB}}\geq 0
\end{split}
\end{align}
as the DPS hierarchy. Importantly, we have $ \operatorname{Ext}_1\lrbracket{K_A, K_B} = K_A\tmax K_B$ and, by \cite[Thm.\ 1]{Aubrun_2024}, 
\begin{align}
    \operatorname{Ext}_{\infty}\lrbracket{K_A, K_B}=K_A\tmin K_B.
\end{align}
Thus, $p_4$ corresponds to the first level of the DPS hierarchy, while $p_3$ corresponds to the problem to which the hierarchy converges.

\subsubsection{The Polarization hierarchy}
Although fundamentally grounded in symmetric extensions, the polarization hierarchy introduced in \cite{plavala2025polarization} differs from the DPS hierarchy outlined above in several important aspects. For $n, k\in\mathbb{N}$ satisfying $k\leq n$, the hierarchy relies on a symmetric double extension, leading to outer approximations in terms of $K_A^{\tmax n}\tmax K_B^{\tmax n}\cong \lrbracket{K_A\tmax K_B}^{\tmax n}$. Accordingly, consider the set of symmetric $k$-th extensions of $n$-extendable states of $K$ \cite[Eq.\ 32]{plavala2025polarization}, i.e.
\begin{align}
    \operatorname{Ext}_n(K,k) \coloneqq  \gamma_K^{n, k}\lrbracket{K^{\tmax n}},
\end{align}
with the generalization of the reduction map $\gamma_K^{n, k}\,:\, V^{\otimes n}\rightarrow V^{\otimes k}$ such that $\gamma_K^{n, k} \coloneqq  \id_K^{\otimes k}\otimes 1_{K}^{\otimes (n-k)}\circ P_{\operatorname{Sym}_n\lrbracket{V}}$. This set is a compact subset of $K^{\tmax k}$(\cite[Lem.\ 6]{plavala2025polarization}). Furthermore, define the compact set
\begin{align}\label{eqn:definition_infinity_exchangeable_PLG}
    \operatorname{Ext}_{\infty}(K,k) = \bigcap_{n\geq k} \operatorname{Ext}_n(K,k).
\end{align} 
For any $n\in\mathbb{N}$, let $y_n\in \lrbracket{K_A\tmax K_B}^{\tmax n}$ be an element of a \emph{bicompatible sequence} $\left\{y_n\right\}_{n\in \mathbb{N}}$ \cite[Def.\ 1]{plavala2025polarization}, i.e. $y_{n-1}=\lrbracket{\id_{A_1^{n-1}B_1^{n-1}}\otimes1_{K_{A_n}}\otimes 1_{K_{B_n}}}\lrbracket{y_n}$ and $y_n$ is invariant under individual and independent permutation of the tensor factors of $K_A^{\tmax n}$, $K_B^{\tmax n}$, respectively, i.e.
\begin{align}\label{eqn:independent_symmetry_constraint}
    \lrbracket{S_{\sigma_A}\otimes S_{\sigma_B}}\lrbracket{y_n}=y_n, \quad \forall \sigma_A, \sigma_B\in S_n.
\end{align}
Let $\left\{y_n\right\}_{n\in \mathbb{N}}$ be a bicompatible sequence, then we have
\begin{align}
    y_{n,k}\coloneqq  \gamma_{K_A \tmax K_B}^{n, k}(y_n)\in \operatorname{Ext}_n(K_A \tmax K_B,k),\quad \forall k\leq n,
\end{align}
and $y_n=y_{n,n}$. Note that $y_n\in \operatorname{Ext}_n\lrbracket{K_A\tmax K_B, n}$ implies $\lrbracket{S_{\sigma}\otimes S_{\sigma}}\lrbracket{y_n}=y_n$ for any $\sigma\in S_n$, but not the stronger constraint formulated in \autoref{eqn:independent_symmetry_constraint}. By a de Finetti-type argument \cite[Lem.\ 2]{plavala2025polarization}, there exists a probability Borel measure $\mu$ on $K_A\times K_B$ such that an element $y_n$ from a bicompatible sequence can be written as 
\begin{align}
    y_n = \int_{K_A \times K_B} x_A^{\otimes n}\otimes x_B^{\otimes n} d\mu(x_A,x_B).
\end{align}

The polarization hierarchy is given by the following program for $n\geq2$, where $p=\hat{P}_{\hookrightarrow}$ and $f=\hat{F}_{\hookrightarrow}$:
\begin{align}\label{eqn:Polarization_hierarchy}
    \begin{split}
         \operatorname{PLG}^{(n)}= \displaystyle\min_{y_n\in \operatorname{Ext}_n\lrbracket{K_A\tmax K_B, n}} \ &p\lrbracket{y_{n,1}}\\
         \text{s.t. } & \lrbracket{S_{\sigma_A}\otimes S_{\sigma_B}}\lrbracket{y_n}=y_n,\\
         & \Pi\lrbracket{\lrbracket{f\otimes f}\lrbracket{\mathcal{S}_2\lrbracket{y_{n, 2}}}}=0,
    \end{split}
\end{align}
where $\mathcal{S}_2$ is the isomorphism between $K_A^{\tmax 2}\tmax K_B^{\tmax 2}$ and $\lrbracket{K_A\tmax K_B}^{\tmax 2}$ that regroups the tensor factors from $A_1A_2B_1B_2$ to $A_1B_1A_2B_2$; these tensor orderings are identified when taking marginals. We use the linear \emph{polarization map} $\Pi=\operatorname{id}_{V\otimes V}$, where $V$ is the codomain of $f$.\footnote{Other polarization maps and stronger constraints are discussed in \cite[Sec.\ II, Thm.\ 4]{plavala2025polarization}.} For this choice, $\Pi(\sum_j a_j\otimes a_j)=0$ implies $a_j=0$ for every member of any finite family. This is the strengthened condition of \cite[Prop.\ 4.1, Rem.\ 5.2]{zeiss2026commentpolarization}, which corrects the two-vector condition stated in \cite{plavala2025polarization}.

Importantly, a sequence $(y_n^*)_n$ of optimal points for \autoref{eqn:Polarization_hierarchy} need not be bicompatible, since $y^*_{n}$ and $y_{n-1}^*$ are not necessarily related by a partial trace. Assuming the product-feasible set is nonempty, marginalization preserves feasibility and the objective, while feasible products admit extensions at every level, so $\operatorname{PLG}^{(n)}\leq\operatorname{PLG}^{(n+1)}\leq p_1$. Compactness and a diagonal subsequence argument \cite[Thm.\ 5.4]{zeiss2026commentpolarization} yield indices $n_j\to\infty$ such that $y^*_{n_j,k}$ converges to a point $x_k^*$ for every fixed $k$. Marginal consistency and the independent permutation symmetries pass to this entire limiting family by continuity, so $(x_k^*)_{k\in\mathbb{N}}$ is bicompatible. The polarization constraint passes to $x_2^*$ and the objective to $x_1^*$. By \cite[Lem.\ 2]{plavala2025polarization}, the resulting bicompatible family has the representation
\begin{align}
    \mathcal{S}_2\lrbracket{x_2^*}=\int_{K_A\times K_B}\lrbracket{x_A\otimes x_B}^{\otimes 2}\operatorname{d}\mu_A\lrbracket{x_A, x_B}
\end{align}
Fix a Euclidean inner product on the codomain of $f$. Applying its contraction to the identity-polarization constraint gives
\begin{align}
  0=\int_{K_A\times K_B}\|f(x_A,x_B)\|_2^2\,\mathrm d\mu_A(x_A,x_B).
\end{align}
Thus $\hat{F}_{\hookrightarrow}(x_A\otimes x_B)=f(x_A,x_B)=0$ almost everywhere with respect to $\mu_A$. The objective limit is therefore an average of feasible objective values and is at least $p_1$. Conversely, every feasible product has feasible extensions at every level. These observations establish asymptotic convergence to $p_1$ for the identity polarization map when $\hat{G}_{\hookrightarrow}$ is the zero map. The local-constraint specialization is detailed in Appendix~\ref{sec:Comparison_to_Polarization_Hierarchy}.  

\subsubsection{This work}
The polarization hierarchy extends the DPS framework to provide convergence guarantees in settings where DPS alone is insufficient, but uses extensions on both parties with independent permutation symmetries. In this work, we propose an alternative approach. While our method enjoys several advantages -- most notably analytical finite-level bounds on approximation errors and compatibility with inequality constraints -- the proof technique underlying our construction entails one essential drawback: the constraint maps are required to exhibit a specific structural form. In \autoref{eq:bivariate_problem}, assume the following mapping
\begin{align}
    f : K_A \times K_B \to U_A \oplus U_B, \quad (x_A,x_B) \mapsto f_A(x_A) \oplus f_B(x_B)
\end{align}
where $f_A : K_A \to U_A$ and $f_B : K_B \to U_B$ are affine functions. Let $\mathcal{F}\subseteq K_A\times K_B$ denote the set of states which satisfy $f(x_A, x_B)=0$. We show in detail in \autoref{sec:affine_maps} that there exist \emph{unique} linear functions $F_A$ and $F_B$ corresponding to $f_A$ and $f_B$ such that    
\begin{align}
    \mathcal{F} = \{(x_A,x_B) \in K_A \times K_B \ : \ F_A(x_A) = 0,\ F_B(x_B) = 0\},
\end{align}
and, by \autoref{fig:affine_maps_to_linear_maps}, there exists a linear map $F:A(K_A)^\star\otimes A(K_B)^\star\to U_A\oplus U_B$, namely $F=(F_A\otimes1_{K_B})\oplus(1_{K_A}\otimes F_B)$, such that
\begin{align}
    \mathcal{F} = \{(x_A,x_B) \in K_A \times K_B \ : \ F(x_A \otimes x_B) = 0\}.
\end{align}
Thus, in this work, we consider problems of the form
\begin{align}\label{eq:final_optimization_for_linear_constraints}
\begin{split}
    p^{*}= \min_{(x_A,x_B)\in K_A\times K_B} \quad & P\lrbracket{x_A \otimes x_B}\\
    \text{s.t.} \quad & F_A(x_A) = 0 \quad \text{and} \quad F_B(x_B)=0 \\
    & G_A(x_A)\geq 0 \quad \text{and} \quad G_B(x_B) \geq 0,
\end{split}
\end{align}
where $F_A, F_B, G_A,$ and $G_B$ are linear. The inequality codomains are finite-dimensional ordered spaces with specified closed proper cones $D_A,D_B$, and $G_i(x_i)\geq0$ means $G_i(x_i)\in D_i$. We assume that the local feasible sets
\begin{align}
K_i^{\mathrm{feas}}&\coloneqq\{x_i\in K_i:F_i(x_i)=0,\ G_i(x_i)\in D_i\},\quad i=A,B,
\end{align}
are nonempty. Define
\begin{align}
\begin{split}
\mathcal P_F&\coloneqq\{x_A\otimes x_B:x_A\in K_A^{\mathrm{feas}},\ x_B\in K_B^{\mathrm{feas}}\},\\
\mathcal S_F&\coloneqq\operatorname{conv}(\mathcal P_F).
\end{split}
\end{align}
Both sets are compact. Since $P$ is linear, $\min_{\mathcal P_F}P=\min_{\mathcal S_F}P=p^*$: the value of a feasible mixture is an average of feasible-product values, and at least one component has value no larger than that average. Thus \autoref{eq:final_optimization_for_linear_constraints} has an equivalent constrained-separable formulation over $\mathcal S_F$. Membership in $\mathcal S_F$ requires a decomposition in which every product component satisfies the local constraints; merely imposing those constraints on averaged marginals is generally weaker.

The local constraints above include familiar conditions from quantum information: prescribed marginals of normalized Choi matrices in quantum coding problems \cite[Lem.\ 4.2]{Berta2021}, and non-signalling identities within assemblages in quantum game formulations \cite[Lem.\ 3.1]{zeiss2025approximatingfixedsizequantum}. Local inequalities can impose bounds on measurement probabilities or expectation values, as illustrated in \autoref{sec:numerics}. Here locality refers to the optimization variables: each condition restricts $x_A$ or $x_B$ separately, even when that variable contains several physical registers.

Based on a novel information-theoretic de Finetti theorem in \autoref{sec:information_theoretic_de_finetti_representation_theorem}, we provide a converging hierarchy of outer approximations to $p^*$ in \autoref{sec:convergent_hierarchy_for_polynomial_optimization}. For general coupled constraints imposed only on the average, as in $p_3$, the DPS hierarchy still converges asymptotically when the feasible set is nonempty and the constraints are closed. The finite-level rate proved here for local constraints does not automatically extend to that setting: it would additionally require control of the distance from an approximately feasible separable point to an exactly feasible one. See \autoref{sec:Comparison_to_Polarization_Hierarchy} for a detailed analysis of these constraints within the polarization hierarchy.

\section{Information-theoretic de Finetti representation theorem}\label{sec:information_theoretic_de_finetti_representation_theorem}
In this section we extend techniques from quantum information theory to the GPT framework and prove novel information-theoretic finite de Finetti theorems for arbitrary convex bodies. These results provide a rigorous quantitative characterization of the monogamy-of-entanglement between cones, strengthening the asymptotic relation obtained in \cite{Aubrun_2024}.

\subsection{Relative entropy and mutual information for convex bodies}
A central question in GPTs is the formulation of a suitable divergence, serving as an analogue of the relative entropy in quantum theory. From an operational perspective, one seeks quantities that characterize tasks such as hypothesis testing, mirroring, for example, the role played by the Umegaki relative entropy \cite{Umegaki1962RelativeEntropy} in the quantum Stein’s lemma \cite{Hiai1991}. From a purely mathematical viewpoint, one instead looks for quantities on state space that satisfy structural properties such as the data-processing inequality, positivity, additivity, or joint convexity -- each of which has a significant impact on the analyses in which the divergence is employed. In this work, we use a recently developed integral representation of the quantum relative entropy \cite{Frenkel_2023}, which has the advantage of naturally generalizing to GPTs and of coinciding with the Kullback-Leibler (KL) divergence $D_{\operatorname{KL}}(\cdot \Vert \cdot)$ on simplices.

Consider a state space $K$. For any $x,y\in K$, we define their relative entropy according to \cite{Frenkel_2023} as:
\begin{align}\label{eq:def_relative_entropy}
    D(x\Vert y)\coloneqq \int_{\mathbb{R}} \frac{dt}{\vert t\vert (t-1)^2} \sup_{f\in E(K)} -[(1-t)f(x) + tf(y)].
\end{align}
The nonnegative integral is understood as an improper integral with values in $[0,+\infty]$, and all logarithms are natural. It obeys the \emph{data processing inequality} (DPI) for channels \cite[Thm.\ 14]{Frenkel_2023}: if $\Phi:K\to K'$ is a channel, every effect on $K'$ pulls back to an effect on $K$, so comparison of the integrands gives $D(\Phi(x)\Vert\Phi(y))\leq D(x\Vert y)$.

Since states in GPTs, much like quantum states, can be viewed as preparation procedures within a statistical theory, the notion of mutual information can be defined in an analogous way. \emph{Mutual information} for probability distributions quantifies how much knowledge about one distribution can be gained from knowing another. In quantum theory, this concept carries over directly to quantum states. There, mutual information can be defined via the Umegaki relative entropy, and it is well known that it also admits the following variational characterization (cf., e.g., \cite{Berta_2011})
\begin{align}\label{eq:formulation_mutual_information_quantum_theory}
    I(A:B)_{\rho_{AB}} = \inf_{\sigma_A \in \mathcal{S}(\mathcal{H}_A)} D(\rho_{AB} \Vert \sigma_A \otimes \rho_B), \quad \rho_{AB} \in \mathcal{S}(\mathcal{H}_A \otimes \mathcal{H}_B),
\end{align}
where $\mathcal{S}(\mathcal{H}_A \otimes \mathcal{H}_B)$ denotes the quantum state space on finite dimensional Hilbert space $\mathcal{H}_A\otimes\mathcal{H}_B$. Since GPTs generally lack an intrinsic notion of entropy \cite{kimura2010distinguishability,kimura2016entropies,takakura2019entropy}, but do admit a relative entropy given by \autoref{eq:def_relative_entropy}, we adopt the concept behind \autoref{eq:formulation_mutual_information_quantum_theory} to our setting. For state spaces $K_A$ and $K_B$ the mutual information w.r.t.\ a state is defined as
\begin{align}\label{eq:def_mutual_information}
    I(A:B)_{x_{AB}} \coloneqq \inf_{y_A \in K_A} D(x_{AB} \Vert y_A \otimes x_B), \quad x_{AB} \in K_A \hat{\otimes} K_B,
\end{align}
and where $x_B$ is the reduced state of $x_{AB}$. Here $D$ is evaluated on the maximal composite state space, which need not coincide with the usual joint quantum state space even when the local cones are positive-semidefinite. The following result follows from \cite[Cor.\ 1]{Jen_ov__2024}.
\begin{proposition}\label{prop:jencova}
    Consider state space $K$ and $0\leq\mu\leq1\leq\lambda<\infty$. Then, for all $x,y\in K$ obeying\footnote{Order induced by the ambient cone.} $\mu y \leq x \leq \lambda y$, we obtain
    \begin{align}
        D(x\Vert y) = \int_\mu^\lambda \frac{ds}{s} \left[\sup_{f\in E(K)} sf(y) - f(x) \right] + \ln\lambda +1 - \lambda.
    \end{align}
\end{proposition}
\begin{proof}
    The proof adapts the change of variables in \cite[Cor.\ 1]{Jen_ov__2024} to the effect interval of $K$. The full argument is provided in \autoref{sec:Proofs_of_auxiliary_results}.
\end{proof}
A second result that can be directly adapted from the literature is a Pinsker-type inequality, formulated in terms of the divergence defined in \autoref{eq:def_relative_entropy}.
\begin{lemma}\label{lem:pinsker}
For every state space $K$ and all $x,y\in K$,
\begin{align}\label{eqn:pinsker_lemma}
D(x\Vert y)\geq \frac12\lVert x-y\rVert_{\star,1_K}^2
=2\left(\sup_{f\in E(K)}f(x-y)\right)^2.
\end{align}
\end{lemma}
\begin{proof}
Compactness of $E(K)$ gives an effect $f^*$ maximizing $f(x-y)$. The binary measurement $M=(f^*,1_K-f^*)$ satisfies
\begin{align}
\lVert M(x)-M(y)\rVert_1
=2f^*(x-y)=\lVert x-y\rVert_{\star,1_K}.
\end{align}
On a simplex, $D$ is the classical KL divergence. Its Pinsker inequality and the channel DPI give
\begin{align}
D(x\Vert y)\geq D_{\mathrm{KL}}(M(x)\Vert M(y))
\geq \tfrac12\lVert M(x)-M(y)\rVert_1^2,
\end{align}
which proves the claim, including the case of infinite divergence.
\end{proof}
We conclude this preliminary discussion of mutual information with several remarks on the definition in \autoref{eq:def_mutual_information}. Although we do not know whether the equality
\begin{align}\label{eq:questioned_equality}
    D(x_{AB} \Vert x_A \otimes x_B) \stackrel{?}{=} \inf_{y_A \in K_A} D(x_{AB} \Vert y_A \otimes x_B), \quad x_{AB} \in K_A \hat{\otimes} K_B
\end{align}
holds in general, the quantity defined in \autoref{eq:def_mutual_information} nevertheless enjoys several useful properties. In particular, it is nonnegative and we have for any local channels $\Phi_{A\to A^\prime}$ or $\Phi_{B \to B^\prime}$ 
\begin{align}\label{eq:DPI_mutual_information}
\begin{split}
   \inf_{y^\prime_A \in K_{A^\prime}} D(\Phi_{A\to A^\prime}\otimes \Phi_{B\to B^\prime}(x_{AB})&\Vert y_A^\prime \otimes \Phi_{B\to B^\prime}(x_{B})) \\
   &\leq \inf_{y_A \in K_A} D(x_{AB}\Vert y_A\otimes x_B),
\end{split}
\end{align}
because by the DPI for the relative entropy \cite[Sec.\ 8, Thm.\ 14]{Frenkel_2023} we have
\begin{align}
\begin{split}
    &\inf_{y_A \in K_A} D(x_{AB}\Vert y_A\otimes x_B) \\
    &\quad\geq \inf_{y_A \in K_A} D(\Phi_{A \to A^\prime}\otimes \Phi_{B\to B^\prime}(x_{AB})\Vert \Phi_{A \to A^\prime}(y_A)\otimes \Phi_{B\to B^\prime}(x_B)) \\
    &\quad\geq \inf_{y_A^\prime \in K_{A^\prime}} D(\Phi_{A \to A^\prime}\otimes \Phi_{B\to B^\prime}(x_{AB})\Vert y_A^\prime \otimes \Phi_{B\to B^\prime}(x_B)).
\end{split}
\end{align}
We may therefore use it as a well-defined notion of mutual information.

\subsection{Monogamy}
The question is whether, in general, there exists an upper bound $c(A)$ independent of $K_B$, depending solely on $K_A$, such that  
\begin{align}\label{eq:uniform_mutual_information_bound}
    I(A:B)_{x_{AB}} \stackrel{?}{\leq} c(A), \quad x_{AB} \in K_A \hat{\otimes} K_B.
\end{align}
The above question admits an affirmative answer, which we establish using the following lemmas.
\begin{lemma} \label{lemma:relEnt-bound}
Let $x,y\in K$ satisfy $x\leq\lambda y$. Then $\lambda\geq1$ and
\begin{align}
D(x\Vert y)\leq\ln\lambda.
\end{align}
\end{lemma}
\begin{proof}
Normalization gives $\lambda\geq1$. For every effect $f$, positivity of $1_K-f$ on $\lambda y-x$ gives
\begin{align}
f(y)\leq \frac{f(x)+\lambda-1}{\lambda}.
\end{align}
Consequently, for $0\leq s\leq\lambda$,
\begin{align}
sf(y)-f(x)
\leq (s/\lambda-1)f(x)+s(1-1/\lambda)
\leq s(1-1/\lambda).
\end{align}
Apply \autoref{prop:jencova} with $\mu=0$. The integral is at most $\lambda-1$, giving $D(x\Vert y)\leq\ln\lambda$.
\end{proof}
The following lemma formalizes the intuition that any state\footnote{Because $K$ lies in an affine hyperplane (a codimension-one slice of the cone $A(K)^{\star +}$), its ambient interior $\operatorname{int}(K)$ is typically empty, so the correct notion of “interior state” is the relative interior $\operatorname{relint}(K)$ taken within $\operatorname{aff}(K)$.} $\tau\in\operatorname{relint}(K)$ serves as a uniform “reference” or order unit for $A(K)^{\star +}$: every normalized state $x\in K$ is bounded above by a finite multiple of $\tau$, and the required factor can be chosen uniformly over $K$.
\begin{lemma} \label{lemma:cA-orderUnitBound}
For a state space $K$ and $\tau\in\operatorname{relint}(K)$,
\begin{align}
\lambda_K(\tau)\coloneqq
\max_{x\in K}\inf\{\lambda\geq0:x\leq\lambda\tau\}<\infty.
\end{align}
In particular, $x\leq\lambda_K(\tau)\tau$ for every $x\in K$, and $\lambda_K(\tau)\geq1$.
\end{lemma}
\begin{proof}
Under the homogeneous embedding, $\tau$ lies in the interior of $C=A(K)^{\star+}$ (\autoref{lem:strictly_positive_functionals}) and is therefore an order unit. For $x\in C$, the infimum in the statement equals its order-unit norm $\lVert x\rVert_\tau$: the lower inequality $-\lambda\tau\leq x$ is automatic. This norm is continuous, so it attains a finite maximum on compact $K$. Closedness of $C$ gives the asserted domination, and normalization gives $\lambda_K(\tau)\geq1$.
\end{proof}
The following proposition is central to the proof of our main theorem in \autoref{thm:main_theorem}, a finite de Finetti theorem for convex bodies, and provides an affirmative answer to the question raised in \autoref{eq:uniform_mutual_information_bound}.
\begin{proposition}\label{prop:uniform_A_upper_bound_mutual_information}
Fix $\tau_A\in\operatorname{relint}(K_A)$ and let $\lambda_A=\lambda_{K_A}(\tau_A)$ be as in \autoref{lemma:cA-orderUnitBound}. For every state space $K_B$ and $x_{AB}\in K_A\tmax K_B$,
\begin{align}
I(A:B)_{x_{AB}}
\leq D(x_{AB}\Vert\tau_A\otimes x_B)
\leq\ln\lambda_A.
\end{align}
\end{proposition}
\begin{proof}
Write $C_i=A(K_i)^{\star+}$. For $g_B\in C_B^\star$, the conditional cone element
\begin{align}
a_g=(\operatorname{id}_A\otimes g_B)(x_{AB})
\end{align}
belongs to $C_A$ and has mass $g_B(x_B)$. If this mass is zero, $a_g=0$ because $1_{K_A}$ is strictly positive on $C_A\setminus\{0\}$ (\autoref{lem:strictly_positive_functionals}(a)). Otherwise normalize $a_g$ and apply \autoref{lemma:cA-orderUnitBound}. In both cases,
\begin{align}
a_g\leq\lambda_A g_B(x_B)\tau_A.
\end{align}
Testing this inequality with every $f_A\in C_A^\star$ proves
\begin{align}
x_{AB}\leq\lambda_A\tau_A\otimes x_B
\end{align}
in the maximal tensor cone. The two states are normalized, so \autoref{lemma:relEnt-bound} bounds their divergence by $\ln\lambda_A$. Choosing $\tau_A$ in the variational definition of $I$ completes the proof.
\end{proof}
\autoref{prop:uniform_A_upper_bound_mutual_information} shows that the mutual information is uniformly bounded above by a constant that depends only on the construction of the cone $A(K_A)^{\star +}$, namely the parameter $\lambda_A$	introduced in \autoref{lemma:cA-orderUnitBound}. As we shall see below, the relations derived here can be leveraged to obtain a rigorous quantitative characterization of the monogamy-of-entanglement between cones.

\subsection{de Finetti theorem}

Broadly speaking, a finite de Finetti theorem asserts that, for a permutation-invariant state on the maximal tensor product of many identical state spaces, any reduced state is close  --  in an appropriate sense  --  to an element of the minimal tensor product. Such results provide a fundamental structural characterization of state spaces and, more generally, of cones, and consequently admit a wide range of applications. Accordingly, numerous versions of finite de Finetti theorems have appeared in the literature.

Our approach follows a proof strategy adapted from \cite{Brando2017}, which introduces a parent quantity that bounds the total amount of information shared across the entire $n$-fold tensor-power state. If this quantity is finite and quantifies information that is uniformly shared among all subsystems, one can show that an averaged conditional information increment vanishes asymptotically as the number of subsystems $n\rightarrow\infty$. Notably, mutual information provides precisely such a measure of correlation. In particular, the availability of a chain rule (at least on simplices), together with the validity of the data-processing inequality on arbitrary cones, yields, in principle, a de Finetti-type argument. 
For the purposes of this argument, we recall in the following lemma a well-known fact concerning informationally complete measurements.

\begin{lemma}\label{lem:injectivity}
    Let $T_A:V_A\to W_A$ be an injective linear map between finite dimensional vector spaces. Then, for any finite dimensional vector space $V_B$, the map
    \begin{align}
        T_A\otimes \operatorname{id}_B:V_A \otimes V_B \to W_A \otimes V_B
    \end{align}
    is injective. In particular, for any choice of norms $\Vert \cdot \Vert_{\operatorname{in}}$ on $V_A\otimes V_B$ and $\Vert\cdot \Vert_{\operatorname{out}}$ on $W_A \otimes V_B$, there exists a constant $\tilde{f}(A,B,T_A)\in\mathbb{R}_+$, depending only on $V_A$, $V_B$ and $T_A$ and the chosen norms, such that 
    \begin{align}
        \Vert T_A \otimes \operatorname{id}_B(X_{AB})\Vert_{\operatorname{out}} \geq \tilde{f}(A,B,T_A) \Vert X_{AB}\Vert_{\operatorname{in}}, \quad X_{AB} \in V_A\otimes V_B.
    \end{align}
\end{lemma}
\begin{proof}
    See \autoref{sec:Proofs_of_auxiliary_results}.
\end{proof}
Before stating the main theorem, we introduce a key notion that will be used throughout this section. Let $K_A$ and $K_B$ be state spaces, $\tilde{\otimes}$ any tensor product and $\Omega$ an alphabet of finite cardinality. Given a bipartite state $x_{AB}\in K_A\tilde{\otimes} K_B$ and a measurement channel $\mathcal{M}\,:\, K_B\rightarrow \Delta_{\lvert \Omega \rvert}$ with effects $\left\{M_B^z\right\}_{z\in \Omega}$, the \emph{unnormalized conditional state} is the positive element $\tilde{x}_{A\vert z}\in A(K_A)^{\star +}$ defined by
\begin{align}
    \tilde{x}_{A\vert z}\coloneqq  \lrbracket{\operatorname{id}_A\otimes M^z_B}(x_{AB}).
\end{align}
Note that $\sum_{z\in \Omega} M_B^z=1_{K_B}$ implies $\sum_{z\in\Omega} \tilde{x}_{A\vert z}=x_A$. The probability of obtaining outcome $z\in \Omega$ is 
\begin{align}
    p(z)=\lrbracket{1_{K_A}\otimes M^z_B}(x_{AB})
\end{align}
such that, whenever $p(z)>0$, by \cite[Prop\ 5.17]{Pl_vala_2023},
\begin{align}
    x_{A\vert z}=\frac{\tilde{x}_{A\vert z}}{p(z)}
\end{align}
is the \emph{normalized conditional state}. Outcomes of zero probability are omitted from all sums of normalized conditional states below. We can now establish the main result of this paper, which we formulate as the following theorem\footnote{The extension to cones follows directly from \cite[Lem.\ 3.34]{Pl_vala_2023}.}.
Fix informationally complete measurements $\MM_{A\to Y}:K_A\to\Delta_{|Y|}$ and $\MM_{B\to Z}:K_B\to\Delta_{|Z|}$, and put
\begin{align}
T=\MM_{A\to Y}\otimes\MM_{B\to Z},\qquad
W_{AB}=\ker(1_{K_A}\otimes1_{K_B}).
\end{align}
These measurements are injective on the homogeneous spaces: if $\MM(X)=0$, then $1_K(X)=0$, and the decomposition $X=t(a-b)$ reduces injectivity to that on states. Thus $T$ is injective by \autoref{lem:injectivity}. Define
\begin{align}\label{eq:def_definetti_constant}
\begin{split}
\kappa(A,B)&\coloneqq
\sup_{0\ne X\in W_{AB}}
\frac{\lVert X\rVert_{\star,1_{K_A\tmax K_B}}}{\lVert T(X)\rVert_1},\\
c_{\mathrm{cl}}(A)&\coloneqq\min\{\ln\lambda_A,\ln|Y|\},\qquad
c(A,B)\coloneqq\kappa(A,B)\sqrt{c_{\mathrm{cl}}(A)/2}.
\end{split}
\end{align}
We set $\kappa(A,B)=0$ if $W_{AB}=\{0\}$. Otherwise it is finite by compactness of the unit sphere and injectivity. The measurements and $\tau_A$ are fixed for each pair of bodies, so these constants are independent of the extension order.

\begin{theorem}[de Finetti for convex bodies]\label{thm:main_theorem}
If $x_{AB}\in\operatorname{Ext}_n\lrbracket{K_A,K_B}$, then there exists $\tilde{x}_{AB}\in K_A\tmin K_B$ such that
\begin{align}
\left\lVert\tilde{x}_{AB}-x_{AB}\right\rVert_{\star,1_{K_A\tmax K_B}}
\leq\frac{2c(A,B)}{\sqrt n}.
\end{align}
Here $c(A,B)$ is defined in \autoref{eq:def_definetti_constant}. The state $\tilde{x}_{AB}$ can be constructed as a mixture of products of conditional marginals obtained by measuring $m\in\{0,\ldots,n-1\}$ copies of a symmetric extension with the fixed measurement $\MM_{B\to Z}$.
\end{theorem}
\begin{proof}
Let $y_{AB^n}\in K_A\tmax K_B^{\tmax n}$ be a permutation-invariant extension of $x_{AB}$. If $W_{AB}=\{0\}$, the normalized composite consists of a single product state, so the claim is exact with $m=0$. Otherwise apply the fixed measurements to obtain the classical distribution
\begin{align}
Q_{YZ^n}=(\MM_{A\to Y}\otimes\MM_{B\to Z}^{\otimes n})(y_{AB^n}).
\end{align}
Channel DPI and \autoref{prop:uniform_A_upper_bound_mutual_information}, together with $I(Y:Z^n)_Q\leq H(Y)_Q\leq\ln|Y|$, give
\begin{align}\label{eq:proof_deFinetti_mutual_information_bound}
I(Y:Z^n)_Q\leq c_{\mathrm{cl}}(A).
\end{align}
The classical chain rule reads
\begin{align}
I(Y:Z^n)_Q=\sum_{j=0}^{n-1}I(Y:Z_{j+1}\mid Z^j)_Q,
\end{align}
where
\begin{align}\label{eq:proof_deFinetti1}
I(Y:Z_{j+1}\mid Z^j)_Q
=\sum_h p(h)D_{\mathrm{KL}}\!\left(
Q_{YZ_{j+1}\mid h}\,\middle\Vert\,
Q_{Y\mid h}\otimes Q_{Z_{j+1}\mid h}\right).
\end{align}
Here $h=z^j=(z_1,\ldots,z_j)\in Z^j$ denotes a measurement-outcome record, and the sum includes only those with positive probability, for which normalization is defined. Choose $m$ to be the smallest index minimizing these $n$ conditional mutual informations. Then
\begin{align}\label{eq:proof_deFinetti_existence_m}
\sum_h p(h)D_{\mathrm{KL}}\!\left(
Q_{YZ_{m+1}\mid h}\,\middle\Vert\,
Q_{Y\mid h}\otimes Q_{Z_{m+1}\mid h}\right)
\leq\frac{c_{\mathrm{cl}}(A)}{n}.
\end{align}

For a measurement-outcome record $h=z^m=(z_1,\ldots,z_m)\in Z^m$, measure the first $m$ B copies with effects $M_B^{z_1},\ldots,M_B^{z_m}$, retain $A,B_{m+1}$ and marginalize the remaining copies. The normalized conditional state is denoted by $\sigma^h_{AB}$, with marginals $a_h,b_h$. These are physical states because the contracted effects are positive. Completeness of the measurement and permutation invariance give
\begin{align}
\sum_h p(h)\sigma^h_{AB}=x_{AB}.
\end{align}
Moreover,
\begin{align}
T(\sigma^h_{AB})=Q_{YZ_{m+1}\mid h},\qquad
T(a_h\otimes b_h)=Q_{Y\mid h}\otimes Q_{Z_{m+1}\mid h}.
\end{align}
Define the separable state
\begin{align}
\tilde{x}_{AB}=\sum_h p(h)a_h\otimes b_h.
\end{align}
The triangle inequality, classical Pinsker and Jensen's inequality yield
\begin{align}
\begin{split}
\lVert T(x_{AB}-\tilde{x}_{AB})\rVert_1
&\leq\sum_h p(h)\lVert Q_{YZ_{m+1}\mid h}
-Q_{Y\mid h}\otimes Q_{Z_{m+1}\mid h}\rVert_1\\
&\leq\sum_h p(h)\sqrt{2D_{\mathrm{KL}}\!\left(
Q_{YZ_{m+1}\mid h}\,\middle\Vert\,
Q_{Y\mid h}\otimes Q_{Z_{m+1}\mid h}\right)}\\
&\leq\sqrt{2c_{\mathrm{cl}}(A)/n}.
\end{split}
\end{align}
The difference $x_{AB}-\tilde{x}_{AB}$ belongs to $W_{AB}$, so \autoref{eq:def_definetti_constant} gives
\begin{align}
\lVert x_{AB}-\tilde{x}_{AB}\rVert_{\star,1_{K_A\tmax K_B}}
\leq\kappa(A,B)\sqrt{2c_{\mathrm{cl}}(A)/n}
=\frac{2c(A,B)}{\sqrt n}.
\end{align}
When $m=0$, no copies are measured; there is a single empty measurement-outcome record $h=z^0=()$, occurring with probability one, and the conditional local states are the original marginals $x_A$ and $x_B$. This is necessarily the case when $n=1$.
\end{proof}

\subsection{Computing the upper bound $c(A,B)$}
The constant in \autoref{eq:def_definetti_constant} separates an information bound $c_{\mathrm{cl}}(A)$ from the distortion $\kappa(A,B)$ of the fixed measurements. The minimum defining $c_{\mathrm{cl}}(A)$ combines two upper bounds on the same classical mutual information. For the measured distribution $Q_{YZ^n}$ in the proof of \autoref{thm:main_theorem}, data processing and \autoref{prop:uniform_A_upper_bound_mutual_information} give
\begin{align}
I(Y:Z^n)_Q\leq I(A:B^n)_{y_{AB^n}}\leq\ln\lambda_A,
\end{align}
whereas the classical entropy bound gives $I(Y:Z^n)_Q\leq H(Y)_Q\leq\ln|Y|$. Either bound feeds into the same classical chain rule and Pinsker argument, so we may use the smaller one. The conditioning argument of \cite[Lem.~2.1]{Berta2021} similarly bounds the information after measuring the $B$ systems, but retains $A$ as a quantum system and uses $I(A:Z^n)\leq\ln d_A$, with logarithms expressed here in natural units.

The term $\ln\lambda_A$ can also be obtained directly after measurement, without conic relative entropy. Indeed, put $\rho=\MM_{A\to Y}(\tau_A)$. Applying the measurement to the conditional-state bound $x_{A\mid z^n}\leq\lambda_A\tau_A$ gives the pointwise inequality $Q_{YZ^n}\leq\lambda_A\rho\otimes Q_{Z^n}$. Hence
\begin{align}
I(Y:Z^n)_Q+D_{\mathrm{KL}}(Q_Y\Vert\rho)
=D_{\mathrm{KL}}(Q_{YZ^n}\Vert\rho\otimes Q_{Z^n})
\leq\ln\lambda_A.
\end{align}
Thus $c_{\mathrm{cl}}(A)$ records the better of two information bounds; the conic result additionally controls the mutual information before measurement.

For fixed measurements, one may optimize the reference state $\tau_A$ through the order-domination parameter $\lambda_{K_A}(\tau_A)$. For a non-singleton Euclidean ball, its center gives $\lambda_A=2$; when $K_A$ is a singleton, $\lambda_A=1$, so $c_{\mathrm{cl}}(A)=0$ and the de Finetti conclusion is exact.

An explicit bound on the distortion is obtained from any linear left inverse $S$ of $T$:
\begin{align}
\kappa(A,B)\leq
\max_{y\in Y,\ z\in Z}
\lVert S(e_y\otimes e_z)\rVert_{\star,1_{K_A\tmax K_B}},
\end{align}
where $e_y,e_z$ are the classical basis states. Indeed, the right-hand side is the operator norm of $S$ from $\ell_1$ to the base norm. Both this distortion and the order-domination parameter depend on the dimension, geometry and chosen measurements. No dimension-independent bound or automatic efficient procedure for general cone descriptions is asserted.

\section{A convergent hierarchy for polynomial optimization}\label{sec:convergent_hierarchy_for_polynomial_optimization}

Recall, the optimization problem in \autoref{eq:final_optimization_for_linear_constraints} given by
\begin{align}
\begin{split}
    p^{*}= \min_{(x_A,x_B)\in K_A\times K_B} \quad &  P\lrbracket{x_A \otimes x_B}\\
    \text{s.t.} \quad & F_A(x_A) = 0 \quad \text{and} \quad F_B(x_B)=0 \\
    & G_A(x_A)\geq 0 \quad \text{and} \quad G_B(x_B) \geq 0,
\end{split} \label{eq:optimization_for_local_constraints}
\end{align}
where $F_A, F_B, G_A,$ and $G_B$ are linear maps. We retain the closed-cone and nonemptiness assumptions of \autoref{sec:comparison_DPS_PLG}, and equivalently write $p^*=\min_{w\in\mathcal S_F}P(w)$, with $\mathcal S_F$ the convex hull of locally feasible products defined there. In this section, we develop a hierarchy of outer approximations to \autoref{eq:optimization_for_local_constraints}, whose quantitative convergence is guaranteed by \autoref{thm:main_theorem}.

\subsection{The hierarchy}\label{sec:hierarchy_definition}

Consider $x_{AB}\in K_A\tmax K_B$ and $n\in\mathbb{N}$. Then, $x_{AB} \in \Sigma^{n}(A:B)\subseteq K_A\tmax K_B$ if 
\begin{enumerate}
    \item[(1)] $x_{AB}\in\operatorname{Ext}_n\lrbracket{K_A, K_B}$, i.e.\ there exists a permutation-invariant $n$-extension $y_{AB^n} \in K_A \hat{\otimes}K_B^{\hat{\otimes}n}$ of $x_{AB}$,
    \item[(2)] For $y_{AB_1^n}$ we have:
    \begin{enumerate}
        \item $F_A \otimes \operatorname{id}_{B^n}(y_{AB^n}) = 0$,
        \item $\operatorname{id}_{A} \otimes F_B \otimes \operatorname{id}_{B^{n-1}}(y_{AB^n}) = 0$,
        \item $G_A \otimes \operatorname{id}_{B^n}(y_{AB^n}) \geq 0$,
        \item $\operatorname{id}_{A} \otimes G_B \otimes \operatorname{id}_{B^{n-1}} (y_{AB^n}) \geq 0$.
    \end{enumerate}
\end{enumerate}
Here $\operatorname{id}_{B^n}$ is a shorthand for the identity channel on $K_B^{\tmax n}$. The inequalities in (2c) and (2d) use the maximal tensor cones obtained by replacing the corresponding A or B cone by $D_A$ or $D_B$, respectively. Each $w=\sum_\lambda p_\lambda x_A^\lambda\otimes x_B^\lambda\in\mathcal S_F$ has the feasible symmetric extension $\sum_\lambda p_\lambda x_A^\lambda\otimes(x_B^\lambda)^{\otimes n}$. The feasible extension set is a closed subset of a compact state space, so its marginal image $\Sigma^n(A:B)$ is nonempty and compact. Marginalizing one copy preserves all constraints. Note that $\mathcal S_F\subseteq\Sigma^{n+1}(A:B)\subseteq\Sigma^{n}(A:B)$ and we define
\begin{align}
    \Sigma(A:B)\coloneqq\bigcap_{n\geq 1}\Sigma^n(A:B).
\end{align}
such that $\Sigma(A:B)\subseteq K_A\tmin K_B$. Let $p^*$ denote the optimal value of the polynomial optimization problem with local constraints considered in \autoref{eq:optimization_for_local_constraints}.
\begin{theorem}[Convergence of the hierarchy]\label{thm:convergence_hierarchy}
    Let $K_A$ and $K_B$ be state spaces, with nonempty local feasible sets and closed proper inequality cones as above, $n\in\mathbb{N}$ and consider the optimization program 
        \begin{align}\label{eq:n_level_relaxation}
        \begin{split}
            p^{(n)}\coloneqq \min \ &  P\lrbracket{x_{AB}}\\
            \operatorname{s.t.} \ & x_{AB}\in \Sigma^{n}(A:B).
        \end{split}
        \end{align}
    Then,
    \begin{align}
       0\leq  p^* - p^{(n)}  \leq \frac{2\ \Vert P\Vert_{1_{K_A\tmax K_B}}c(A,B)}{\sqrt{n}},
    \end{align}
    where the constant $c(A,B)$ is the one appearing in \autoref{thm:main_theorem}.
\end{theorem}
\begin{proof}
Let $x^{(n)}_{AB}$ attain the minimum in \autoref{eq:n_level_relaxation}, and let $y^{(n)}_{AB^n}$ be a feasible symmetric extension. Fix the informationally complete measurements $\MM_{A\to Y}$ and $\MM_{B\to Z}$ from the proof of \autoref{thm:main_theorem}, and put $u_i=1_{K_i}$. For $0\leq m<n$ and a measurement-outcome record $z^m=(z_1,\ldots,z_m)\in Z^m$, define
\begin{align}
    e_{z^m}\coloneqq M_{B_1}^{z_1}\otimes\cdots\otimes M_{B_m}^{z_m},
\end{align}
where $M_{B_j}^{z_j}$ denotes the effect for outcome $z_j$ of the fixed measurement $\MM_{B\to Z}$ acting on copy $B_j$. The measurement is the same on every copy, while the entries $z_1,\ldots,z_m$ may differ. For $m=0$, set $e_{z^0}=1$, the empty tensor product. The unnormalized conditional joint state is
\begin{align}
    \widehat x_{AB_{m+1}\mid z^m}
    \coloneqq
    \bigl(\operatorname{id}_A\otimes e_{z^m}\otimes\operatorname{id}_B
        \otimes u_B^{\otimes(n-m-1)}\bigr)(y^{(n)}_{AB^n}).
\end{align}
Its local marginals $\widehat x_{A\mid z^m}$ and $\widehat x_{B_{m+1}\mid z^m}$ have the same mass
$p(z^m)=(u_A\otimes u_B)(\widehat x_{AB_{m+1}\mid z^m})$.
For $p(z^m)>0$, divide these marginals by $p(z^m)$ to obtain
$x_{A\mid z^m}\in K_A$ and $x_{B_{m+1}\mid z^m}\in K_B$.
We retain only measurement-outcome records with positive probability, for which normalization is defined. When $m=0$, no copies are measured; there is a single empty measurement-outcome record $z^0=()$, occurring with probability one, and the conditional local states are the original marginals $x^{(n)}_A$ and $x^{(n)}_B$.

These conditional local states satisfy the constraints for every $m$ and every measurement-outcome record with positive probability.\footnote{This contraction argument parallels those in \cite[proof of Thm.\ 2.3]{Berta2021} and \cite[App.\ A]{zeiss2025approximatingfixedsizequantum} in the quantum case.}
Indeed, the equalities commute with contraction on the other tensor factors, and symmetry moves the B constraint to the retained copy. Thus
$F_A(x_{A\mid z^m})=0$ and $F_B(x_{B_{m+1}\mid z^m})=0$.
For the A inequality, every $\varphi_A\in D_A^\star$ satisfies
\begin{align}
    \varphi_A\bigl(G_A(\widehat x_{A\mid z^m})\bigr)
    &=
    \bigl(\varphi_A\otimes e_{z^m}\otimes u_B^{\otimes(n-m)}\bigr)
    \bigl((G_A\otimes\operatorname{id}_{B^n})(y^{(n)}_{AB^n})\bigr)
    \geq0.
\end{align}
This is a product of positive functionals evaluated on the maximal tensor cone in (2c). Closedness of $D_A$ and duality give $G_A(\widehat x_{A\mid z^m})\in D_A$, and positive normalization preserves this membership. The same argument with $D_B^\star$, after moving (2d) to $B_{m+1}$, gives $G_B(x_{B_{m+1}\mid z^m})\in D_B$. Consequently, every resulting product belongs to $\mathcal P_F$.

The preceding argument establishes feasibility of the conditional products for every $0\leq m<n$ and every measurement-outcome record of positive probability. We now apply the approximation argument from the proof of \autoref{thm:main_theorem} to choose $m$ so that their mixture lies within $2c(A,B)/\sqrt n$ of $x^{(n)}_{AB}$ in the base norm. To this end, form the classical distribution
$Q^{(n)}_{YZ^n}=(\MM_{A\to Y}\otimes\MM_{B\to Z}^{\otimes n})(y^{(n)}_{AB^n})$
and choose the smallest index
\begin{align}
    m\in\mathop{\rm arg\,min}_{0\leq j<n}
    I(Y:Z_{j+1}\mid Z^j)_{Q^{(n)}}.
\end{align}
This choice fixes one $m$ for all measurement-outcome records. By the classical chain rule and an averaging argument, the minimum is no larger than the average, so
\begin{align}
\begin{split}
    I(Y:Z_{m+1}\mid Z^m)_{Q^{(n)}}
    &\leq \frac1n\sum_{j=0}^{n-1}I(Y:Z_{j+1}\mid Z^j)_{Q^{(n)}}\\
    &=\frac1n I(Y:Z^n)_{Q^{(n)}}
    \leq\frac{c_{\mathrm{cl}}(A)}n.
\end{split}
\end{align}
Expressing conditional mutual information as the average relative entropy between the conditional joint distribution and the product of its marginals, as in \eqref{eq:proof_deFinetti1}, gives precisely \eqref{eq:proof_deFinetti_existence_m}. Define the mixture over all measurement-outcome records with positive probability by
\begin{align}\label{eq:proof_convergence_hierarchy_candidate_state}
    \tilde x^{(n)}_{AB}
    =\sum_{\substack{z^m\in Z^m\\p(z^m)>0}}p(z^m)
      x_{A\mid z^m}\otimes x_{B_{m+1}\mid z^m}
    \in\mathcal S_F.
\end{align}
Applying the construction from the proof of \autoref{thm:main_theorem} to the feasible extension $y^{(n)}_{AB^n}$ with the fixed measurements and the chosen $m$ yields precisely the mixture in \eqref{eq:proof_convergence_hierarchy_candidate_state}. The approximation estimate established there therefore gives
\begin{align}
    \bigl\|\tilde x^{(n)}_{AB}-x^{(n)}_{AB}\bigr\|_{\star,1_{K_A\tmax K_B}}
       \leq\frac{2c(A,B)}{\sqrt n}.
\end{align}
The preceding contraction argument ensures that every product component satisfies the local constraints.

Since $\mathcal S_F\subseteq\Sigma^n(A:B)$, Hölder's inequality
(\autoref{prop:hoelder_ineq_order_unit_norm}) now gives
\begin{align}
\begin{split}
    0\leq p^*-p^{(n)}
    &\leq P(\tilde x^{(n)}_{AB})-P(x^{(n)}_{AB})\\
    &\leq \|P\|_{1_{K_A\tmax K_B}}
          \bigl\|\tilde x^{(n)}_{AB}-x^{(n)}_{AB}\bigr\|_{\star,1_{K_A\tmax K_B}}\\
    &\leq \frac{2\|P\|_{1_{K_A\tmax K_B}}c(A,B)}{\sqrt n}.
\end{split}
\end{align}
The objective values converge to $p^*$, and compactness gives optimal accumulation points of the mixtures in $\mathcal S_F$. The same construction applies to every point of $\Sigma^n(A:B)$. Since $\mathcal S_F$ is closed, the vanishing approximation error and the reverse inclusion already proved imply $\Sigma(A:B)=\mathcal S_F$.
\end{proof}

\begin{remark}[Per-level cost versus rate]
Write $d_i=\dim A(K_i)^\star$. Although an unsymmetrized extension has $d_A d_B^n$ coordinates, permutation invariance reduces the ambient dimension to
\begin{align}
    d_A\binom{n+d_B-1}{d_B-1},
\end{align}
which is polynomial in $n$ for fixed state spaces. For fixed polyhedral state and inequality cones, dual generators likewise give polynomially many product-effect tests after symmetry reduction. For example, if the state cones have $r_A,r_B$ dual generators, membership of the symmetric extension in the maximal tensor cone requires at most
$r_A\binom{n+r_B-1}{r_B-1}$ such tests. The lifted local constraints also have polynomially many rows for fixed maps and fixed polyhedral inequality cones. This polyhedral description and symmetry reduction are implemented in \autoref{sec:numerics}, where products of local dual-cone generators provide the linear inequalities describing the maximal tensor cone.

For an additive objective gap at most $\varepsilon>0$, the theorem gives the sufficient level
\begin{align}
    n\geq\max\left\{1,\left\lceil
       \left(\frac{2\|P\|_{1_{K_A\tmax K_B}}c(A,B)}{\varepsilon}\right)^2
       \right\rceil\right\}.
\end{align}
These statements concern fixed cone data; their dimensions, numbers of generators, and the constant $c(A,B)$ may all vary with the instance. For general cones, the cost depends additionally on the available descriptions of the tensor-cone constraints, as in other extension and moment hierarchies. The result supplies a finite-level error bound and the rounding guarantee below, without a general claim of efficient optimization over arbitrary proper cones.
\end{remark}

\subsection{Generating inner points}
We show that the proof technique underlying \autoref{thm:main_theorem} yields at each level of the hierarchy \autoref{eq:n_level_relaxation}, an explicit feasible mixture in $\mathcal S_F$ via a constructive rounding scheme, together with a quantitative bound on the resulting objective gap to the optimum $p^*$ of \autoref{eq:optimization_for_local_constraints}. A compactness argument then yields a subsequence of these feasible mixtures converging to an optimum in $\mathcal S_F$. Selecting a suitable product component also gives an original feasible pair with the same objective guarantee; no interiority is claimed. The argument is inspired by \cite{zeiss2025approximatingfixedsizequantum,kossmann2025symmetric}, but is formulated for general conic programs of the form studied here.

\begin{theorem}[Convergent inner sequence]\label{thm:inner_bounds}
    In the setting of \autoref{thm:convergence_hierarchy}, for each $n\in\mathbb{N}$, consider
    \begin{align}\label{eq:optimizer_thm_inner_bounds}
    \begin{split}
        x^{(n)}_{AB}
        \;\in\;
        \arg \min \ &
        P \lrbracket{x_{AB}}\\
        \operatorname{s.t.}\ & x_{AB} \in \Sigma^n(A:B)
    \end{split}
    \end{align}
    and denote by $y^{(n)}_{AB^n}$ a feasible symmetric extension. Then there exists an explicitly constructible sequence $\lrbracket{w^{(n)}_{AB}}_{n\geq 1}$ such that $w^{(n)}_{AB}\in\mathcal S_F$, the equivalent constrained-separable feasible set of \autoref{eq:final_optimization_for_linear_constraints}, and 
    \begin{align}
        \left\lVert w^{(n)}_{AB} - x^{(n)}_{AB}\right\rVert_{\star, 1_{K_A\tmax K_B}}
        \leq \frac{2 c(A,B)}{\sqrt{n}}.
    \end{align}
    Hence, 
    \begin{align}
        0\leq  P\lrbracket{w^{(n)}_{AB}} - p^*\leq  \Vert P\Vert_{1_{K_A\tmax K_B}}\frac{2 c(A,B)}{\sqrt{n}}.
    \end{align}
    In particular, $ P\lrbracket{w^{(n)}_{AB}}\rightarrow p^*$ as $n\rightarrow\infty$. One can also select an explicitly constructible product $v^{(n)}_{AB}\in\mathcal P_F$ satisfying
    \begin{align}
        0\leq P(v^{(n)}_{AB})-p^*\leq \Vert P\Vert_{1_{K_A\tmax K_B}}\frac{2c(A,B)}{\sqrt n}.
    \end{align}
    The norm estimate above concerns the mixture $w^{(n)}_{AB}$, not necessarily the selected product.
\end{theorem}
\begin{proof}
Choose $m$ by minimizing the classical conditional mutual information as in the proof of \autoref{thm:convergence_hierarchy}, and define
\begin{align}\label{eq:sigma_inner_cone_def}
    w^{(n)}_{AB}
    \coloneqq \sum_{z^m:\,p(z^m)>0}p(z^m)
       x_{A\mid z^m}\otimes x_{B_{m+1}\mid z^m}.
\end{align}
The contraction argument in that proof makes each product locally feasible, so $w^{(n)}_{AB}\in\mathcal S_F\subseteq\Sigma^k(A:B)$ for every $k\geq1$. The de Finetti estimate gives
\begin{align}\label{eq:proof_inner_bound_estimate_cone}
    \left\|w^{(n)}_{AB}-x^{(n)}_{AB}\right\|_{\star,1_{K_A\tmax K_B}}
       \leq \frac{2c(A,B)}{\sqrt n}.
\end{align}
For the product output, compute the conditional states for all $0\leq j<n$ and select
\begin{align}
\begin{split}
    (j_*,z^{j_*}_*)&\in
    \mathop{\rm arg\,min}_{\substack{0\leq j<n\\ z^j:\,p(z^j)>0}}
       P\bigl(x_{A\mid z^j}\otimes x_{B_{j+1}\mid z^j}\bigr),\\
    v^{(n)}_{AB}&\coloneqq
       x_{A\mid z^{j_*}_*}\otimes x_{B_{j_*+1}\mid z^{j_*}_*}.
\end{split}
\end{align}
Every candidate is in $\mathcal P_F$. Its minimum objective is no larger than the average defining $P(w^{(n)}_{AB})$. Feasibility and Hölder's inequality therefore yield the complete bracket
\begin{align}\label{eq:estimate_bound_inner_proof_objective_cone}
\begin{split}
    p^{(n)}\leq p^*
       &\leq P(v^{(n)}_{AB})\leq P(w^{(n)}_{AB})\\
       &\leq p^{(n)}
          +\|P\|_{1_{K_A\tmax K_B}}\frac{2c(A,B)}{\sqrt n}.
\end{split}
\end{align}
This proves the two objective guarantees and bounds the width of the resulting lower and upper bracket by the same error term. Compactness ensures optimal accumulation points, but neither the mixtures nor the selected products need converge to a single optimizer. The norm estimate \eqref{eq:proof_inner_bound_estimate_cone} applies to the mixture.
\end{proof}

\subsection{Representability}

In the GPT setting, state spaces may be arbitrary compact convex bodies, and their associated cones need not admit explicit descriptions suitable for computation. A conic lift can sometimes provide a more convenient representation by expressing a state space as the linear image of an affine slice of another cone \cite{THOMAS2019}. Its computational usefulness depends on the availability and size of the lifting data and on the description of the lifting cone. Following \cite{THOMAS2019}, we formalize this representation and show how it yields an equivalent bilinear optimization problem with conic constraints. Such a reformulation does not by itself make the resulting optimization tractable or circumvent known hardness results: quantum separability is NP-hard \cite{gharibian2009strongnphardnessquantumseparability}, even though the local quantum state spaces are trace-one affine slices of positive semidefinite cones.

Let $C\subseteq \mathbb{R}^m$ be a proper cone and $K \subseteq V$ a convex body in a finite-dimensional vector space $V$. A $C$-\emph{lift} of $K$ \cite[Def.\ 2.1]{THOMAS2019} consists of an affine subspace $L \subseteq \mathbb{R}^m$ and a linear map $T:\mathbb{R}^m\to V$ such that
    \begin{align}
        K = T(C \cap L).
    \end{align}
Equivalently, $K$ is the linear image of an affine slice $C \cap L$ of the cone $C$. A particularly important class of $C$-lifts arises when $C$ is the positive semidefinite cone. In this case, $C \cap L$ is a spectrahedron (an affine section of $C$), and its linear image $T(C \cap L)$ is a spectrahedral shadow (a projected spectrahedron). Spectrahedral shadows coincide with linear images of feasible regions of semidefinite programs (SDPs) (cf., e.g., \cite{THOMAS2019}).

To relate this framework to the bivariate polynomial optimization problem in \autoref{eq:final_optimization_for_linear_constraints}, use the canonical realizations $K_i\subseteq V_i=A(K_i)^\star$, $i=A,B$. We henceforth assume that 
$K_A\subseteq V_A$ and $K_B\subseteq V_B$ admit cone lifts. Concretely, there exist finite-dimensional vector spaces $W_A$ and $W_B$, proper cones $C_A \subseteq W_A$ and $C_B \subseteq W_B$, affine subspaces $L_A \subseteq W_A$ and $L_B \subseteq W_B$, and linear maps $T_A: W_A \to V_A$ and $T_B: W_B \to V_B$ such that 
\begin{align}\label{eqn:lift_data}
    K_A = T_A\lrbracket{C_A \cap L_A} \quad \text{and} \quad K_B = T_B\lrbracket{C_B \cap L_B}.  
\end{align}
Writing the affine subspaces as 
\begin{align}
\begin{split}
L_A&=\left\lbrace y_A\in W_A\, \vert\, R_A(y_A)=r_A\right\rbrace,\\
L_B&=\left\lbrace y_B\in W_B\, \vert\, R_B(y_B)=r_B\right\rbrace
\end{split}
\end{align}
for finite-dimensional vector spaces $U_A$, $U_B$, linear maps $R_A\,:\,W_A\rightarrow U_A$, $R_B\,:\,W_B\rightarrow U_B$ and vectors $r_A\in U_A$, $r_B\in U_B$, the lift assumption become
\begin{align}\label{eqn:lift_data_rewritten}
    K_A = T_A\lrbracket{C_A \cap R_A^{-1}(r_A)} \quad \text{and} \quad K_B = T_B\lrbracket{C_B \cap R_B^{-1}(r_B)}.
\end{align}
Thus, we can rewrite the optimization from \autoref{eq:final_optimization_for_linear_constraints} as a bilinear optimization problem with conic constraints.

\begin{proposition}\label{prop:lifted_conic_program}
   Let $K_A$ and $K_B$ be convex bodies. With the notation and standing assumptions of the preceding paragraph, assume that $K_A$ and $K_B$ admit conic lifts $K_A = T_A(C_A \cap L_A)$ and $K_B = T_B(C_B \cap L_B)$. Consider the cost functional $P$ in \autoref{eq:final_optimization_for_linear_constraints} and define the pulled-back cost functional $\tilde{P}\in (W_A\otimes W_B)^\star$ along $T_A\otimes T_B$ by 
   \begin{align}
       \tilde{P}\lrbracket{y_A\otimes y_B}\coloneqq P\lrbracket{T_A(y_A)\otimes T_B(y_B)},\quad y_A\in W_A,\, y_B\in W_B,
   \end{align}
   extended linearly to all of $W_A\otimes W_B$. Then the polynomial optimization problem in \autoref{eq:final_optimization_for_linear_constraints} is equivalent to the lifted bilinear problem with conic constraints
   \begin{align}\label{eq:lifted_problem}
        \begin{split}
            p^{*}= \inf_{(y_A,y_B) \in C_A \times C_B} \ & \tilde{P} \lrbracket{y_A \otimes y_B}\\
            \operatorname{s.t.}\quad
            &F_A\lrbracket{T_A(y_A)}=0, \quad F_B\lrbracket{T_B(y_B)}=0,\\
            &G_A\lrbracket{T_A(y_A)}\geq 0,\quad G_B\lrbracket{T_B(y_B)}\geq 0,\\
            &R_A(y_A)=r_A,\quad R_B(y_B)=r_B.
        \end{split}
   \end{align}
\end{proposition}
\begin{proof}
     Define the lifted feasible sets
    \begin{align}
    \begin{split}
        \widetilde{C}_A \coloneqq \Big\{y_A\in C_A \;\big|\;& R_A(y_A)=r_A,\; F_A\lrbracket{T_A(y_A)}=0,\\
        & G_A\lrbracket{T_A(y_A)}\geq 0\Big\},
    \end{split}
    \end{align}
    \begin{align}
    \begin{split}
        \widetilde{C}_B \coloneqq \Big\{y_B\in C_B \;\big|\;& R_B(y_B)=r_B,\; F_B\lrbracket{T_B(y_B)}=0,\\
        & G_B\lrbracket{T_B(y_B)}\ge 0\Big\}.
    \end{split}
    \end{align}
    If $y_A\in  \widetilde{C}_A$, then $y_A\in C_A\cap L_A$, hence $x_A\coloneqq T_A(y_A)\in T_A\lrbracket{C_A\cap L_A}$, and by construction $x_A$ satisfies the local constraints $F_A(x_A)=0$ and $G_A(x_A)\ge 0$. Thus $x_A$ is feasible for \autoref{eq:final_optimization_for_linear_constraints}. The same holds for $x_B\coloneqq T_B(y_B)$ with $y_B\in \widetilde{C}_B$. Conversely, if $x_A$ is feasible in \autoref{eq:final_optimization_for_linear_constraints}, then $x_A\in K_A=T_A\lrbracket{C_A\cap L_A}$, so there exists $y_A\in C_A\cap L_A$ with $T_A(y_A)=x_A$, i.e.\ $R_A(y_A)=r_A$. Since the constraints of \autoref{eq:final_optimization_for_linear_constraints} are imposed on $x_A$ and $x_A=T_A(y_A)$, this $y_A$ also satisfies $F_A\lrbracket{T_A(y_A)}=0$ and $G_A\lrbracket{T_A(y_A)}\geq 0$, hence $y_A\in  \widetilde{C}_A$. Likewise for the $B$-system. Therefore, feasible pairs $(x_A, x_B)$ of \autoref{eq:final_optimization_for_linear_constraints} are precisely the images of feasible points $(y_A,y_B)$ of the lifted program \autoref{eq:lifted_problem} via 
    \begin{align}
        x_A = T_A(y_A),\qquad x_B=T_B(y_B).
    \end{align}
    Finally, by definition of $\tilde{P}$,
    \begin{align}
         \tilde{P}\lrbracket{ y_A\otimes y_B}=P \lrbracket{T_A(y_A)\otimes T_B(y_B)}= P\lrbracket{x_A\otimes x_B}, 
    \end{align}
    So the two optimization problems have the same set of attainable objective values, hence the same infimum $p^*$.    
\end{proof}

A takeaway from \autoref{prop:lifted_conic_program} is that the complexity of the lifted formulation \autoref{eq:lifted_problem} is governed not only by the polynomial objective, but also by the difficulty of describing (and optimizing over) the lifting cones $C_A$ and $C_B$. If, for instance, $C_A$ and $C_B$ are polyhedral cones, positive semidefinite cones, or cones generated by simplices, then membership in each individual lifting cone has a standard LP or SDP description; bilinear optimization and membership in tensor cones remain separate issues. In contrast, for general lifting cones one must also account for the intrinsic computational and descriptive hardness of the cones $C_A$ and $C_B$ themselves.

The proposition establishes a value-preserving reformulation. Applying \autoref{thm:convergence_hierarchy} to lifted objects additionally requires compact normalized lifted state spaces, with interior order units on their proper cones, and a description of the resulting tensor cones. A conic lift alone does not supply these properties or identify the original and lifted maximal tensor products. In particular, maximal tensor products of positive semidefinite cones are generally larger than the ordinary positive semidefinite cone on the tensor-product space. More generally, even when $K_A$ and $K_B$ admit finite semidefinite representations, their maximal tensor product $K_A\tmax K_B$ need not admit one, as follows by duality from \cite{Fawzi2021}.

\section{Application: non-local games over general probabilistic theories}\label{sec:games}

In this section we work out the application announced in the introduction. We recall how two-player non-local games are played with the resources of a general probabilistic theory, introduce a natural relaxation of their optimal value in which Alice's state and measurements are merged into an \emph{assemblage} on Bob's system, and show that this relaxation is an instance of the polynomial optimization problem \autoref{eq:final_optimization_for_linear_constraints} with local linear constraints (\autoref{prop:game_as_47}). The hierarchy of \autoref{thm:convergence_hierarchy} and the rounding scheme of \autoref{thm:inner_bounds} therefore apply to it verbatim; \autoref{sec:numerics} exercises both on concrete examples.

Throughout this section, $K$, $K_A$ and $K_B$ denote state spaces, for $d\in\mathbb{N}$ we abbreviate $[d]\coloneqq\lbrace 1,2,\ldots,d\rbrace$, and $\Delta_k$ denotes the simplex of probability distributions on $k$ outcomes.

\subsection{Games, strategies and the GPT value}\label{sec:games_setup}

Fix finite nonempty question sets $\Xcal$, $\Ycal$ and finite nonempty answer sets $\Acal$, $\Bcal$. A \emph{two-player non-local game} $G=(\pi,V)$ is specified by a probability distribution $\pi$ on $\Xcal\times\Ycal$ and a verification predicate
\begin{align}
    V\,:\, \Acal\times \Bcal\times \Xcal\times \Ycal \rightarrow \lbrace 0,1\rbrace.
\end{align}
In a single round, the referee samples $(x,y)\sim \pi$, sends $x$ to Alice and $y$ to Bob, receives answers $a\in \Acal$ and $b\in \Bcal$, and declares the players to have won if and only if $V(a,b,x,y)=1$. The players may agree on a strategy beforehand, but do not communicate during the round. We collect the data of the game in the coefficients
\begin{align}\label{eqn:games_coefficients}
    g_{a,b,x,y}\coloneqq \pi(x,y)\,V(a,b,x,y)\in[0,1],\qquad a\in\Acal,\, b\in\Bcal,\, x\in\Xcal,\, y\in\Ycal.
\end{align}

\paragraph{Multimeters.} A strategy requires one complete measurement per question. Following \cite[Def.\ 3.10]{achenbach2025factorizationmultimetersunifiedview}, we describe the collection of all measurements of one player by a single channel, a \emph{multimeter}. For $k,g\in\mathbb{N}$, let $CS^1_{k,g}$ denote the \emph{polysimplex} of column-stochastic $k\times g$ matrices, i.e.\ of nonnegative matrices whose columns sum to one, so that $CS^1_{k,g}\cong \lrbracket{\Delta_k}^{\times g}$ and $CS^1_{k,1}\cong\Delta_k$. It is a state space: its linear span
\begin{align}
\begin{split}
    CS_{k,g}\coloneqq \operatorname{span}\lrbracket{CS^1_{k,g}}\\
    =\lrbrace{M\in\mathbb{R}^{k\times g}\,:\, \text{all columns of } M \text{ have the same sum}}
\end{split}
\end{align}
has dimension $kg-g+1$, and $CS^1_{k,g}$ is the base of the proper cone $CS^+_{k,g}\coloneqq\operatorname{cone}\lrbracket{CS^1_{k,g}}$ of nonnegative matrices in $CS_{k,g}$, cut out by the order unit
\begin{align}\label{eqn:games_polysimplex_unit}
    1_{CS^1_{k,g}}\,:\, CS_{k,g}\rightarrow\mathbb{R},\qquad M\mapsto \frac{1}{g}\left\langle J_{k,g}, M\right\rangle_{\operatorname{HS}},
\end{align}
where $J_{k,g}$ is the $k\times g$ matrix of all ones and $\langle\cdot,\cdot\rangle_{\operatorname{HS}}$ the Hilbert-Schmidt inner product \cite[Sec.\ 3.4]{achenbach2025factorizationmultimetersunifiedview}. For $i\in[k]$ and $j\in[g]$, the \emph{entry-readout effects} $\epsilon^{(j)}_i\in E\lrbracket{CS^1_{k,g}}$, $\epsilon^{(j)}_i(M)\coloneqq M_{i,j}$ (written $m^{(j)}_i$ in \cite{achenbach2025factorizationmultimetersunifiedview}), satisfy
\begin{align}\label{eqn:games_readout_relation}
    \sum_{i=1}^k \epsilon^{(j)}_i = 1_{CS^1_{k,g}}\qquad \text{for every } j\in[g],
\end{align}
because on $CS_{k,g}$ every column sum equals the mean of the column sums. The readout effects generate the dual cone $\lrbracket{CS^+_{k,g}}^\star$, of which they are the extreme rays \cite[Lem.\ 1]{Jenov2018}, \cite[Sec.\ 3.4]{achenbach2025factorizationmultimetersunifiedview}.\footnote{Indeed, $CS^+_{k,g}$ is the intersection of the orthant $\mathbb{R}^{k\times g}_{\geq 0}$ with the subspace $CS_{k,g}$, which contains the interior point $J_{k,g}$ of the orthant; hence the dual cone of $CS^+_{k,g}$ consists of the restrictions to $CS_{k,g}$ of the functionals in the dual cone of the orthant, which is generated by the entry functionals.} In particular they span $CS_{k,g}^\star$, but for $g\geq 2$ they are linearly dependent, the relations being precisely the equalities \autoref{eqn:games_readout_relation} between the column-sum functionals. (The bases of \cite[Eqs.\ 3.22--3.23]{achenbach2025factorizationmultimetersunifiedview} instead drop one outcome per question and adjoin the order unit; we do not need them.) To lighten the notation we write
\begin{align}
    CS_{\Acal\Xcal}\coloneqq CS_{\lrvert{\Acal},\lrvert{\Xcal}},\qquad CS_{\Bcal\Ycal}\coloneqq CS_{\lrvert{\Bcal},\lrvert{\Ycal}},
\end{align}
and likewise for the cones $CS^+$, the polysimplices $CS^1$ and their order units; after fixing enumerations of the question and answer sets, we index the rows of $CS_{\Acal\Xcal}$ by $\Acal$ and its columns by $\Xcal$, and write $\epsilon^{(x)}_a$ for the readout effects (similarly $\epsilon^{(y)}_b$ on $CS_{\Bcal\Ycal}$).

A \emph{multimeter} on $K_A$ with $\lrvert{\Xcal}$ measurements of $\lrvert{\Acal}$ outcomes each is a channel
\begin{align}
    \mathcal{M}_A\,:\, K_A\rightarrow CS^1_{\Acal\Xcal}.
\end{align}
Its measurements are read off through the readout effects: the effects $m_{a\vert x}\coloneqq \epsilon^{(x)}_a\circ\mathcal{M}_A\in E(K_A)$ satisfy $\sum_{a\in\Acal}m_{a\vert x}=1_{K_A}$ for every $x\in\Xcal$ by \autoref{eqn:games_readout_relation}, so that $\mathcal{M}^{x}_A\coloneqq\lbrace m_{a\vert x}\rbrace_{a\in\Acal}$ is a complete measurement for each question $x$. Conversely, every family of complete measurements $\lbrace m_{a\vert x}\rbrace_{a\in\Acal,\,x\in\Xcal}\subseteq E(K_A)$ defines a multimeter through $\mathcal{M}_A(z)\coloneqq\lrbracket{m_{a\vert x}(z)}_{a,x}$. At the level of cones, $\mathcal{M}_A$ extends uniquely to a positive linear map $A(K_A)^{\star+}\rightarrow CS^+_{\Acal\Xcal}$ (cf.\ \autoref{sec:preliminaries_on_GPTs}), which, as in \autoref{eqn:general_Choi}, we identify with a tensor in $A(K_A)^+\tmax CS^+_{\Acal\Xcal}$. Bob's multimeter $\mathcal{N}_B\,:\, K_B\rightarrow CS^1_{\Bcal\Ycal}$, with measurements $\mathcal{N}^{y}_B=\lbrace n_{b\vert y}\rbrace_{b\in\Bcal}$, $n_{b\vert y}\coloneqq\epsilon^{(y)}_b\circ\mathcal{N}_B$, is described in the same way.

\paragraph{Strategies and the GPT value.} A \emph{strategy} for $G$ in the GPT setting consists of a shared state $w_{AB}\in K_A\tmax K_B$ together with multimeters $\mathcal{M}_A$ for Alice and $\mathcal{N}_B$ for Bob. Since the tensor product of two positive maps is positive with respect to the maximal tensor products,\footnote{For $f\in\lrbracket{CS^+_{\Acal\Xcal}}^\star$ and $g\in\lrbracket{CS^+_{\Bcal\Ycal}}^\star$ the functionals $f\circ\mathcal{M}_A$ and $g\circ\mathcal{N}_B$ are positive, so $(f\otimes g)\lrrec{\lrbracket{\mathcal{M}_A\otimes\mathcal{N}_B}(w_{AB})}\geq 0$ by definition of the maximal tensor product; the normalization follows from $1_{CS^1_{\Acal\Xcal}}\circ\mathcal{M}_A=1_{K_A}$, which is \autoref{eqn:games_readout_relation}. Compare \cite[Prop.\ 6.23]{Pl_vala_2023} for measurements.}
\begin{align}\label{eqn:games_box}
    \lrbracket{\mathcal{M}_A\otimes\mathcal{N}_B}\lrbracket{w_{AB}}\in CS^1_{\Acal\Xcal}\tmax CS^1_{\Bcal\Ycal},
\end{align}
and reading off the entries of this tensor yields the \emph{correlation} of the strategy,
\begin{align}\label{eqn:games_correlation}
\begin{split}
p(a,b\vert x,y)&\coloneqq \lrbracket{m_{a\vert x}\otimes n_{b\vert y}}\lrbracket{w_{AB}}\\
&=\lrbracket{\epsilon^{(x)}_a\otimes\epsilon^{(y)}_b}\lrrec{\lrbracket{\mathcal{M}_A\otimes\mathcal{N}_B}\lrbracket{w_{AB}}},
\end{split}
\end{align}
which for every pair of questions $(x,y)$ is a probability distribution on $\Acal\times\Bcal$, since $p(a,b\vert x,y)\geq 0$ by \autoref{eqn:games_box} and $\sum_{a,b}p(a,b\vert x,y)=\lrbracket{1_{K_A}\otimes 1_{K_B}}(w_{AB})=1$, and which is no-signalling by \autoref{eqn:games_readout_relation}. The \emph{value} of the strategy is its winning probability
\begin{align}\label{eqn:games_value}
    \operatorname{Val}\lrbracket{w_{AB}, \mathcal{M}_A, \mathcal{N}_B}\coloneqq\sum_{x,y}\sum_{a,b}g_{a,b,x,y}\, p(a,b\vert x,y),
\end{align}
and the \emph{GPT value} of the game is
\begin{align}\label{eqn:games_beta}
    \beta^*(G)\coloneqq\sup\lbrace \operatorname{Val}\lrbracket{w_{AB}, \mathcal{M}_A, \mathcal{N}_B}\, : \,  w_{AB}\in K_A\tmax K_B,\, \mathcal{M}_A,\, \mathcal{N}_B\rbrace,
\end{align}
the supremum over all shared states in the maximal tensor product and all local multimeters.

\subsection{The assemblage relaxation}\label{sec:games_assemblage}

The value $\beta^*(G)$ is a supremum over three objects, and the objective is linear in each of them. We now relax it by merging Alice's state and multimeter into a single object living on Bob's system. By \cite[Def.\ 5.4]{achenbach2025factorizationmultimetersunifiedview}, an \emph{assemblage} on $K$ with inputs $\Xcal$ and outputs $\Acal$ consists of states $\rho_{a\vert x}\in K$ and conditional probability distributions $\lrbracket{q_{a\vert x}}_{a\in \Acal}$, $x\in\Xcal$, whose averages do not depend on the input,
\begin{align}
    \Bar{\rho}=\sum_{a\in \Acal} q_{a\vert x}\, \rho_{a\vert x}\quad \text{for all } x\in \Xcal,
\end{align}
for some \emph{averaged state} $\Bar{\rho}\in K$. Equivalently, an assemblage is a family of subnormalized states
\begin{align}
    \tilde{\rho}_{a\vert x} \coloneqq  q_{a\vert x}\rho_{a\vert x}\in A(K)^{\star +}, \qquad q_{a\vert x} =  1_{K}\lrbracket{\tilde{\rho}_{a\vert x}}\in [0,1],
\end{align}
whose sums $\sum_{a}\tilde{\rho}_{a\vert x}=\Bar\rho\in K$ do not depend on $x$; we mostly use this second description, cf.\ \cite[Sec.\ 5]{achenbach2025factorizationmultimetersunifiedview}. By \cite[Thm.\ 4]{Jenov2018}, the assemblages on $K$ are in one-to-one correspondence with the points of the state space $CS^1_{\Acal\Xcal}\tmax K$; we re-derive this correspondence in \autoref{lemma:games_geometry}(i) below in the form in which we need it. Conditions under which an assemblage admits a local hidden-state model are discussed in \cite{achenbach2025factorizationmultimetersunifiedview}.

For a game $G$ we define the \emph{assemblage value}
\begin{align}\label{eqn:game_assemblage_formulation}
\begin{split}
    \gamma^*(G) \coloneqq \sup_{\tilde{\rho}_{a\vert x},\, n_{b\vert y}}\quad &\sum_{a,b,x,y} g_{a,b,x,y}\, n_{b\vert y}\lrbracket{\tilde{\rho}_{a\vert x}}\\
    \operatorname{s.t.}\quad & \tilde{\rho}_{a\vert x}\in A(K_B)^{\star +},\quad \forall a\in\Acal,\, x\in\Xcal,\\
    & \sum_{a\in \Acal}\tilde{\rho}_{a\vert x}= \sum_{a\in \Acal}\tilde{\rho}_{a\vert x'}\in K_B,\quad \forall x,x'\in \Xcal,\\
    & n_{b\vert y}\in E(K_B),\quad \forall b\in \Bcal,\, y\in \Ycal,\\
    & \sum_{b\in \Bcal}n_{b\vert y}=1_{K_B},\quad \forall y\in \Ycal,
\end{split}
\end{align}
in which Alice is represented by an assemblage on $K_B$ and Bob by a multimeter on $K_B$; the state space $K_A$ no longer appears.

\begin{proposition}[The assemblage value upper-bounds the GPT value]\label{prop:GPT_Assemblage_Game_separation}
Let $K_A$ and $K_B$ be state spaces, let $\Xcal, \Ycal, \Acal, \Bcal$ be finite and nonempty, and let $G=(\pi,V)$ be a two-player non-local game. Then\footnote{Note that $\beta^*(G)$ depends on both $K_A$ and $K_B$, whereas $\gamma^*(G)$ depends on $K_B$ alone.}
\begin{align}
    \beta^*(G)\leq \gamma^*(G).
\end{align}
\end{proposition}
\begin{proof}
Let $\lrbracket{w_{AB}, \mathcal{M}_A, \mathcal{N}_B}$ be a strategy. Using the positive map $\psi_{w_{AB}}$ of \autoref{eqn:general_Choi}, define the \emph{induced assemblage} (cf.\ \cite{Jenov2018,Jenov2022,achenbach2025factorizationmultimetersunifiedview})
\begin{align}\label{eqn:subnormalized_assemblages}
    \tilde{\rho}_{a\vert x} \coloneqq  \psi_{w_{AB}}\lrbracket{m_{a\vert x}}=\lrbracket{m_{a\vert x}\otimes \operatorname{id}_{A(K_B)^\star}}\lrbracket{w_{AB}}\in A(K_B)^{\star},
\end{align}
for $a\in\Acal$ and $x\in\Xcal$, i.e.\ $\tilde{\rho}_{a\vert x}$ is the element characterized by $f\lrbracket{\tilde{\rho}_{a\vert x}}=\lrbracket{m_{a\vert x}\otimes f}(w_{AB})$ for all $f\in A(K_B)$. Since $m_{a\vert x}\geq 0$ and $w_{AB}\in K_A\tmax K_B$, the definition of the maximal tensor product gives $f\lrbracket{\tilde{\rho}_{a\vert x}}\geq 0$ for all $f\in A(K_B)^+$, i.e.\ $\tilde{\rho}_{a\vert x}\in A(K_B)^{\star +}$. Moreover, for every $x\in\Xcal$,
\begin{align}
    \sum_{a\in\Acal}\tilde{\rho}_{a\vert x}=\lrbracket{1_{K_A}\otimes \operatorname{id}_{A(K_B)^\star}}\lrbracket{w_{AB}}=w_B\in K_B
\end{align}
is the reduced state of $w_{AB}$ on $K_B$, which does not depend on $x$. Hence the family $\lbrace\tilde{\rho}_{a\vert x}\rbrace_{a,x}$ together with Bob's effects $\lbrace n_{b\vert y}\rbrace_{b,y}$ is feasible for \autoref{eqn:game_assemblage_formulation}. Finally, the correlation of the strategy can be written as
\begin{align}\label{eqn:games_correlation_assemblage}
    p(a,b\vert x,y)=\lrbracket{m_{a\vert x}\otimes n_{b\vert y}}\lrbracket{w_{AB}}= n_{b\vert y}\lrbracket{\tilde{\rho}_{a\vert x}},
\end{align}
so that the objective of \autoref{eqn:game_assemblage_formulation} at this feasible point equals $\operatorname{Val}\lrbracket{w_{AB}, \mathcal{M}_A, \mathcal{N}_B}$. Taking the supremum over all strategies proves the claim.
\end{proof}

\subsection{The assemblage value as a polynomial optimization problem with local constraints}\label{sec:games_reformulation}

We now show that $\gamma^*(G)$ is an instance of \autoref{eq:final_optimization_for_linear_constraints}. The objective of \autoref{eqn:game_assemblage_formulation} is bilinear in the pair (assemblage, multimeter), and every constraint concerns only one of the two; what remains is to write both objects as points of suitable state spaces and to linearize the objective and the constraints in the sense of \autoref{sec:affine_maps}. Two ingredients accomplish this. First, we encode the game in the functional
\begin{align}\label{eqn:games_G_functional}
    G_{\Acal\Bcal\Xcal\Ycal}\coloneqq \sum_{a,b,x,y}g_{a,b,x,y}\,\epsilon_{a}^{(x)}\otimes \epsilon_{b}^{(y)}\in \,\lrbracket{CS^{+}_{\Acal\Xcal}}^\star \tmin \lrbracket{CS^{+}_{\Bcal\Ycal}}^\star,
\end{align}
so that $\operatorname{Val}\lrbracket{w_{AB}, \mathcal{M}_A, \mathcal{N}_B}=G_{\Acal\Bcal\Xcal\Ycal}\lrrec{\lrbracket{\mathcal{M}_A\otimes\mathcal{N}_B}\lrbracket{w_{AB}}}$ by \autoref{eqn:games_correlation}. Second, let $d\coloneqq\dim A(K_B)^{\star}$, let $\lbrace e_i \rbrace_{i\in [d]}$ be a basis of $A(K_B)^{\star}$ and $\lbrace a_i \rbrace_{i\in [d]}\subseteq A(K_B)$ the dual basis, and consider the tensor
\begin{align}\label{eqn:games_choi}
    \chi_{A(K_B)^{\star}}\coloneqq\sum_{i=1}^d a_i\otimes e_i \in A(K_B)\otimes A(K_B)^{\star},
\end{align}
which represents the identity map of $A(K_B)^\star$ and is independent of the choice of basis (cf.\ \autoref{eqn:Choi_identity}, \cite[Lem.\ A.1]{Jenov2018}, \cite[Sec.\ 3.2]{achenbach2025factorizationmultimetersunifiedview}); explicitly, $\sum_{i=1}^d a_i(\phi)\, e_i=\phi$ for every $\phi\in A(K_B)^\star$. For $\xi\in CS_{\Acal\Xcal}\otimes A(K_B)^\star$ and $N\in A(K_B)\otimes CS_{\Bcal\Ycal}$ we let the two tensors act on the two factors of $\chi_{A(K_B)^{\star}}$ through their curryings $a\mapsto\lrbracket{\operatorname{id}\otimes a}(\xi)\in CS_{\Acal\Xcal}$, $a\in A(K_B)$, and $\phi\mapsto\lrbracket{\phi\otimes\operatorname{id}}(N)\in CS_{\Bcal\Ycal}$, $\phi\in A(K_B)^\star$, which are the only type-correct ones (cf.\ \autoref{eqn:general_Choi}):
\begin{align}\label{eqn:games_contraction}
\begin{split}
\lrbracket{\xi\otimes N}\lrbracket{\chi_{A(K_B)^{\star}}}&\coloneqq\sum_{i=1}^d\lrbracket{\operatorname{id}\otimes a_i}\lrbracket{\xi}\otimes\lrbracket{e_i\otimes\operatorname{id}}\lrbracket{N}\\
&\in CS_{\Acal\Xcal}\otimes CS_{\Bcal\Ycal},
\end{split}
\end{align}
with Alice's slot first, matching \autoref{eqn:games_G_functional}. The contraction is bilinear in $(\xi,N)$ and pairs the $A(K_B)^\star$-factor of $\xi$ with the $A(K_B)$-factor of $N$: writing
\begin{align}\label{eqn:games_entries}
\begin{split}
    \tilde{\rho}_{a\vert x}&\coloneqq\lrbracket{\epsilon^{(x)}_a\otimes\operatorname{id}}(\xi)\in A(K_B)^\star,\\
    n_{b\vert y}&\coloneqq\lrbracket{\operatorname{id}\otimes\epsilon^{(y)}_b}(N)\in A(K_B)
\end{split}
\end{align}
for the entries of the two tensors, one finds
\begin{align}\label{eqn:games_contraction_entries}
\begin{split}
&\lrbracket{\epsilon^{(x)}_a\otimes\epsilon^{(y)}_b}\lrrec{\lrbracket{\xi\otimes N}\lrbracket{\chi_{A(K_B)^{\star}}}}\\
&\qquad=\sum_{i=1}^d a_i\lrbracket{\tilde{\rho}_{a\vert x}}\, n_{b\vert y}\lrbracket{e_i}=n_{b\vert y}\lrbracket{\tilde{\rho}_{a\vert x}}.
\end{split}
\end{align}
With this notation, consider the optimization problem
\begin{align}\label{eqn:game_as_polynomial_optimization_1}
\begin{split}
    \sup_{\xi_B,\, N_B} \quad & G_{\Acal\Bcal\Xcal\Ycal}\left[{\lrbracket{\xi_B\otimes N_B}}\lrbracket{\chi_{A(K_B)^{\star}}}\right]\\
    \operatorname{s.t.} \quad & \xi_B \in CS^1_{\Acal\Xcal}\tmax K_B,\\
    & N_B \in A(K_B)^+\tmax CS^+_{\Bcal\Ycal},\\
    & \lrbracket{\operatorname{id}_{A(K_B)}\otimes 1_{CS^1_{\Bcal\Ycal}}}\lrbracket{N_B}=1_{K_B}.
\end{split}
\end{align}
Here $\xi_B$ is an assemblage in the tensor form of \cite[Thm.\ 4]{Jenov2018} and $N_B$ is the tensor of Bob's multimeter; \autoref{prop:game_as_47} below shows that the optimal value of \autoref{eqn:game_as_polynomial_optimization_1} is $\gamma^*(G)$ and that it is attained.

Two features of \autoref{eqn:game_as_polynomial_optimization_1} deserve comment before we cast it in the form of \autoref{eq:final_optimization_for_linear_constraints}. First, the objective is bilinear in the pair $\lrbracket{\xi_B, N_B}$ and the supremum runs over \emph{pairs}, i.e.\ over the product elements $\xi_B\otimes N_B$ of a minimal tensor product, exactly as in \autoref{eq:final_optimization_for_linear_constraints}; the maximal tensor products in the constraints are internal to the two local bodies and do not couple the two optimization variables. Second, all constraints are local, but they are of two different kinds. The variable $\xi_B$ ranges over the state space $CS^1_{\Acal\Xcal}\tmax K_B$, whose ambient space $CS_{\Acal\Xcal}\otimes A(K_B)^\star$ already encodes the no-signalling condition of \autoref{eqn:game_assemblage_formulation}: by \autoref{eqn:games_readout_relation}, the marginal $\sum_a\tilde\rho_{a\vert x}=\lrbracket{1_{CS^1_{\Acal\Xcal}}\otimes\operatorname{id}}(\xi_B)$ is independent of $x$ by linearity alone. The normalization of the multimeter, in contrast, is an \emph{affine} condition with non-zero right-hand side. Its solution set is not a linear subspace, and the set of complete multimeters is a compact affine slice of the proper cone $A(K_B)^+\tmax CS^+_{\Bcal\Ycal}$, but for non-singleton $K_B$ not a base of it.

The hierarchy of \autoref{sec:convergent_hierarchy_for_polynomial_optimization} uses a state space together with a linear map $F_B$ vanishing on the feasible slice. To retain the ambient multimeter cone, we first \emph{enlarge} the set of complete multimeters to a compact base $\widehat K_{\mathrm{mm}}$ of that cone, and then \emph{cut it back} to the complete multimeters by a linear constraint. On the base, the constant $1$ is the value of the order unit, so the affine normalization becomes linear. This is the homogenization of \autoref{sec:affine_maps} in concrete form.

We make this construction precise. Write
\begin{align}\label{eqn:games_bodies}
\begin{split}
 C_{\mathrm{asm}}&\coloneqq CS^+_{\Acal\Xcal}\tmax A(K_B)^{\star+},\\
 K_{\mathrm{asm}}&\coloneqq CS^1_{\Acal\Xcal}\tmax K_B,\\
 C_{\mathrm{mm}}&\coloneqq A(K_B)^+\tmax CS^+_{\Bcal\Ycal}.
\end{split}
\end{align}
For $N\in A(K_B)\otimes CS_{\Bcal\Ycal}$, let
\begin{align}\label{eqn:games_common_sum}
 s(N)\coloneqq
 \bigl(\operatorname{id}_{A(K_B)}\otimes1_{CS^1_{\Bcal\Ycal}}\bigr)(N)
 =\sum_b n_{b\mid y},\qquad y\in\Ycal.
\end{align}
The last equality follows from \autoref{eqn:games_readout_relation}; in particular, the sum is independent of $y$. The complete multimeters form the slice
\begin{align}\label{eqn:games_complete_slice}
 K_{\mathrm{mm}}\coloneqq
 \{N\in C_{\mathrm{mm}}:s(N)=1_{K_B}\}.
\end{align}
Fix $\omega\in\operatorname{relint}K_B=K_B\cap\intr A(K_B)^{\star+}$ (\autoref{lem:strictly_positive_functionals}) and define
\begin{align}\label{eqn:games_umm}
\begin{split}
 u_{\mathrm{mm}}(N)&\coloneqq s(N)(\omega)
 =\bigl(\omega\otimes1_{CS^1_{\Bcal\Ycal}}\bigr)(N),\\
 \widehat K_{\mathrm{mm}}&\coloneqq
 \{N\in C_{\mathrm{mm}}:u_{\mathrm{mm}}(N)=1\},
\end{split}
\end{align}
and the linear map
\begin{align}\label{eqn:games_FB}
 F_B(N)\coloneqq s(N)-u_{\mathrm{mm}}(N)\,1_{K_B}.
\end{align}
On $\widehat K_{\mathrm{mm}}$, the equality $F_B(N)=0$ is exactly the completeness condition $s(N)=1_{K_B}$.

The enlarged body $\widehat K_{\mathrm{mm}}$ is a known object. The set $\lbrace f\in A(K_B)^+\,:\,f(\omega)=1\rbrace$ is again a state space, with cone $A(K_B)^+$ and order unit $\omega$ \cite[Lem.\ A.1]{Jenov2018}, \cite[Eq.\ (3.9)]{achenbach2025factorizationmultimetersunifiedview}, and $\widehat K_{\mathrm{mm}}$ is the maximal tensor product of this dual state space with the polysimplex $CS^1_{\Bcal\Ycal}$ (see the proof of \autoref{lemma:games_geometry}(ii)). Under the correspondence \autoref{eqn:general_Choi} between tensors and positive maps (cf.\ \cite[Lem.\ A.1]{Jenov2018}, \cite[Lem.\ 3.4]{achenbach2025factorizationmultimetersunifiedview}), the points of $\widehat K_{\mathrm{mm}}$ correspond to the positive maps $A(K_B)^{\star+}\rightarrow CS^+_{\Bcal\Ycal}$ that send $\omega$ to a column-stochastic matrix, and the complete multimeters to those that send every state of $K_B$ to a column-stochastic matrix. The same pair of state spaces appears in \cite[Cor.\ 4.7]{achenbach2025factorizationmultimetersunifiedview}, where compatibility of a multimeter is characterized by membership of its tensor in their minimal tensor product.

\Needspace{8\baselineskip}
\begin{lemma}[State spaces for assemblages and multimeters]\label{lemma:games_geometry}
With the notation above, the following hold.
\begin{enumerate}[label=(\roman*)]
    \item $C_{\mathrm{asm}}$ is a proper cone, and the state space
    \begin{align}
    \begin{split}
        K_{\mathrm{asm}}
        &=CS^1_{\Acal\Xcal}\tmax K_B
        =\{\xi_B\in C_{\mathrm{asm}}:u_{\mathrm{asm}}(\xi_B)=1\},\\
        u_{\mathrm{asm}}&\coloneqq
        1_{CS^1_{\Acal\Xcal}}\otimes1_{K_B},
    \end{split}
    \end{align}
    is its nonempty compact base with order unit $u_{\mathrm{asm}}$. The assignment
    \begin{align}
        \xi_B\longmapsto(\tilde\rho_{a\vert x})_{a,x},
        \qquad
        \tilde\rho_{a\vert x}
        =(\epsilon_a^{(x)}\otimes\operatorname{id})(\xi_B),
    \end{align}
    is an affine bijection between $K_{\mathrm{asm}}$ and the families satisfying the first two constraints of \eqref{eqn:game_assemblage_formulation}, namely the assemblages on $K_B$. The ambient space of $K_{\mathrm{asm}}$ is canonically
    \begin{align}
        A(K_{\mathrm{asm}})^\star
        \cong CS_{\Acal\Xcal}\otimes A(K_B)^\star.
    \end{align}
    \item $C_{\mathrm{mm}}$ is a proper cone, and
    \begin{align}
        u_{\mathrm{mm}}
        =\omega\otimes1_{CS^1_{\Bcal\Ycal}}
        \in\intr C_{\mathrm{mm}}^\star.
    \end{align}
    Consequently, the state space
    \begin{align}
        \widehat K_{\mathrm{mm}}
        =\{N\in C_{\mathrm{mm}}:u_{\mathrm{mm}}(N)=1\}
    \end{align}
    is a nonempty compact base of $C_{\mathrm{mm}}$, with order unit $u_{\mathrm{mm}}$. Its ambient space is canonically
    \begin{align}
        A(\widehat K_{\mathrm{mm}})^\star
        \cong A(K_B)\otimes CS_{\Bcal\Ycal}.
    \end{align}
    Moreover,
    \begin{align}
        K_{\mathrm{mm}}
        =\{N\in C_{\mathrm{mm}}:s(N)=1_{K_B}\}
        \subseteq\widehat K_{\mathrm{mm}}
    \end{align}
    is nonempty and compact. The assignment
    \begin{align}
        N\longmapsto(n_{b\vert y})_{b,y},
        \qquad
        n_{b\vert y}
        =(\operatorname{id}_{A(K_B)}\otimes\epsilon_b^{(y)})(N),
    \end{align}
    is an affine bijection between $K_{\mathrm{mm}}$ and the families satisfying the last two constraints of \eqref{eqn:game_assemblage_formulation}, namely the complete multimeters on $K_B$.
    \item The map $F_B$ is linear and expresses completeness as a local equality constraint on $\widehat K_{\mathrm{mm}}$:
    \begin{subequations}\label{eqn:games_FB_slice}
    \begin{align}
        K_{\mathrm{mm}}
        &=\widehat K_{\mathrm{mm}}\cap\ker F_B.
    \end{align}
    The corresponding cone of nonnegative multiples of complete multimeters is
    \begin{align}
        \operatorname{cone}(K_{\mathrm{mm}})
        &=C_{\mathrm{mm}}\cap\ker F_B.
    \end{align}
    \end{subequations}
\end{enumerate}
\end{lemma}

\begin{proof}
Parts (i) and (ii) are instances of standard facts on proper cones, which we collect and prove in \autoref{sec:Proofs_of_auxiliary_results}: a functional that is strictly positive on the nonzero elements of a proper cone cuts out a compact base, which is a state space with that cone and that unit effect (\autoref{lem:bases_of_proper_cones}); the functionals $1_{K_B}$ and $f\mapsto f(\omega)$ are strictly positive in this sense (\autoref{lem:strictly_positive_functionals}); and the maximal tensor product of a polysimplex with a state space is described by its entries (\autoref{lem:polysimplex_tensors}). We also use from \autoref{sec:notation_and_preliminaries} that $A(K_B)^{+}$ and $A(K_B)^{\star+}$ are proper cones, each the dual cone of the other, and that $K_B=\lbrace \phi\in A(K_B)^{\star+}\,:\,\phi(1_{K_B})=1\rbrace$.

\emph{Part (i).} By \autoref{lem:strictly_positive_functionals}(a), \autoref{lem:polysimplex_tensors} applies to
\begin{align}
    U=A(K_B)^\star,\qquad D=A(K_B)^{\star+},\qquad w=1_{K_B}.
\end{align}
In its notation, $K_D=K_B$, $CS^+_{\Acal\Xcal}\tmax D=C_{\mathrm{asm}}$ and $1_{CS^1_{\Acal\Xcal}}\otimes w=u_{\mathrm{asm}}$, and the entries of $\xi\in CS_{\Acal\Xcal}\otimes A(K_B)^\star$ are the elements $\tilde\rho_{a\vert x}$ of \eqref{eqn:games_entries}. Hence \autoref{lem:polysimplex_tensors}(c) shows that $C_{\mathrm{asm}}$ is a proper cone, that $u_{\mathrm{asm}}\in\intr C_{\mathrm{asm}}^\star$ is an order unit, and that
\begin{align}
    K_{\mathrm{asm}}=CS^1_{\Acal\Xcal}\tmax K_B=\lbrace\xi\in C_{\mathrm{asm}}\,:\,u_{\mathrm{asm}}(\xi)=1\rbrace
\end{align}
is a nonempty compact base of $C_{\mathrm{asm}}$ and a state space whose ambient space is canonically $CS_{\Acal\Xcal}\otimes A(K_B)^\star$, with cone $C_{\mathrm{asm}}$ and unit effect $u_{\mathrm{asm}}$.

By \autoref{lem:polysimplex_tensors}(d), $\xi_B\mapsto(\tilde\rho_{a\vert x})_{a,x}$ is an affine bijection, with affine inverse, from $K_{\mathrm{asm}}$ onto the set of all families $(\tilde\rho_{a\vert x})_{a,x}$ with
\begin{align}
\begin{split}
    &\tilde\rho_{a\vert x}\in A(K_B)^{\star+},\quad \forall a\in\Acal,\, x\in\Xcal,\\
    &\sum_{a\in\Acal}\tilde\rho_{a\vert x}=\sum_{a\in\Acal}\tilde\rho_{a\vert x'}\in K_B,\quad \forall x,x'\in\Xcal.
\end{split}
\end{align}
These are precisely the first two constraints of \eqref{eqn:game_assemblage_formulation}, so these families are the assemblages on $K_B$ in the second description of \autoref{sec:games_assemblage}; the bounds $1_{K_B}(\tilde\rho_{a\vert x})\in[0,1]$ required there follow from $\tilde\rho_{a\vert x}\in A(K_B)^{\star+}$ and $\sum_{a}1_{K_B}(\tilde\rho_{a\vert x})=1$.

\emph{Part (ii).} By \autoref{lem:strictly_positive_functionals}(b), \autoref{lem:polysimplex_tensors}, with the tensor factors in the opposite order, applies to
\begin{align}
    U=A(K_B),\qquad D=A(K_B)^{+},\qquad w=\omega,
\end{align}
where $\omega$ is regarded as the evaluation functional $f\mapsto f(\omega)$ on $A(K_B)$. In its notation, $K_D=\lbrace f\in A(K_B)^+\,:\,f(\omega)=1\rbrace$, $D\tmax CS^+_{\Bcal\Ycal}=C_{\mathrm{mm}}$ and $w\otimes 1_{CS^1_{\Bcal\Ycal}}=u_{\mathrm{mm}}$ by \eqref{eqn:games_umm}, and the entries of $N\in A(K_B)\otimes CS_{\Bcal\Ycal}$ are the functions $n_{b\vert y}$ of \eqref{eqn:games_entries}. Hence \autoref{lem:polysimplex_tensors}(c) shows that $C_{\mathrm{mm}}$ is a proper cone, that $u_{\mathrm{mm}}$ is strictly positive on $C_{\mathrm{mm}}\setminus\lbrace 0\rbrace$ and an order unit in $\intr C_{\mathrm{mm}}^\star$, and that
\begin{align}
    \widehat K_{\mathrm{mm}}=K_D\tmax CS^1_{\Bcal\Ycal}=\lbrace N\in C_{\mathrm{mm}}\,:\,u_{\mathrm{mm}}(N)=1\rbrace
\end{align}
is a nonempty compact base of $C_{\mathrm{mm}}$ and a state space whose ambient space is canonically $A(K_B)\otimes CS_{\Bcal\Ycal}$, with cone $C_{\mathrm{mm}}$ and unit effect $u_{\mathrm{mm}}$.

We turn to $K_{\mathrm{mm}}$. By \autoref{lem:polysimplex_tensors}(a) and (d), $s(N)=\sum_{b}n_{b\vert y}$ for every $y\in\Ycal$, as recorded in \eqref{eqn:games_common_sum}, and $N\mapsto(n_{b\vert y})_{b,y}$ is an affine bijection, with affine inverse, from $\widehat K_{\mathrm{mm}}$ onto the set of all families of functions $n_{b\vert y}\in A(K_B)^+$ whose sums $\sum_{b}n_{b\vert y}$ do not depend on $y$ and take the value $1$ at $\omega$. If $N\in K_{\mathrm{mm}}$, then
\begin{align}
    u_{\mathrm{mm}}(N)=s(N)(\omega)=1_{K_B}(\omega)=1,
\end{align}
because $\omega\in K_B$. Hence $K_{\mathrm{mm}}\subseteq\widehat K_{\mathrm{mm}}$, and $K_{\mathrm{mm}}=\lbrace N\in\widehat K_{\mathrm{mm}}\,:\,s(N)=1_{K_B}\rbrace$ is convex and, being a closed subset of a compact set, compact.

Moreover, the bijection maps $K_{\mathrm{mm}}$ onto the set of all families $(n_{b\vert y})_{b,y}$ with
\begin{align}
\begin{split}
    &n_{b\vert y}\in A(K_B)^{+},\quad \forall b\in\Bcal,\, y\in\Ycal,\\
    &\sum_{b\in\Bcal}n_{b\vert y}=1_{K_B},\quad \forall y\in\Ycal.
\end{split}
\end{align}
For such a family, $1_{K_B}-n_{b\vert y}=\sum_{b'\neq b}n_{b'\vert y}\geq 0$, so that $n_{b\vert y}\in E(K_B)$. These are therefore precisely the families satisfying the last two constraints of \eqref{eqn:game_assemblage_formulation}, that is, the families of complete measurements $\lbrace n_{b\vert y}\rbrace_{b\in\Bcal}$, $y\in\Ycal$, or the multimeters $\rho\mapsto(n_{b\vert y}(\rho))_{b,y}$ on $K_B$ (\autoref{sec:games_setup}). Since $(\phi\otimes\epsilon^{(y)}_b)(N)=\phi(n_{b\vert y})$ for $\phi\in A(K_B)^\star$, the tensor $N$ represents this multimeter in the sense of \eqref{eqn:general_Choi}. In particular, $K_{\mathrm{mm}}\neq\emptyset$, since $n_{b\vert y}\coloneqq\lrvert{\Bcal}^{-1}1_{K_B}$ is such a family.

\emph{Part (iii).} The maps $s$ and $u_{\mathrm{mm}}$ are linear, hence so is $F_B$. If $N\in K_{\mathrm{mm}}$, then $N\in\widehat K_{\mathrm{mm}}$ by part (ii), and
\begin{align}
    F_B(N)=1_{K_B}-u_{\mathrm{mm}}(N)\,1_{K_B}=0.
\end{align}
Conversely, if $N\in\widehat K_{\mathrm{mm}}$ and $F_B(N)=0$, then $N\in C_{\mathrm{mm}}$ and $s(N)=u_{\mathrm{mm}}(N)\,1_{K_B}=1_{K_B}$, i.e.\ $N\in K_{\mathrm{mm}}$. This proves $K_{\mathrm{mm}}=\widehat K_{\mathrm{mm}}\cap\ker F_B$.

For the second identity, note that
\begin{align}
    \operatorname{cone}(K_{\mathrm{mm}})=\lbrace tN\,:\,t\geq 0,\ N\in K_{\mathrm{mm}}\rbrace,
\end{align}
because $K_{\mathrm{mm}}$ is nonempty and convex, as shown in the proof of part (ii). Every element $tN$ of this set lies in the cone $C_{\mathrm{mm}}$ and satisfies $F_B(tN)=t\,F_B(N)=0$. Conversely, let $N\in C_{\mathrm{mm}}\cap\ker F_B$. If $N=0$, there is nothing to show. Otherwise $t\coloneqq u_{\mathrm{mm}}(N)>0$ by \autoref{lem:polysimplex_tensors}(c), so that $t^{-1}N\in\widehat K_{\mathrm{mm}}\cap\ker F_B=K_{\mathrm{mm}}$ and $N=t\,(t^{-1}N)\in\operatorname{cone}(K_{\mathrm{mm}})$.
\end{proof}

The bilinear contraction in \eqref{eqn:games_contraction} induces a unique linear functional $L_G$ on $A(K_{\mathrm{asm}})^\star\otimes A(\widehat K_{\mathrm{mm}})^\star$. On product tensors it is given by
\begin{align}\label{eqn:games_LG}
\begin{split}
    L_G(\xi\otimes N)
    &\coloneqq G_{\Acal\Bcal\Xcal\Ycal}
       \left[{\lrbracket{\xi\otimes N}}\lrbracket{\chi_{A(K_B)^\star}}\right]
\\
    &=\sum_{a,b,x,y}g_{a,b,x,y}\,
       n_{b\vert y}(\tilde\rho_{a\vert x}),
\end{split}
\end{align}
where the second equality is \eqref{eqn:games_contraction_entries}.

\Needspace{8\baselineskip}
\begin{proposition}[$\gamma^*(G)$ as optimization with local linear constraints]\label{prop:game_as_47}
The optimal value of \eqref{eqn:game_as_polynomial_optimization_1} equals $\gamma^*(G)$ and is attained. Moreover,
\begin{align}\label{eqn:game_as_47}
\begin{split}
    -\gamma^*(G)=\min_{\xi_B,\,N_B}\quad
        &-L_G(\xi_B\otimes N_B)\\
    \operatorname{s.t.}\quad
        &\xi_B\in K_{\mathrm{asm}},\quad
          N_B\in\widehat K_{\mathrm{mm}},\\
        &F_B(N_B)=0.
\end{split}
\end{align}
This is an instance of \eqref{eq:final_optimization_for_linear_constraints} with local state spaces $K_{\mathrm{asm}}$ and $\widehat K_{\mathrm{mm}}$, objective $P=-L_G$, the linear equality map $F_B$ from \eqref{eqn:games_FB}, and $F_A=G_A=G_B=0$.
\end{proposition}

\begin{proof}
Throughout, $\tilde\rho_{a\vert x}$ and $n_{b\vert y}$ denote the entries \eqref{eqn:games_entries} of $\xi_B$ and $N_B$.

\emph{Feasible sets.} By \eqref{eqn:games_bodies}, \eqref{eqn:games_common_sum} and \eqref{eqn:games_complete_slice}, the constraints of \eqref{eqn:game_as_polynomial_optimization_1} say that $\xi_B\in K_{\mathrm{asm}}$ and $N_B\in K_{\mathrm{mm}}$. By \autoref{lemma:games_geometry}(i) and (ii), the assignment
\begin{align}
    (\xi_B,N_B)\longmapsto\bigl((\tilde\rho_{a\vert x})_{a,x},\,(n_{b\vert y})_{b,y}\bigr)
\end{align}
is therefore a bijection from the feasible set $K_{\mathrm{asm}}\times K_{\mathrm{mm}}$ of \eqref{eqn:game_as_polynomial_optimization_1} onto the set of pairs of families satisfying the four constraints of \eqref{eqn:game_assemblage_formulation}, that is, onto the feasible set of \eqref{eqn:game_assemblage_formulation}.

\emph{Objective and optimal value.} By \eqref{eqn:games_LG},
\begin{align}
\begin{split}
    G_{\Acal\Bcal\Xcal\Ycal}\bigl[(\xi_B\otimes N_B)(\chi_{A(K_B)^\star})\bigr]
    &=L_G(\xi_B\otimes N_B)\\
    &=\sum_{a,b,x,y}g_{a,b,x,y}\,n_{b\vert y}(\tilde\rho_{a\vert x}).
\end{split}
\end{align}
The left-hand side is the objective of \eqref{eqn:game_as_polynomial_optimization_1} at $(\xi_B,N_B)$, and the right-hand side is the objective of \eqref{eqn:game_assemblage_formulation} at the corresponding pair of families. Hence the two problems take the same objective values at corresponding feasible points, and the optimal value of \eqref{eqn:game_as_polynomial_optimization_1} is $\gamma^*(G)$.

\emph{Attainment.} The sets $K_{\mathrm{asm}}$ and $K_{\mathrm{mm}}$ are nonempty and compact by \autoref{lemma:games_geometry}(i) and (ii), and the map $(\xi_B,N_B)\mapsto L_G(\xi_B\otimes N_B)$ is bilinear on finite-dimensional spaces and hence continuous. Therefore the supremum in \eqref{eqn:game_as_polynomial_optimization_1} is attained,
\begin{align}
    \gamma^*(G)=\max_{(\xi_B,N_B)\in K_{\mathrm{asm}}\times K_{\mathrm{mm}}}L_G(\xi_B\otimes N_B),
\end{align}
and, by the bijection above, so is the supremum in \eqref{eqn:game_assemblage_formulation}.

\emph{Reformulation with a linear constraint.} By \autoref{lemma:games_geometry}(iii),
\begin{align}
    \lbrace N_B\in\widehat K_{\mathrm{mm}}\,:\,F_B(N_B)=0\rbrace=\widehat K_{\mathrm{mm}}\cap\ker F_B=K_{\mathrm{mm}}.
\end{align}
The feasible set of \eqref{eqn:game_as_47} is therefore again $K_{\mathrm{asm}}\times K_{\mathrm{mm}}$, and its objective is the negative of the objective of \eqref{eqn:game_as_polynomial_optimization_1}. Consequently, the minimum in \eqref{eqn:game_as_47} exists and equals
\begin{align}
    \min_{(\xi_B,N_B)\in K_{\mathrm{asm}}\times K_{\mathrm{mm}}}\bigl(-L_G(\xi_B\otimes N_B)\bigr)=-\gamma^*(G).
\end{align}

\emph{Comparison with the general problem.} It remains to match \eqref{eqn:game_as_47} with \eqref{eq:final_optimization_for_linear_constraints}, where $K_{\mathrm{asm}}$ and $\widehat K_{\mathrm{mm}}$ take the roles of the two local state spaces. By \autoref{lemma:games_geometry}(i) and (ii), both are state spaces, with ambient spaces
\begin{align}
\begin{split}
    A(K_{\mathrm{asm}})^\star&\cong CS_{\Acal\Xcal}\otimes A(K_B)^\star,\\
    A(\widehat K_{\mathrm{mm}})^\star&\cong A(K_B)\otimes CS_{\Bcal\Ycal}.
\end{split}
\end{align}
The objective $P=-L_G$ is a linear functional on $A(K_{\mathrm{asm}})^\star\otimes A(\widehat K_{\mathrm{mm}})^\star$, and the equality map $F_B$ of \eqref{eqn:games_FB} is linear on $A(\widehat K_{\mathrm{mm}})^\star$ by \autoref{lemma:games_geometry}(iii). The maps $F_A$, $G_A$ and $G_B$ are zero, so that the conditions $F_A=0$, $G_A\geq 0$ and $G_B\geq 0$ are void. Finally, the local feasible sets $K_i^{\mathrm{feas}}$ of \autoref{sec:comparison_DPS_PLG} are
\begin{align}
    K_{\mathrm{asm}}
    \qquad\text{and}\qquad
    \widehat K_{\mathrm{mm}}\cap\ker F_B=K_{\mathrm{mm}},
\end{align}
which are nonempty, as required there. Hence \eqref{eqn:game_as_47} is the problem \eqref{eq:final_optimization_for_linear_constraints} for these data, with optimal value $p^*=-\gamma^*(G)$.
\end{proof}

The optimization in \eqref{eqn:game_as_47} is over pairs, or equivalently over the corresponding product elements $\xi_B\otimes N_B$. By linearity of $L_G$, the same optimal value is obtained over convex combinations of these feasible products, with $F_B(N_B)=0$ on every product component.

When \eqref{eqn:game_as_47} is read as an instance of \eqref{eq:final_optimization_for_linear_constraints}, system $A$ has the state space $K_{\mathrm{asm}}$ and its variable $x_A=\xi_B$ is Alice's assemblage, while system $B$ has the state space $\widehat K_{\mathrm{mm}}$ and its variable $x_B=N_B$ is Bob's multimeter. The labels $A$ and $B$ thus refer to the players' optimization variables and not to the state spaces $K_A$ and $K_B$ of \autoref{sec:games_setup}: both $K_{\mathrm{asm}}$ and $\widehat K_{\mathrm{mm}}$ are built from Bob's state space $K_B$, to which the subscript of $\xi_B$ and $N_B$ refers. In particular, the hierarchy $\Sigma^n(A:B)$ extends system $B$: its extensions $y_{AB^n}\in K_{\mathrm{asm}}\tmax\widehat K_{\mathrm{mm}}^{\tmax n}$ involve one copy of the assemblage and $n$ copies of the multimeter. Alice's state space $K_A$ does not appear in \eqref{eqn:game_as_47}: $\xi_B$ ranges over all assemblages on $K_B$, whether or not they are induced by a shared state and a multimeter on $K_A$ (\autoref{sec:games_realizability}).

\Needspace{6\baselineskip}
\begin{remark}[Independence of the interior state]\label{rem:game_as_47}
The outer feasible sets $\Sigma^n(A:B)$ and optimal level values are independent of the choice of $\omega$. Indeed, the subspace $H\coloneqq s^{-1}(\mathbb{R}\,1_{K_B})=\ker F_B$ is independent of $\omega$, and $u_{\mathrm{mm}}(N)=t$ whenever $s(N)=t1_{K_B}$. The lifted equality constraint and permutation symmetry imply that every feasible extension satisfies
\begin{align}
    y_{AB^n}\in A(K_{\mathrm{asm}})^\star\otimes H^{\otimes n}.
\end{align}
Thus all choices of $\omega$ give the same normalization and marginalization on feasible extensions, while the ambient tensor cone is unchanged. The constants in \autoref{thm:convergence_hierarchy} may still depend on $\omega$.
\end{remark}

\subsection{Assemblages versus bipartite strategies}\label{sec:games_realizability}

We close this section by returning to the relaxation in \autoref{prop:GPT_Assemblage_Game_separation}. Fix $K_A$, $K_B$ and Alice's sets $\Xcal$, $\Acal$. We call an assemblage on $K_B$ \emph{realizable} (over $K_A$) if it is the induced assemblage \autoref{eqn:subnormalized_assemblages} of some state $w_{AB}\in K_A\tmax K_B$ and some multimeter on $K_A$, and we write $\mathcal{R}$ for the convex hull of the set of realizable assemblages. In quantum theory, the Gisin-Hughston-Jozsa-Wootters (GHJW) theorem \cite{Gisin1989StochasticQD,Hughston1993} implies that every assemblage is realizable provided Alice's Hilbert space is chosen sufficiently large ($\dim\mathcal{H}_A\geq \operatorname{rank}\Bar{\rho}$ suffices; see \cite{Uola2020} for an overview). Hence, in the quantum case, $\beta^*(G)= \gamma^*(G)$ whenever $\dim\mathcal{H}_A\geq\dim\mathcal{H}_B$ \cite[Lem.\ 3.1]{zeiss2025approximatingfixedsizequantum}. In general GPTs this equivalence breaks down: not every assemblage is realizable \cite{Jenov2018,Barnum2013,Stevens2014} (cf.\ also \cite[Prop.\ 5.12]{achenbach2025factorizationmultimetersunifiedview} for an explicit example).

A failure of realizability does not, however, automatically separate the game values, for two reasons. First, a game accesses an assemblage only through the correlations $p(a,b\vert x,y)=n_{b\vert y}\lrbracket{\tilde{\rho}_{a\vert x}}$, i.e.\ through finitely many linear functionals. Second, and more basically, for a fixed multimeter of Bob the objective of \autoref{eqn:game_assemblage_formulation} is linear in the assemblage, so that $\beta^*(G)$ does not change if the realizable assemblages are replaced by their convex hull $\mathcal{R}$.\footnote{By the proof of \autoref{prop:GPT_Assemblage_Game_separation}, $\beta^*(G)$ is the supremum of the objective of \autoref{eqn:game_assemblage_formulation} over realizable assemblages and complete multimeters $\lbrace n_{b\vert y}\rbrace$, and for fixed $\lbrace n_{b\vert y}\rbrace$ the objective is linear in the assemblage.} Consequently, $\beta^*(G)=\gamma^*(G)$ for every game as soon as $\mathcal{R}$ is the set of all assemblages, even if some assemblages are not realizable.

This is the case, for instance, whenever $K_B$ is a simplex: writing $\tilde\rho_{a\vert x}=\sum_k\Bar\rho(k)\,q(a\vert x,k)\,\delta_k$ with the vertices $\delta_k$ of $K_B$ exhibits every assemblage as a convex combination of the assemblages $\lrbracket{q(a\vert x,k)\,\delta_k}_{a,x}$, each of which is realizable over any $K_A$ (take $w_{AB}=\omega_A\otimes\delta_k$ for some $\omega_A\in K_A$ and $m_{a\vert x}=q(a\vert x,k)\,1_{K_A}$). If $\Bar\rho(k)=0$, the conditional distributions $q(\cdot\vert x,k)$ may be chosen arbitrarily, since their terms have zero weight. Thus, for a simplex $K_B$, $\beta^*(G)=\gamma^*(G)$ for every game and every $K_A$ -- both equal the classical value of $G$, the maximum over local deterministic strategies -- although, e.g., for $K_A=\Delta_2$ and $K_B=\Delta_3$ not every assemblage is realizable.\footnote{For $\Xcal=\Acal=\lbrace 0,1\rbrace$, a realizable assemblage has $q(a\vert x,k)=\sum_{\lambda\in\lbrace 1,2\rbrace}w(\lambda\vert k)\,m_{a\vert x}(\lambda)$, so that the three points $\lrbracket{q(0\vert 0,k),\,q(0\vert 1,k)}\in[0,1]^2$, $k\in[3]$, lie on the segment joining $\lrbracket{m_{0\vert 0}(\lambda),\,m_{0\vert 1}(\lambda)}$, $\lambda=1,2$; an assemblage with uniform $\Bar\rho$ whose three response points are not collinear is therefore not realizable.} The natural question is therefore whether the existence of an assemblage outside $\mathcal{R}$ is always witnessed by some game. We propose the following conjecture.

\begin{conjecture}[Assemblages outperform bipartite strategies]\label{conj:assemblage_game_separation}
    Let $K_A$, $K_B$ be state spaces and $\Xcal$, $\Acal$ finite nonempty question and answer sets such that some assemblage on $K_B$ does not belong to the convex hull $\mathcal{R}$ of the assemblages realizable over $K_A$, i.e.\ of the assemblages of the form
    \begin{align}
        \tilde{\rho}_{a\vert x}=\lrbracket{m_{a\vert x}\otimes \operatorname{id}_{A(K_B)^\star}}\lrbracket{w_{AB}},\quad a\in\Acal,\, x\in\Xcal,
    \end{align}
    with a state $w_{AB}\in K_A\tmax K_B$ and a multimeter $\lbrace m_{a\vert x}\rbrace_{a\in\Acal,\, x\in\Xcal}$ on $K_A$. Then there exist a finite nonempty question set $\Ycal$, a finite nonempty answer set $\Bcal$ and a two-player non-local game $G=(\pi,V)$ with question sets $\Xcal, \Ycal$ and answer sets $\Acal, \Bcal$ such that
    \begin{align}
        \beta^*(G)<\gamma^*(G).
    \end{align}
\end{conjecture}

The converse implication holds by the second observation above: if $\mathcal{R}$ is the set of all assemblages on $K_B$, then $\beta^*(G)=\gamma^*(G)$ for every game whose question and answer sets for Alice are $\Xcal$ and $\Acal$.

\section{Numerical illustration: non-local games over polygon theories}
\label{sec:numerics}

In this section we instantiate the hierarchy of \autoref{thm:convergence_hierarchy} and the rounding scheme of \autoref{thm:inner_bounds} on the assemblage value $\gamma^*(G)$, in the form of \autoref{prop:game_as_47}, for concrete polytopic state spaces. The purpose is a proof of concept: in every instance considered, $\gamma^*(G)$ can be computed exactly by an independent, direct method (cf.\ \autoref{sec:numerics_setup}), so that the finite-level bounds produced by the two approximation schemes can be compared with the exact answer. Two features make this an attractive testbed. First, for polytopic state spaces every level of the hierarchy is an exact linear program. Second, the pair of theorems is exercised in full: the outer levels decrease towards the optimum while the rounded inner points approach it from below, so that the optimum is sandwiched by explicit bounds on both sides at every level.

Throughout this section, $\gamma^*(G)$ denotes the optimal assemblage value for the experiment under consideration; in the constrained pentagon experiment, this includes the additional local inequality \autoref{eqn:numerics_inequality}. The feasible sets $\Sigma^n(A:B)$ and the outer and inner values $p^{(n)}$ and $L^{(n)}$ likewise refer to the same experiment.

Since \autoref{eqn:game_as_47} minimizes $-L_G$, we adopt the sign convention of a maximization throughout this section and write
\begin{align}\label{eqn:numerics_pn}
    p^{(n)}\coloneqq\max\lbrace L_G\lrbracket{x_{AB}}\,:\, x_{AB}\in\Sigma^n(A:B)\rbrace
\end{align}
for the negative of the level value of \autoref{eq:n_level_relaxation} with $P=-L_G$. By \autoref{thm:convergence_hierarchy}, these level values form a non-increasing sequence of upper bounds on $\gamma^*(G)$ converging at rate $n^{-1/2}$, and, by \autoref{thm:inner_bounds}, the single-product selection following that theorem yields feasible pairs whose values are lower bounds on $\gamma^*(G)$ converging at the same rate. The implementation, an explanatory companion document detailing the computational pipeline, and a notebook reproducing all numbers, tables and figures of this section are available in the accompanying repository \cite{JuliusAZeissGithub_GPT_Games}.

\subsection{Setup}
\label{sec:numerics_setup}

We consider the CHSH game \cite{Clauser1969}, $\Acal=\Bcal=\Xcal=\Ycal=\lbrace 0,1\rbrace$, $\pi$ uniform and $V(a,b,x,y)=1$ if and only if $a\oplus b = x\wedge y$, and evaluate $\gamma^*(G)$ for three local state spaces $K_B$:
\begin{enumerate}
    \item the classical bit $K_B=\Delta_2$;
    \item the square gbit (``boxworld''), with vertices $(1,\pm 1,\pm 1)$;
    \item the regular pentagon with vertices $\lrbracket{1,\, r_5\cos(2\pi i/5),\, r_5\sin(2\pi i/5)}$, $i\in[5]$, of circumradius
    \begin{align}
        r_5=\sqrt{\sec(\pi/5)}.
    \end{align}
    At this radius the pentagon is strongly self-dual\footnote{Odd polygons are strongly self-dual, even ones only weakly so \cite{Janotta_2011}.}: the effect cone $A(K_B)^+$ coincides with the state cone, each of its extreme rays being aligned with one of the five vertices \cite{Janotta_2011}. Equivalently, the polar body of the pentagon is the same pentagon rotated by $\pi/5$.
\end{enumerate}
In the square and pentagon representations, the order unit $1_{K_B}$ is the first-coordinate functional, $1_{K_B}(t,x,y)=t$; normalized states therefore have first coordinate equal to $1$. In the notation of \autoref{lemma:games_geometry}, with $\lrvert{\Acal}=\lrvert{\Xcal}=\lrvert{\Bcal}=\lrvert{\Ycal}=2$, the two convex bodies are Alice's assemblage body and Bob's complete multimeter body,
\begin{align}
\begin{split}
    K_{\mathrm{asm}}&= CS^1_{2,2}\tmax K_B,\\
    K_{\mathrm{mm}}&= \left\lbrace N\in A(K_B)^+\tmax CS^+_{2,2}\,:\, \lrbracket{\operatorname{id}_{A(K_B)}\otimes 1_{CS^1_{2,2}}}\lrbracket{N}=1_{K_B}\right\rbrace,
\end{split}
    \label{eqn:numerics_bodies}
\end{align}
with the bilinear objective $L_G$ of \autoref{eqn:games_LG}. By \autoref{prop:game_as_47}, $\gamma^*(G)$ is an instance of \autoref{eq:final_optimization_for_linear_constraints} with the state spaces $K_{\mathrm{asm}}$ and $\widehat K_{\mathrm{mm}}\supseteq K_{\mathrm{mm}}$ and the local constraints $F_A=0$, $G_A=G_B=0$ and $F_B$ from \autoref{eqn:games_FB}: Alice's no-signalling condition requires no constraint of its own (\autoref{lemma:games_geometry}(i)), whereas the multimeter normalization is the local \emph{linear} equality constraint $F_B$, imposed on every copy of the extension via constraint (2b) in the definition of $\Sigma^n(A:B)$. As a fourth experiment we additionally impose the local inequality constraint
\begin{align}
    G_A(\xi_B)\;=\; q_{1\vert 0}(\xi_B) - 0.7\cdot u_{\mathrm{asm}}(\xi_B)\;\geq\; 0,
    \label{eqn:numerics_inequality}
\end{align}
where $q_{1\vert 0}\coloneqq\epsilon^{(0)}_{1}\otimes 1_{K_B}$ reads out the weight of answer $a=1$ to question $x=0$, requiring Alice to output $a=1$ on question $x=0$ with probability at least $0.7$; on the extension it is imposed in the conic sense of constraint (2c). This exercises precisely the feature that distinguishes \autoref{thm:convergence_hierarchy} from the polarization hierarchy (cf.\ \autoref{sec:comparison_DPS_PLG}). The remaining local constraints of \autoref{eq:optimization_for_local_constraints} are not needed: $F_A=0$ and $G_B=0$, so constraints (2a) and (2d) are vacuous.

\paragraph{Polyhedral descriptions.} The two bodies in \autoref{eqn:numerics_bodies} are polytopes for all three choices of $K_B$. Indeed, $CS^1_{2,2}$ is the Cartesian product of two line segments, and $K_B$ is a polytope in each case; hence $CS^+_{2,2}$, $A(K_B)^{\star+}=\operatorname{cone}(K_B)$ and $A(K_B)^+$ are polyhedral cones. Here
\begin{align}
    C_{\mathrm{asm}}&= CS^+_{2,2}\tmax A(K_B)^{\star+},
    &
    C_{\mathrm{mm}}&= A(K_B)^+\tmax CS^+_{2,2}
\end{align}
are the ambient cones of \autoref{lemma:games_geometry}, and $V_{\mathrm{asm}}\coloneqq\operatorname{span}\lrbracket{C_{\mathrm{asm}}}$, $V_{\mathrm{mm}}\coloneqq\operatorname{span}\lrbracket{C_{\mathrm{mm}}}$ their spans. Both cones are polyhedral: $\lrbracket{C_1\tmax C_2}^\star=C_1^\star\tmin C_2^\star$, and the minimal tensor product of polyhedral cones is generated by the finitely many tensor products of extreme rays of its factors. The same identity exhibits generating families for the two \emph{dual} cones. For $C_{\mathrm{asm}}^\star=\lrbracket{CS^+_{2,2}}^\star\tmin A(K_B)^+$ we take
\begin{align}
    \lbrace\varphi_j : j\in[r_{\mathrm{asm}}]\rbrace=\lbrace\epsilon^{(x)}_a\otimes h\rbrace,
\end{align}
where $\epsilon^{(x)}_a$ ranges over the entry-readout effects -- the extreme rays of $\lrbracket{CS^+_{2,2}}^\star$ (cf.\ \autoref{sec:games_setup}) -- and $h$ over the facet-defining functionals of $\operatorname{cone}(K_B)$, which generate $A(K_B)^+$; for $C_{\mathrm{mm}}^\star=A(K_B)^{\star+}\tmin\lrbracket{CS^+_{2,2}}^\star$ we take
\begin{align}
    \lbrace\psi_k : k\in[r_{\mathrm{mm}}]\rbrace=\lbrace v\otimes\epsilon^{(y)}_b\rbrace,
\end{align}
where $v$ ranges over the vertices of $K_B$, normalized as states, $1_{K_B}(v)=1$. The common value $r_{\mathrm{asm}}=r_{\mathrm{mm}}$ is $8$, $16$, and $20$ for the bit, the square, and the pentagon, respectively (four readout effects each; two, four and five facets, respectively vertices). The body $K_{\mathrm{asm}}$ is the base of $C_{\mathrm{asm}}$ cut out by $u_{\mathrm{asm}}=1_{CS^1_{2,2}}\otimes 1_{K_B}$ (cf.\ \autoref{lemma:games_geometry}(i)). The body $K_{\mathrm{mm}}$ is a bounded affine slice of $C_{\mathrm{mm}}$ -- bounded because $K_{\mathrm{mm}}\cong E(K_B)\times E(K_B)$, a complete multimeter with binary outcomes being determined by the pair of effects $\lrbracket{n_{0\vert 0}, n_{0\vert 1}}$ -- and is therefore also a polytope.

Finally, since $\tmax$ is associative, membership in the ambient cone of the level-$n$ extension is tested on products of the dual generators: with $i$ indexing the multimeter copy and $k_i$ its selected generator, $y\in C_{\mathrm{asm}}\tmax C_{\mathrm{mm}}^{\tmax n}$ if and only if
\begin{align}
    \lrbracket{\varphi_j\otimes\bigotimes_{i=1}^{n}\psi_{k_i}}(y)\;\geq\;0 \qquad \text{for all }j\in[r_{\mathrm{asm}}],\ (k_1,\ldots,k_n)\in[r_{\mathrm{mm}}]^n.
    \label{eqn:numerics_facets}
\end{align}

\paragraph{Ambient normalization for polytopes.} We instantiate \autoref{lemma:games_geometry} with $\omega\coloneqq\bar v$, the barycentre $\bar v\coloneqq\lrvert{\operatorname{vert}K_B}^{-1}\sum_{v}v$ of the vertices of $K_B$, which is an interior state. With $s\coloneqq\lrvert{\Ycal}\cdot\lrvert{\operatorname{vert}K_B}=2\,\lrvert{\operatorname{vert}K_B}$, the order unit \autoref{eqn:games_umm} of $\widehat K_{\mathrm{mm}}$ is then the normalized sum of all generators of $C_{\mathrm{mm}}^\star$,
\begin{align}
    u_{\mathrm{mm}}=\bar v\otimes 1_{CS^1_{2,2}}=\frac{1}{s}\sum_{k=1}^{r_{\mathrm{mm}}}\psi_k,
    \label{eqn:numerics_umm}
\end{align}
since $\sum_b\epsilon^{(y)}_b=1_{CS^1_{2,2}}$ for each of the $\lrvert{\Ycal}=2$ questions (cf.\ \autoref{eqn:games_readout_relation}); this identity is what makes the rounding scheme below explicit. In the polyhedral case the strict positivity of $u_{\mathrm{mm}}$ on $C_{\mathrm{mm}}\setminus\lbrace 0\rbrace$, from which its interiority follows as in the proof of \autoref{lemma:games_geometry}(ii), can also be read off directly: any $N\in C_{\mathrm{mm}}=\lrbracket{C_{\mathrm{mm}}^\star}^\star$ with $u_{\mathrm{mm}}(N)=0$ satisfies $\psi_k(N)=0$ for all $k$, hence $N=0$, because the $\psi_k$ span $V_{\mathrm{mm}}^\star$.

The homogenized equality constraint $F_B$ of \autoref{eqn:games_FB}, the slice identities \autoref{eqn:games_FB_slice} and the linearized objective $L_G$ of \autoref{eqn:games_LG} -- whose first tensor slot is Alice's -- are used as they stand; on feasible pairs, $L_G$ is the game value by \autoref{eqn:games_LG}. The hierarchy is thus instantiated with the state spaces $K_{\mathrm{asm}}$ and $\widehat K_{\mathrm{mm}}$, the objective $L_G$ (maximized, cf.\ \autoref{eqn:numerics_pn}), and the local constraints $F_A=0$, $G_B=0$, $F_B$ from \autoref{eqn:games_FB} and, in the constrained experiment, $G_A$ from \autoref{eqn:numerics_inequality}.

Every feasible pair of \autoref{eqn:game_as_polynomial_optimization_1} -- $\xi\in K_{\mathrm{asm}}$, $N\in K_{\mathrm{mm}}$, and $G_A(\xi)\geq 0$ in the constrained experiment -- yields the $\Sigma^n(A:B)$-feasible product extension $\xi\otimes N^{\otimes n}$, so $p^{(n)}\geq\gamma^*(G)$ at every level; conversely, the de Finetti candidate in the proof of \autoref{thm:convergence_hierarchy} is a convex mixture of feasible pairs, whose value is at most $\gamma^*(G)$. Hence
\begin{align}
    \gamma^*(G)\;\leq\;p^{(n)}\;\leq\;\gamma^*(G)+\frac{2\,\lVert L_G\rVert\, c}{\sqrt n},
    \qquad c=c\lrbracket{K_{\mathrm{asm}},\widehat K_{\mathrm{mm}}},
\end{align}
with the norm and the de Finetti constant of \autoref{thm:convergence_hierarchy}.

\paragraph{The level-$n$ outer LP.} Write $\Psi_{\mathbf{k}}\coloneqq\bigotimes_{i=1}^{n}\psi_{k_i}$ for $\mathbf{k}=(k_1,\ldots,k_n)\in[r_{\mathrm{mm}}]^n$. In the optimization variable $y\in V_{\mathrm{asm}}\otimes V_{\mathrm{mm}}^{\otimes n}$, the level-$n$ program \autoref{eq:n_level_relaxation} becomes the linear program
\begin{align}\label{eqn:numerics_level_lp}
\begin{split}
    p^{(n)}=\max_y\quad & \lrbracket{L_G\otimes u_{\mathrm{mm}}^{\otimes(n-1)}}(y)\\
    \operatorname{s.t.}\quad
    & \lrbracket{\varphi_j\otimes\Psi_{\mathbf{k}}}(y)\geq 0,\qquad j\in[r_{\mathrm{asm}}],\ \mathbf{k}\in[r_{\mathrm{mm}}]^n,\\
    & \lrbracket{u_{\mathrm{asm}}\otimes u_{\mathrm{mm}}^{\otimes n}}(y)=1,\\
    & \lrbracket{\operatorname{id}_{V_{\mathrm{asm}}}\otimes S_\sigma}(y)=y,\qquad \sigma\in S_n,\\
    & \lrbracket{\operatorname{id}_{V_{\mathrm{asm}}}\otimes F_B\otimes\operatorname{id}_{V_{\mathrm{mm}}}^{\otimes(n-1)}}(y)=0,\\
    & \lrbracket{G_A\otimes\Psi_{\mathbf{k}}}(y)\geq 0,\qquad \mathbf{k}\in[r_{\mathrm{mm}}]^n,
\end{split}
\end{align}
where the last row is present only in the constrained experiment defined by \autoref{eqn:numerics_inequality}. The constraints on $y$ are precisely the extension conditions used to define $\Sigma^n(A:B)$ in Sect.~\ref{sec:hierarchy_definition}, expressed here as finitely many linear equalities and inequalities.

The first row is the ambient cone membership \autoref{eqn:numerics_facets}; together with the second row it makes $y$ a state of $K_{\mathrm{asm}}\tmax\widehat K_{\mathrm{mm}}^{\tmax n}$. The third row characterizes the symmetric elements of the ambient cone, i.e.\ the image of $\operatorname{id}_{V_{\mathrm{asm}}}\otimes P_{\operatorname{Sym}_n(V_{\mathrm{mm}})}$, as required by condition (1); it suffices to impose it for the $n-1$ adjacent transpositions, which generate $S_n$. The fourth row is the local equality constraint (2b) on the first copy; by permutation invariance this is equivalent to imposing it on every copy. The fifth row is constraint (2c) in the conic sense,
\begin{align}
    \lrbracket{G_A\otimes\operatorname{id}_{V_{\mathrm{mm}}}^{\otimes n}}(y)\in C_{\mathrm{mm}}^{\tmax n},
\end{align}
tested against the generating family $\lbrace\Psi_{\mathbf{k}}\rbrace$; constraints (2a) and (2d) are vacuous here. Finally, the objective evaluates $L_G$ on the reduction
\begin{align}
    x_{AB}=\lrbracket{\operatorname{id}_{V_{\mathrm{asm}}}\otimes\operatorname{id}_{V_{\mathrm{mm}}}\otimes u_{\mathrm{mm}}^{\otimes(n-1)}}(y).
\end{align}
By symmetry the choice of the retained copy is irrelevant, and as $y$ ranges over the feasible set, $x_{AB}$ ranges exactly over $\Sigma^n(A:B)$.

Every level of the hierarchy is therefore an exact finite linear program; no cone approximation or discretization is involved, and the mathematical optimum of \autoref{eqn:numerics_level_lp} is exactly $p^{(n)}$, while the solver returns a floating-point approximation at its numerical stopping tolerances. As written, the formulation comprises $r_{\mathrm{asm}}\,r_{\mathrm{mm}}^n$ cone inequalities, one scalar normalization and, in the constrained experiment, $r_{\mathrm{mm}}^n$ additional inequalities, while the symmetry and $F_B$ rows are vector equalities.

In an implementation the symmetry row is used to eliminate variables, and it reduces the cone inequalities to one representative per $S_n$-orbit of the multi-indices $\mathbf{k}$, since the functionals $\Psi_{\mathbf{k}}$ are permuted among themselves by the transpose action of $\operatorname{id}_{V_{\mathrm{asm}}}\otimes S_\sigma$. All levels reported below are solved in this reduced form with the LP solver HiGHS \cite{Huangfu_2018}; levels $n\leq 5$ take between fractions of a second and a few minutes per level on a single CPU core, and the implementation is documented in the accompanying repository \cite{JuliusAZeissGithub_GPT_Games}.

\paragraph{Exact reference values.} For comparison, $\gamma^*(G)$ itself is computed directly, bypassing the hierarchy. Since $K_{\mathrm{asm}}$ and $K_{\mathrm{mm}}$ are polytopes and the objective is bilinear, the optimum is attained at a pair of vertices; this remains true in the constrained experiment with $K_{\mathrm{asm}}$ replaced by the polytope $K_{\mathrm{asm}}\cap\lbrace G_A\geq 0\rbrace$. Through the affine isomorphism $K_{\mathrm{mm}}\cong E(K_B)\times E(K_B)$, the multimeter polytope has at most $12^2=144$ vertices -- pairs of extreme effects, the count $12$ being attained by the pentagon. For each of them the remaining optimization over the assemblage polytope is a linear program; we solve all of them and re-derive, for each of them, an optimal vertex of the assemblage polytope together with an exact dual (Farkas) certificate in exact arithmetic over $\mathbb{Q}(\sqrt5)$ -- the pentagon is affinely equivalent to a polygon with vertices in $\mathbb{Q}(\sqrt5)^2$, and $\gamma^*$ is invariant under affine isomorphisms of $K_B$ -- so that the resulting values of $\gamma^*$ are exact closed forms (last column of \autoref{tab:numerics_results}). Details are given in the accompanying repository.

\paragraph{Rounded inner values.} For the inner sequence of \autoref{thm:inner_bounds} we round the optimizer $y^{(n)}$ of each level returned by the LP solver as in the proof of \autoref{thm:convergence_hierarchy}. The family $\lrbracket{\psi_k/s}_{k\in[r_{\mathrm{mm}}]}$ is an $r_{\mathrm{mm}}$-outcome measurement on $\widehat K_{\mathrm{mm}}$: completeness is the identity $\sum_k\psi_k/s=u_{\mathrm{mm}}=1_{\widehat K_{\mathrm{mm}}}$, which holds by construction \autoref{eqn:numerics_umm}, and it is informationally complete because the $\psi_k$ span $V_{\mathrm{mm}}^\star$.

For each $m\in\lbrace 0,\ldots,n-1\rbrace$, $m$ of the $n$ multimeter copies are measured with this measurement; conditioned on an outcome multiset $z^m$ of nonzero probability (by permutation invariance only the multiset matters), one unmeasured multimeter copy is retained and the others are marginalized, and the product of the conditional pair $\lrbracket{x_{A\vert z^m}, x_{B\vert z^m}}$ is formed. Every conditional pair inherits feasibility from the extension exactly as in the proof of \autoref{thm:convergence_hierarchy}: constraint (2b) gives $F_B\lrbracket{x_{B\vert z^m}}=0$, so that $x_{B\vert z^m}$ is a bona fide complete multimeter by \autoref{eqn:games_FB_slice}, and the inequality \autoref{eqn:numerics_inequality} survives conditioning via constraint (2c). Writing $\operatorname{Val}\lrbracket{\xi,N}\coloneqq L_G\lrbracket{\xi\otimes N}$ for the objective of \autoref{eqn:game_as_polynomial_optimization_1} on pairs, we therefore report
\begin{align}
    L^{(n)}\;\coloneqq\;\max_{0\leq m<n}\ \max_{\substack{z^m\\ p(z^m)>0}}\ \operatorname{Val}\lrbracket{x_{A\vert z^m}, x_{B\vert z^m}}\;\leq\;\gamma^*(G),
    \label{eqn:numerics_inner_value}
\end{align}
a lower bound at every level: each competitor is the value of an explicitly constructed feasible assemblage--multimeter pair. This construction concerns $\gamma^*(G)$; it does not by itself establish realizability over a separately fixed Alice state space or a lower bound on $\beta^*(G)$. In the numerical implementation, all reported rounded candidates satisfy the constraints with residual at most $10^{-10}$.\footnote{For these fixed nonempty polytopes, a polyhedral error bound controls the distance to the feasible set by $H\delta$, with a constant $H$ depending on the constraint representation and norms. A validated constant and an explicit certified repair would yield a rigorous correction to the objective. The reported residuals alone do not establish such a correction or certify the displayed digits.}
\FloatBarrier

\subsection{Results}
\label{sec:numerics_results}

\begin{table}[t]
\caption{Outer levels $p^{(n)}$ of \autoref{eqn:numerics_pn} (upper bounds on $\gamma^*(\mathrm{CHSH})$, winning-probability convention) and inner values $L^{(n)}$ of \autoref{eqn:numerics_inner_value} (numerical values of the feasible-strategy construction obtained by rounding the optimizer returned by our implementation, subject to the numerical-feasibility qualification in \autoref{sec:numerics_setup}; see observation (e) in \autoref{sec:numerics_results} for their dependence on the optimizer) for the levels $n=1,\ldots,5$, compared with the exact optimum $\gamma^*(\mathrm{CHSH})$ obtained by the direct computation described in \autoref{sec:numerics_setup}. All outer entries are numerical LP outputs obtained with solver tolerances of order $10^{-8}$, without a certified error bound, the reference column consists of exact closed forms; entries are rounded to four decimals.}\label{tab:numerics_results}
\centering\resizebox{\textwidth}{!}{
\begin{tabular}{llccccc c}\toprule
$K_B$ & & $n=1$ & $n=2$ & $n=3$ & $n=4$ & $n=5$ & $\gamma^*(\mathrm{CHSH})$ (exact)\\\midrule
classical bit & $p^{(n)}$ & $1.0000$ & $0.7500$ & $0.7500$ & $0.7500$ & $0.7500$ & $3/4=0.7500$\\
 & $L^{(n)}$ & $0.5000$ & $0.7500$ & $0.7500$ & $0.7500$ & $0.7500$ & \\
\midrule
square gbit & $p^{(n)}$ & $1.5000$ & $1.0000$ & $1.0000$ & $1.0000$ & $1.0000$ & $1$\\
 & $L^{(n)}$ & $0.5000$ & $0.5625$ & $0.6250$ & $1.0000$ & $1.0000$ & \\
\midrule
pentagon & $p^{(n)}$ & $1.3416$ & $0.9208$ & $0.9042$ & $0.8949$ & $0.8868$ & $(2\sqrt5-1)/4\approx 0.8680$\\
 & $L^{(n)}$ & $0.5000$ & $0.5459$ & $0.6044$ & $0.7074$ & $0.7555$ & \\
\midrule
pentagon, \autoref{eqn:numerics_inequality} & $p^{(n)}$ & $1.1733$ & $0.9181$ & $0.8952$ & $0.8774$ & $0.8597$ & $3(9+\sqrt5)/40\approx 0.8427$\\
 & $L^{(n)}$ & $0.5106$ & $0.6097$ & $0.6118$ & $0.6664$ & $0.7399$ & \\
\bottomrule\end{tabular}}\vspace{1mm}\end{table}

\begin{figure}[t]
\centering
\includegraphics[width=0.98\linewidth]{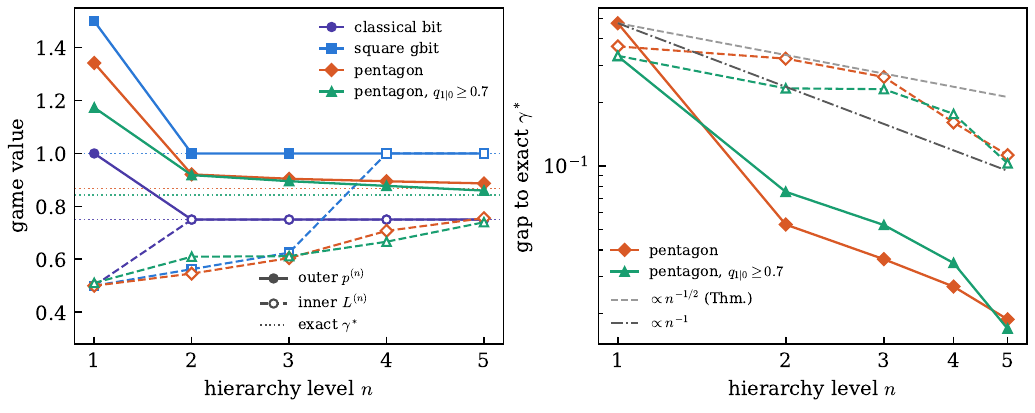}
\caption{Left: the sandwich. Outer levels $p^{(n)}$ (filled markers, solid) and inner rounded values $L^{(n)}$ (open markers, dashed) for the four experiments of \autoref{tab:numerics_results}; dotted horizontal lines mark the exact optima $\gamma^*(\mathrm{CHSH})$. The bracket closes to the reported numerical precision at $n=2$ for the classical bit and at $n=4$ for the square gbit. Right: gaps to the exact optimum on a log-log scale for the two pentagon experiments; filled markers show the outer gap $p^{(n)}-\gamma^*$, open markers the inner gap $\gamma^*-L^{(n)}$. The dashed guide indicates the $n^{-1/2}$ rate guaranteed by \autoref{thm:convergence_hierarchy} and \autoref{thm:inner_bounds}, the dash-dotted guide a $1/n$ decay.}
\label{fig:numerics_convergence}
\end{figure}

\autoref{tab:numerics_results} and \autoref{fig:numerics_convergence} summarize the outcome; we highlight six observations.

\paragraph{(a) Level $1$ is the maximal tensor product.} In the three unconstrained experiments, at $n=1$ the relaxation optimizes over $K_{\mathrm{asm}}\tmax K_{\mathrm{mm}}$,\footnote{Here $K_{\mathrm{mm}}$ is regarded as the base of the proper cone $\operatorname{cone}(K_{\mathrm{mm}})=C_{\mathrm{mm}}\cap\ker F_B$ of \autoref{eqn:games_FB_slice} in the space $\ker F_B$. Since $\ker F_B$ contains interior points of $C_{\mathrm{mm}}$ (e.g.\ the trivial multimeter $n_{b\vert y}=\frac{1}{2}1_{K_B}$), the dual of $C_{\mathrm{mm}}\cap\ker F_B$ consists of the restrictions of the functionals in $C_{\mathrm{mm}}^\star$, and therefore $\Sigma^1(A:B)=\lbrace x\in K_{\mathrm{asm}}\tmax\widehat K_{\mathrm{mm}}\,:\,\lrbracket{\operatorname{id}\otimes F_B}(x)=0\rbrace$ coincides with $K_{\mathrm{asm}}\tmax K_{\mathrm{mm}}$.} whereas $\gamma^*$ is the maximum of the bilinear objective over product pairs $\xi_B\otimes N_B$, i.e.\ over $K_{\mathrm{asm}}\tmin K_{\mathrm{mm}}$; in the constrained experiment the level-$1$ domain is cut further by the conic constraint $\lrbracket{G_A\otimes\operatorname{id}}(x)\in\operatorname{cone}\lrbracket{K_{\mathrm{mm}}}$ of the last row of \autoref{eqn:numerics_level_lp}. The difference between the two tensor products is precisely what the hierarchy has to resolve, and at $n=1$ it is fully visible.

Contracting the level-$1$ optimizer with the Choi tensor $\chi_{A(K_B)^\star}$ as in \autoref{eqn:games_contraction} yields a normalized, no-signalling table $p(a,b\vert x,y)$. For the classical bit, whose cone is simplicial, this table is a bona fide no-signalling box -- in fact the PR box -- and $p^{(1)}=1$ is the no-signalling value of the CHSH game: for a simplex, $\chi_{A(K_B)^\star}=\sum_k f_k\otimes\delta_k$ with the vertices $\delta_k$ and the non-negative coordinate effects $f_k$, so that the contraction is positive on the maximal tensor product $C_{\mathrm{asm}}\tmax C_{\mathrm{mm}}$. For the square gbit and the pentagon the contraction is no longer positive on the maximal tensor product, the table acquires negative entries ($-1/4$ for the gbit), and $p^{(1)}=3/2$ resp.\ $p^{(1)}=3/\sqrt5$ even exceed the algebraic maximum $1$ of a winning probability. Level $1$ is thus a valid but trivial bound; the non-trivial information is produced by the symmetric-extension constraints from $n=2$ on, where all four values drop to at most $1$.

\paragraph{(b) Numerical bracket closure for the classical bit and the square gbit.} For the classical bit, outer level and rounded inner point coincide already at $n=2$, for the square gbit at $n=4$, and in both cases the common value agrees with the exact optimum up to solver tolerance: $\gamma^*=3/4$, the classical CHSH bound \cite{Clauser1969}, and $\gamma^*=1$, the algebraic maximum of the game, i.e.\ the value of the PR box \cite{Popescu1994}. For these two state spaces the computed outer and rounded inner values coincide at a small finite level to solver tolerance, consistently with the exact reference value. This numerical agreement alone does not certify the exact equality $L^{(n)}=p^{(n)}$. We do not claim a finite-level collapse theorem for these cases: the bodies entering the hierarchy are $K_{\mathrm{asm}}$ and $\widehat K_{\mathrm{mm}}$ rather than $K_B$, and whether exactness at a finite level can be predicted from their structure, e.g.\ in the spirit of the rigidity theorem \cite[Thm.\ 2]{Aubrun_2024}, is left open.

\paragraph{(c) Strict decrease for the pentagon.} For the pentagon, the outer sequence decreases strictly over the computed levels $n=1,\ldots,5$, and the gap to the exact optimum shrinks from $0.47$ at $n=1$ to $p^{(5)}-\gamma^*=1.9\times 10^{-2}$ at $n=5$ ($1.7\times 10^{-2}$ in the constrained experiment). As \autoref{fig:numerics_convergence} (right) shows, over this range the outer gap is consistent with a decay of order $1/n$ -- about $0.1/n$ for the unconstrained pentagon -- and hence faster than the $n^{-1/2}$ rate guaranteed by \autoref{thm:convergence_hierarchy}; the a priori guarantee is valid but conservative in this instance, although five levels do not, of course, establish an asymptotic rate.

\paragraph{(d) The inner sequence consists of explicit strategies.} Every value $L^{(n)}$ is attained by an explicit assemblage-multimeter pair that satisfies all constraints of \autoref{eqn:game_as_polynomial_optimization_1} -- including \autoref{eqn:numerics_inequality} in the constrained experiment -- to accuracy $10^{-10}$, giving a numerical lower-bound candidate; the constraint residual is not itself a certified error bound on the objective. The rounded values approach the optimum from below as $n$ grows, more slowly than the outer values approach it from above: at $n=5$ the reported numerical bracket for the pentagon reads $[0.7555,\,0.8868]$ with the exact value $0.8680$ inside, and the inner gap $\gamma^*-L^{(5)}=0.11$ is about six times the outer gap. For all four experiments and all levels the exact optimum lies inside the reported numerical bracket $[L^{(n)},p^{(n)}]$, consistently with the exact inequalities for mathematical optima and exactly feasible rounded pairs. In practice the maximization over the number $m$ of measured copies in \autoref{eqn:numerics_inner_value} matters: for the optimizers rounded in \autoref{tab:numerics_results} the best candidate is always obtained at $m=n-1$, i.e.\ when all but one copy are measured.

\paragraph{(e) The inner sequence depends on the optimizer.} In contrast to the outer value $p^{(n)}$, the rounded value is not determined by the level alone. The level-$n$ LP typically has many optimizers -- its optimal face has positive dimension -- and \autoref{thm:inner_bounds} applies to each of them, but different optimizers round to different feasible strategies and hence to different lower bounds.

We probed this at $n=2$ by rounding optimizers of the same LP obtained with different LP algorithms (dual simplex and interior-point methods, with and without crossover) and with equivalent formulations of the LP: while $p^{(2)}$ is the same throughout, the rounded values observed in these probes range over $[0.520,\,0.750]$ for the classical bit, $[0.551,\,1.000]$ for the square gbit, $[0.532,\,0.562]$ for the pentagon and $[0.606,\,0.615]$ for the constrained pentagon. These are sampled ranges, not extrema over the full optimal face. For the square gbit, one vertex of the optimal face rounds to a strategy with PR-box correlations already at $m=0$, so that for this optimizer the bracket closes at $n=2$ rather than at $n=4$; in general, vertices of the optimal face tended to round to higher values than its interior points.

All candidates passed feasibility checks with residual at most $10^{-10}$, subject to the numerical-certification qualification above; the exact construction and its a priori guarantee in \autoref{thm:inner_bounds} apply to every exact optimizer, but the inner sequence is not canonical: monotonicity in $n$ is neither guaranteed nor claimed, and \autoref{tab:numerics_results} reports the values obtained from the optimizer returned by our implementation without optimizing over this freedom. The flip side is a cheap practical lever: since rounding costs little compared with solving the level-$n$ LP, one may round several optimizers, or several points of the optimal face, and keep the best strategy.

\paragraph{(f) Local inequality constraints work out of the box.} The constrained experiment exercises the inequality mechanism of constraint (2c) on both sides of the sandwich. The constraint binds -- the exact optimum drops from $(2\sqrt5-1)/4\approx 0.868$ to $3(9+\sqrt5)/40\approx 0.843$ -- every outer level remains a valid upper bound on the constrained optimum, and every rounded strategy satisfies \autoref{eqn:numerics_inequality}, i.e.\ is feasible for the \emph{constrained} problem, because the inequality survives conditioning exactly as in the proof of \autoref{thm:convergence_hierarchy}. The observed convergence is qualitatively the same as in the unconstrained case (\autoref{fig:numerics_convergence}, right).

\bigskip
In summary, the experiments deliver the intended proof of concept: for polytopic state spaces every level of $\Sigma^n(A:B)$ is an explicitly solvable program; the outer levels and the rounded inner points give the predicted sandwich numerically, subject to solver tolerance, closing the gap to the reported precision at finite $n$ for the classical bit and the square gbit and shrinking it steadily for the pentagon; and local equality and inequality constraints are handled by the same machinery on both sides, as promised by \autoref{thm:convergence_hierarchy} and \autoref{thm:inner_bounds}. For state spaces admitting conic lifts, \autoref{prop:lifted_conic_program} gives a bilinear reformulation on the lifting cones. A convergent lifted hierarchy additionally requires compact normalized lifted state spaces and tensor-cone descriptions; the original finite levels need not be preserved.

\section{Discussion and open problems}\label{sec:discussion_and_open_problems}

\subsection{Discussion}

We have established a finite de Finetti theorem for arbitrary convex bodies using the integral representation of relative entropy introduced in \cite{Frenkel_2023}. The proof combines the uniform bound on variational mutual information of \autoref{prop:uniform_A_upper_bound_mutual_information} with the classical chain rule after measurement and a bound on the distortion of informationally complete measurements. Together, these ingredients yield the $O(n^{-1/2})$ approximation in \autoref{thm:main_theorem}. This gives a quantitative form of monogamy of entanglement for general cones, strengthening the qualitative relation between symmetric extendibility and the minimal tensor product established in \cite{Aubrun_2024}.

Within the polynomial optimization framework of \cite{plavala2025polarization}, this theorem gives a convergent hierarchy of outer approximations with finite-level error bounds, including local affine equality and inequality constraints. The same proof produces mixtures whose product components satisfy the local constraints exactly. Selecting a suitable component gives a feasible product state for the original optimization problem, with an objective value approaching the optimum at the same $O(n^{-1/2})$ rate (\autoref{thm:convergence_hierarchy} and \autoref{thm:inner_bounds}). Thus, the framework supplies both relaxation bounds and a constructive route to feasible solutions.

The CHSH experiments in \autoref{sec:numerics} illustrate these two constructions for polytopic state spaces, where each relaxation is a finite linear program. The computed outer and rounded inner values bracket the exact reference optima at every reported level. For the classical bit and the square gbit, the brackets close to the reported precision at small extension orders. For the pentagon, the decreasing gaps, also observed under a local inequality constraint, motivate investigating sharper convergence estimates; the five computed levels do not establish an improved asymptotic rate.

Conic lifts provide a further way to represent the optimization problem when the original convex bodies are difficult to describe. \autoref{prop:lifted_conic_program} preserves the bilinear optimization problem under such a reformulation. To implement a convergent hierarchy on the lifting cones, one additionally needs compact normalized lifted state spaces, suitable tensor-cone descriptions and algorithms. An efficient description of each individual lifting cone does not by itself provide these ingredients or preserve the original finite hierarchy levels.

\subsection{Outlook}

The results suggest three directions for further work.
\begin{enumerate}[leftmargin=*,itemsep=0.9\baselineskip]
\item \textbf{Conditional information before measurement.}
Our proof uses classical conditional mutual information after applying measurements. A natural question is whether comparable conditional-information tools can be developed directly for the cone-valued states. One possible starting point is the Pythagorean identity for the Umegaki relative entropy on normalized quantum states, with the usual support conventions (cf.\ \cite{Weis2014}):
\begin{align}\label{eq:pythagorean_property}
 D(x_{AB}\Vert y_A\otimes x_B)
 =D(x_{AB}\Vert x_A\otimes x_B)+D(x_A\Vert y_A).
\end{align}
If an analogous identity held for the divergence on the maximal composite used here, minimizing over $y_A$ would show that the infimum in the definition of $I(A:B)$ is attained at $y_A=x_A$. This would resolve the variational question in \autoref{eq:questioned_equality} for the relevant class of state spaces.

A further question concerns conditional information. Defining $I(A:B\mid C)$ simply as $I(A:BC)-I(A:C)$ gives a formal chain rule, and the data-processing inequality already proved makes this difference nonnegative. What is needed for a direct extension of our argument is additional structure: a conditional representation or inequality that controls the discrepancy between joint conditional states and products of their marginals, as the classical representation does after measurement. It remains to determine when such a representation or estimate exists and which properties of the divergence ensure it. The Pythagorean identity is therefore a candidate structural ingredient, rather than by itself a complete replacement for the measured conditioning argument.

\item \textbf{Constraints coupling the two factors.}
The constructive rounding in \autoref{thm:convergence_hierarchy} preserves local constraints on every product component. For constraints coupling the two factors, the same conditioning construction still produces separable candidates of the form
\begin{align}
 x_{AB}^{(n)}
 =\sum_{\substack{z^{m(n)}\\p(z^{m(n)})>0}}p(z^{m(n)})\,
 x_{A\mid z^{m(n)}}\otimes x_{B\mid z^{m(n)}}.
\end{align}
However, replacing a joint conditional state by the product of its marginals need not preserve a coupled constraint on each component, even when the joint conditional state satisfies that constraint. Closedness preserves feasibility along convergent sequences of feasible points; it does not supply feasibility of these candidates. Extending the componentwise guarantee therefore requires a construction that respects the coupled constraints, or a repair procedure with a quantitative bound on its effect on the approximation error. Identifying assumptions under which such a construction is possible is an open problem.

\item \textbf{Certified computation of relative entropy.}
The quantum specialization of the integral representation has been used to obtain semidefinite approximations to relative-entropy optimization in \cite{kossmann2025optimisingrelativeentropysemidefinite}. Adapting these techniques to the GPT divergence could provide certified conic-program approximations. Such a result would require quantitative control of the integration error and the relevant primal and dual cone optimizations. An efficient conic lift alone does not ensure these properties. A useful next step is to identify classes of state spaces for which all of these ingredients can be supplied with explicit accuracy and complexity bounds.
\end{enumerate}

\section*{Statements and Declarations.}
\paragraph{Competing interests.} The authors declare that they have no competing interests.
\paragraph{Data and materials availability.} The implementation, explanatory notes and numerical examples are available in the GitHub repository \cite{JuliusAZeissGithub_GPT_Games}.
\paragraph{Use of AI tools.} Anthropic's Claude Fable 5 and OpenAI's ChatGPT 5.6 Sol assisted with obtaining the numerical results of Sect.~\ref{sec:numerics}. OpenAI's ChatGPT/Codex assisted with checking mathematical arguments, developing proof simplifications, and preparing this revised draft. The authors are responsible for checking and approving the final manuscript.

\appendix
\section{Preliminaries on GPTs}\label{sec:preliminaries_on_GPTs}

By \cite[Lem.\ 3.18]{Pl_vala_2023}, $A(K)^{\star+}$ is a proper cone, $\lrbracket{A(K)^{\star +}}^{\star}= A(K)^{+}$ \cite[Prop.\ B.11]{Pl_vala_2023}, and \cite[Prop.\ B.4]{Pl_vala_2023} establishes the canonical vector space isomorphism $A(K)^{\star\star}\simeq A(K)$. Consequently, $K$ admits a natural embedding into $A(K)^{\star +}$. In other words, each state $x\in K$ is a positive linear functional on affine functions. Choose $o\in K$ and let $V_0\coloneqq\operatorname{aff}(K)-o$ be the direction space of its affine hull. Consider the set
\begin{align}
    K'\coloneqq \left\{\begin{pmatrix}
           x-o \\
           1 \\
         \end{pmatrix}\, : \, x\in K\right\}\subseteq V_0\oplus \mathbb{R},
\end{align}
which is isomorphic to $K$. Every $f\in A(K)$ has a unique representation $f(x)=\psi(x-o)+c$ with $\psi\in V_0^\star$ and $c=f(o)$, so we have
\begin{align}
    \left\langle\begin{pmatrix}
           x-o \\
           1 \\
         \end{pmatrix}, (\psi, c)\right\rangle = \psi(x-o)+c.
\end{align}
In other words, $A(K)\cong V_0^\star\oplus\mathbb{R}$ as vector spaces. Then, by \cite[Lem.\ 3.20]{Pl_vala_2023}, we have
\begin{align}
    K'= \left\{\phi_x \in A(K)^{\star +}\, : \, \langle \phi_x, 1_K \rangle=1\right\}.
\end{align}
The state space arises as the intersection of a proper cone with an affine hyperplane. Importantly, by \cite[Lem.\ 3.34]{Pl_vala_2023}, $K$ is a base of the proper cone $A(K)^{\star +}$. Because $A(K)^{\star +}$ generates $A(K)^\star$, \cite[Lem.\ 3.35]{Pl_vala_2023} ensures that for every $\phi\in A(K)^\star$ there exist $x,y \in K$ and $\lambda, \mu\in \mathbb{R}_+$ such that 
\begin{align}
    \phi= \lambda x - \mu y.
\end{align}
Thus, any affine map on the state space $K$ is the restriction, under this embedding, of a unique linear map on $A(K)^\star$ (see \autoref{prop:single_variable_affine_to_linear}).

\section{Symmetry action on convex bodies}\label{sec:symmetry_action_on_convex_bodies}

\begin{proposition}
    Let $C\subseteq V$ be a proper cone with associated state space $K$. The symmetric group $S_n$ preserves the $C^{\hat{\otimes}n}$ cone structure.
\end{proposition}
\begin{proof}
    Let 
    \begin{align}
    \begin{split}
    U\,:\, &S_n \rightarrow \text{GL}\left(V^{\otimes n}\right)\\
    & \pi \mapsto U(\pi)
    \end{split}
    \end{align}
    be the natural permutation representation of $S_n$.
    For any $i\in [n]$, let $f_i\in C^{\star}$ and $x_i\in C$. Then, for any $\sigma \in S_n$, we have
    \begin{align}
    \begin{split}
    \left(\bigotimes_{i=1}^n f_i\right)\left(U_{\sigma}\left(\bigotimes_{i=1}^n x_i\right)\right)&=  \left(\bigotimes_{i=1}^n f_i\right)\left(\bigotimes_{i=1}^n x_{\sigma^{-1}(i)}\right)\\
         &=\prod_{i=1}^n f_i(x_{\sigma^{-1}(i)})= \left(\bigotimes_{i=1}^n f_{\sigma(i)}\right)\left(\bigotimes_{i=1}^n x_i\right).
    \end{split}
    \end{align}
    Thus, as linear functionals
    \begin{align}
        \left(\bigotimes_{i=1}^n f_i\right)\circ U_{\sigma} = \left(\bigotimes_{i=1}^n f_{\sigma(i)}\right).
    \end{align}
    Since for any $i\in [n]$ $f_i\in C^{\star}$, we have $f_{\sigma(i)}\in C^{\star}$ for every $\sigma \in S_n$. By definition of the maximal tensor product between cones, for any $z\in C^{\hat{\otimes} n}$ it follows that
    \begin{align}
        \left(\bigotimes_{i=1}^n f_i\right)\left(U_{\sigma}(z)\right)= \left(\bigotimes_{i=1}^n f_{\sigma(i)}\right)(z)\geq 0.
    \end{align}
    The same argument holds for $\sigma^{-1}$ and thus, $U_{\sigma}$ is an order isomorphism and $S_n$ acts via cone-preserving linear maps. 
\end{proof}
The order unit $1_K^{\otimes n}$ is permutation invariant, so the action also preserves normalized states. This result extends naturally to the $S_n$ action on $K_A\tmax K_B^{\tmax n}$.

\section{Affine maps on convex bodies}\label{sec:affine_maps}

In this section, we clarify several notions of affine maps and establish an isomorphism between affine maps and linear maps, providing the mathematical background for \autoref{fig:affine_maps_to_linear_maps}.

\begin{definition}[Affine maps, Def.\ 3.8 \cite{Pl_vala_2023}]
    Let $K$ be a state space and $V$ a vector space. A function $f:K\to V$ is called \emph{affine}, if 
    \begin{align}
        f(\lambda x + (1-\lambda)y) = \lambda f(x) + (1-\lambda)f(y) \quad x,y\in K, \ \lambda \in [0,1].
    \end{align}
    We denote by $A(K,V)$
    the vector space of affine functions from $K$ to $V$, and write $A(K)$ when $V \cong \mathbb{R}$ as a vector space.
\end{definition}

Let $K_A$ and $K_B$ denote state spaces. For maps whose domain is the Cartesian product $K_A\times K_B$, the term affine requires careful specification, as there are two canonical generalizations of affine structure to product spaces.

\begin{definition}[Separately and jointly affine maps]
    Let $K_A$ and $K_B$ be two state spaces and consider a map $f:K_A \times K_B \to V$ for a vector space $V$.
    \begin{enumerate}
        \item $f$ is called \emph{jointly affine} if for $ x_A,\, y_A \in K_A$ and $x_B,\, y_B \in K_B, \quad \lambda \in [0,1]$ we have
        \begin{align}
            f(\lambda(x_A,x_B) + (1-\lambda)(y_A,y_B)) = \lambda f(x_A,x_B) + (1-\lambda)f(y_A,y_B),
        \end{align}
        \item $f$ is called \emph{separately affine} if we have 
        \begin{align}
        \begin{split}
            K_A \to A(K_B, V), \quad  x_A \mapsto [x_B \mapsto f(x_A,x_B)], \\
            K_B \to A(K_A, V), \quad  x_B \mapsto [x_A \mapsto f(x_A,x_B)],
        \end{split}
        \end{align}
        i.e.\ for every fixed element $x_A$ the resulting map is affine from $K_B \to V$ and similar for fixed $x_B$. The set of separately affine maps is denoted as $\operatorname{SAff}(K_A\times K_B,V)$.
    \end{enumerate}
\end{definition}
The scalar affine-function space $A(K)$ has dual $A(K)^\star$. There is a natural embedding $K\subseteq A(K)^\star$ by evaluation: $\psi_x(h)=h(x)$ for $h\in A(K)$. Evaluation of a $V$-valued affine function instead takes values in $V$, and is not a scalar dual functional. We denote by $\mathcal{L}(V,W)$ the set of all linear maps $T\,:\, V\rightarrow W$. 

\begin{proposition}\label{prop:single_variable_affine_to_linear}
    Let $K$ be a state space and $V$ a finite-dimensional vector space. Then, 
    \begin{align}
        A(K,V)\cong \mathcal{L}\lrbracket{A(K)^{\star}, V}
    \end{align}
    as vector spaces.
\end{proposition}
\begin{proof}
The construction uses the scalar evaluation embedding
\begin{align}
\begin{split}
    \iota:\ &K\longrightarrow A(K)^\star,\\
    &x\longmapsto\psi_x,\qquad \psi_x(h)=h(x)\quad(h\in A(K)).
\end{split}
\end{align}
Choose a basis $(v_j)_{j=1}^{\dim V}$ of $V$ and write $f=\sum_j f_jv_j$, with scalar affine coordinate functions $f_j\in A(K)$. Define
\begin{align}
\begin{split}
    F:\ &A(K)^\star\longrightarrow V,\\
    &\phi\longmapsto\sum_j\phi(f_j)v_j.
\end{split}
\end{align}
Then $F$ is linear and $F(\psi_x)=f(x)$. The evaluation vectors span $A(K)^\star$: a scalar affine function annihilated by all of them vanishes on $K$ and is therefore zero. Hence this extension is unique, which also shows that it is independent of the chosen basis of $V$. Conversely, restriction along $\iota$ sends every linear $F$ to an affine map $f$. Finally,
\begin{align}
    A(K,V)\longrightarrow\mathcal L(A(K)^\star,V),\qquad f\longmapsto F
\end{align}
is linear, and the two constructions are inverse to each other.
\end{proof}

\begin{proposition}\label{prop:isomorphism_sep_affine}
    For finite-dimensional $V$, we have
    \begin{align}
        \mathcal{L}(A(K_A)^\star \otimes A(K_B)^\star,V) \cong \operatorname{SAff}(K_A\times K_B,V)
    \end{align}
    as vector spaces.
\end{proposition}
\begin{proof}
Let $f:K_A\times K_B\to V$ be separately affine. For every fixed $x_B$, \autoref{prop:single_variable_affine_to_linear} supplies a unique linear map
\begin{align}
    F_{x_B}:A(K_A)^\star\longrightarrow V
\end{align}
such that $F_{x_B}(\psi^A_{x_A})=f(x_A,x_B)$. Separate affinity and uniqueness imply that the map
\begin{align}
    \mathcal F:K_B\longrightarrow\mathcal L(A(K_A)^\star,V),\qquad x_B\longmapsto F_{x_B}
\end{align}
is affine: each affine identity can be checked on the spanning evaluation vectors $\psi^A_{x_A}$. Applying the single-variable proposition again gives a linear extension
\begin{align}
    \widetilde{\mathcal F}:A(K_B)^\star\longrightarrow\mathcal L(A(K_A)^\star,V).
\end{align}
Consequently,
\begin{align}
    \hat F(\phi_A,\phi_B)\coloneqq\widetilde{\mathcal F}(\phi_B)(\phi_A)
\end{align}
is a well-defined bilinear map on $A(K_A)^\star\times A(K_B)^\star$. The universal property of the algebraic tensor product (see \cite{Ryan2002} or \cite[App.\ C]{Pl_vala_2023}) supplies a unique linear map $F$ satisfying
\begin{align}\label{eq:defining_property}
    F(\psi^A_{x_A}\otimes\psi^B_{x_B})
    =\hat F(\psi^A_{x_A},\psi^B_{x_B})=f(x_A,x_B).
\end{align}
Conversely, restricting any ambient linear map to evaluation products gives a separately affine map. The products of evaluation vectors span the tensor space, so the extension is unique. Both constructions are linear and inverse to one another, proving the asserted isomorphism.
\end{proof}

\section{Proofs of auxiliary results}\label{sec:Proofs_of_auxiliary_results}

We restate and prove \autoref{prop:jencova}.
 \begin{proposition}[Restated]
    Consider state space $K$ and $0\leq\mu\leq1\leq\lambda<\infty$. For $x,y\in K$ satisfying the order bounds $\mu y \leq x \leq \lambda y$ (with respect to the ambient cone $A(K)^{\star +}$), one has
    \begin{align}
        D(x\Vert y) = \int_\mu^\lambda \frac{ds}{s} \left[\sup_{f\in E(K)} sf(y) - f(x) \right] + \ln\lambda +1 - \lambda.
    \end{align}
\end{proposition}
\begin{proof}
    We follow the general proof idea from \cite[Cor.\ 1]{Jen_ov__2024} with modifications to the GPT setting. We view $x,y$ as elements of the positive cone $A(K)^{\star +}$ via the canonical embedding of $K$ into $A(K)^{\star}$ as discussed in \autoref{sec:notation_and_preliminaries}. Each $f\in E(K)\subseteq A(K)$ acts linearly on the ambient ordered vector space $A(K)^{\star}$. Hence, for any real $t$,
    \begin{align}
        (1-t)f(x)+tf(y)=f\lrbracket{(1-t)x+ty}.
    \end{align}
    Split the integral over three regions $(-\infty, 0), [0,1], (1, \infty)$. Since $K$ is convex, 
    \begin{align}
        (1-t)x + ty \in K, \quad t\in [0,1].
    \end{align}
    Every effect $f\in E(K)$ is an affine map $K\rightarrow [0,1]$, hence in particular $f(z)\geq 0$ for all $z\in K$. Therefore, for $t\in [0,1]$,
    \begin{align}
        f\left((1-t)x +ty\right)\geq 0 \implies -\lrbracket{f\left((1-t)x +ty\right)}\leq 0
    \end{align}
    Since $0\in E(K)$, the supremum over $f\in E(K)$ is 0. Hence the integral over $[0,1]$ contributes 0. Via the Möbius substitution $s=t/(t-1)$, for which $dt/ds=-1/(s-1)^2$, and linearity of $f$ on the ambient space $A(K)^{\star}$ we obtain for $t<0$ and $0<s<1$
    \begin{align}
        f\lrbracket{(1-t)x+ ty}=\frac{1}{1-s}f(x-sy),  
    \end{align}
    hence 
    \begin{align}
        -f\lrbracket{(1-t)x+ ty}=\frac{1}{1-s}\lrbracket{-f(x-sy)}.
    \end{align}
    Computation of the Jacobian factor in the measure yields 
    \begin{align}
        \int_{-\infty}^0 \frac{dt}{-t  (t-1)^2} \sup_{f \in E(K)}-f((1-t)x + ty) =\int_{\mu}^1 \frac{ds}{s} \sup_{f\in E(K)} -f(x-ys). 
    \end{align}
    If $0\leq s \leq \mu$, then $x-sy\geq (\mu-s)y\geq 0$. For positive $f$, $f(x-sy)\geq 0$, so $-f(x-sy)\leq 0$ and the supremum over $E(K)$ equals 0 (attained by $f=0$). Hence the integral over $[0,\mu]$ is zero. When $\mu=0$, the quotient in this integral is bounded by $1$, since $\sup_f\{sf(y)-f(x)\}\leq s$. For $t>1$, the same substitution gives $s>1$ and $(1-t)x+ty=(sy-x)/(s-1)$. Since $x-sy\leq0$ for $s\geq\lambda$, we similarly obtain
     \begin{align}
        \int_1^\infty \frac{dt}{t (t-1)^2}\sup_{f\in E(K)} -f((1-t)x + ty) = \int_1^\lambda \frac{ds}{s} \sup_{f\in E(K)}f(x-ys).
    \end{align}
    For any $f\in E(K)$, also $\lrbracket{1_K-f}\in E(K)$ \cite[Lem.\ 3.15]{Pl_vala_2023}. Hence for any ambient vector $z$, 
    \begin{align}
        f(z)=1_K(z)-\lrbracket{1_K-f}(z),
    \end{align}
    and taking suprema over $f\in E(K)$ yields
    \begin{align}
        \sup_{f\in E(K)}f(z)=1_K(z)+\sup_{f\in E(K)}\lrbracket{-f(z)}.
    \end{align}
    Apply this with $z=x-sy$. Since $x,y$ are normalized states $1_K(x)=1_K(y)=1$, hence $ 1_K\lrbracket{x-sy}=1-s$. Therefore, 
     \begin{align}
        \sup_{f \in E(K)} f(x-ys) = 1 - s + \sup_{f\in E(K)} - f(x-ys).
    \end{align}
    Thus, combining terms and explicit calculation yields 
    \begin{align}
        D(x \Vert y) = \int_{\mu}^{\lambda}\frac{ds}{s} \sup_{f\in E(K)}\lrbracket{-f(x-sy)}+ \ln\lambda +1-\lambda.
    \end{align}
    By linearity $\sup_{f\in E(K)}\lrbracket{-f(x-sy)}=\sup_{f\in E(K)}\lrbracket{sf(y)-f(x)}$, completing the proof.
\end{proof}

For the convenience of the reader, we reproduce the following proposition from \cite{Puchaa2015} and provide a concise proof.
\begin{proposition}[\cite{Puchaa2015}]\label{prop:interior_points_of_proper_cones_are_order_units}
    Let $C\subseteq V$ be a proper cone in finite dimension. If $e\in \operatorname{int}(C)$, then $e$ is an order unit for $\lrbracket{V, C}$. Consequently, for any $x\in C$ the quantity 
    \begin{align}
        m_e(x)\coloneqq \inf\left\{\lambda \geq 0\,:\, x\leq \lambda e\right\}
    \end{align}
    is finite. 
\end{proposition}
\begin{proof}
    Fix any norm $\lVert \cdot \rVert$ on $V$. Since $e\in\operatorname{int}(C)$, there exists $r>0$ such that $B(e,r)\coloneqq \left\{y \in V\,:\, \lVert y - e\rVert < r \right\}\subset C$, i.e.\ the open ball around $e$ is contained in $C$. Given $v\in V$, choose $\mu >0$ so that $\left\lVert \frac{v}{\mu}\right\rVert < r$. Then, $e-\frac{1}{\mu} v \in B(e, r)\subset C$, hence by conicity $\mu e-v = \mu \lrbracket{ e- \frac{1}{\mu}v}\in C$. Thus, $v\leq \mu e$. In particular, for $x\in C$ the set $\left\{\lambda \geq 0\,:\, x\leq \lambda e\right\}$ is non-empty (take $\lambda = \mu$), so $0\leq m_e(x)\leq \mu < \infty$. 
\end{proof}
Note that $ m_e(x)\leq \lVert x\rVert_e$, where $\lVert \cdot\rVert_e$ is the order-unit norm \cite[Sec.\ 3.6]{Pl_vala_2023}; moreover, equality holds whenever $x\geq 0$.

We restate and prove \autoref{lem:injectivity}.
\begin{lemma}[Restated]
   Let $T_A:V_A\to W_A$ be an injective linear map between finite dimensional vector spaces. Then, for any finite dimensional vector space $V_B$, the map
    \begin{align}
        T_A\otimes \operatorname{id}_B:V_A \otimes V_B \to W_A \otimes V_B
    \end{align}
    is injective. In particular, for any choice of norms $\Vert \cdot \Vert_{\operatorname{in}}$ on $V_A\otimes V_B$ and $\Vert\cdot \Vert_{\operatorname{out}}$ on $W_A \otimes V_B$, there exists a constant $\tilde{f}(A,B,T_A)\in\mathbb{R}_+$, depending only on $V_A$, $V_B$ and $T_A$ and the chosen norms, such that 
    \begin{align}
        \Vert T_A \otimes \operatorname{id}_B(X_{AB})\Vert_{\operatorname{out}} \geq \tilde{f}(A,B,T_A) \Vert X_{AB}\Vert_{\operatorname{in}}, \quad X_{AB} \in V_A\otimes V_B.
    \end{align}
\end{lemma}
\begin{proof}
If $V_A\otimes V_B=\{0\}$, take $\tilde f=1$. Otherwise, since $T_A$ is injective and the spaces are finite dimensional, it has a linear left inverse $S_A:W_A\to V_A$. Thus $S\coloneqq S_A\otimes\operatorname{id}_B$ satisfies
\begin{align}
    S(T_A\otimes\operatorname{id}_B)=\operatorname{id}_{V_A\otimes V_B},
\end{align}
which proves injectivity. Equip $S$ with the operator norm from $\lVert\cdot\rVert_{\operatorname{out}}$ to $\lVert\cdot\rVert_{\operatorname{in}}$. It is finite and strictly positive, and
\begin{align}
    \lVert X_{AB}\rVert_{\operatorname{in}}
    \leq \lVert S\rVert\,
    \lVert(T_A\otimes\operatorname{id}_B)(X_{AB})\rVert_{\operatorname{out}}.
\end{align}
Taking $\tilde f=1/\lVert S\rVert$ proves the norm bound.
\end{proof}

\begin{proposition}[Hölder-type inequality for unit-order norm and its dual]\label{prop:hoelder_ineq_order_unit_norm}
    Let $K$ be a state space. Then,
    \begin{align}
        \vert\langle x, f \rangle\vert \leq \lVert f \rVert_{1_{K}}\lVert x \rVert_{\star, 1_{K}}
    \end{align}
    for $f\in A(K)$ and $x\in A(K)^\star$. 
\end{proposition}
\begin{proof}
By \cite[Prop.\ 3.40]{Pl_vala_2023}
\begin{align}
    \lVert x\rVert_{\star, 1_K}\coloneqq \sup_{\lVert g\rVert_{1_K}\leq 1}   \vert\langle x, g \rangle\vert.
\end{align}
For $f\neq 0$, we have
     \begin{align}
    \begin{split}
        \vert\langle x, f \rangle\vert &=\lVert f\rVert_{1_{K}} \left\lvert\left\langle x, \frac{f}{\lVert f\rVert_{1_{K}}} \right\rangle \right\rvert \\
        &\leq \lVert f\rVert_{1_{K}} \sup_{\lVert g\rVert_{1_{K}}\leq1}\lvert\langle x, g\rangle\rvert\\
        &= \lVert f \rVert_{1_{K}}\lVert x \rVert_{\star, 1_{K}}.  
    \end{split} 
\end{align}
For $f=0$ the claim follows directly. 
\end{proof}

The remaining three lemmas collect the facts on bases of proper cones and on maximal tensor products with a polysimplex that enter the proof of \autoref{lemma:games_geometry}; \autoref{lem:strictly_positive_functionals} is also used in \autoref{sec:information_theoretic_de_finetti_representation_theorem}. They are standard. \autoref{lem:bases_of_proper_cones} is the correspondence between state spaces and proper cones with an order unit that underlies the GPT formalism \cite[Sec.\ 3.4]{Pl_vala_2023}, \cite[Sec.\ 3.1]{achenbach2025factorizationmultimetersunifiedview}; \autoref{lem:strictly_positive_functionals} is the fact behind the dual state spaces of \cite[Lem.\ A.1]{Jenov2018} and \cite[Eq.\ (3.9)]{achenbach2025factorizationmultimetersunifiedview}; and \autoref{lem:polysimplex_tensors} contains the identification of assemblages with the points of $CS^1_{k,g}\tmax K$ \cite[Thm.\ 4]{Jenov2018}.

We include the short proofs because we need these facts in a basis-free form and for cones that are not presented as $A(K)^{\star+}$. All vector spaces are finite-dimensional and real, and we use the bipolar identity $D^{\star\star}=D$ for closed convex cones $D$ \cite[Prop.\ B.11]{Pl_vala_2023}.

\begin{lemma}[Bases of proper cones]\label{lem:bases_of_proper_cones}
Let $C$ be a proper cone in a vector space $V\neq\lbrace 0\rbrace$, let $u\in V^\star$ satisfy $u(z)>0$ for all $z\in C\setminus\lbrace 0\rbrace$, and put
\begin{align}
    K\coloneqq\lbrace z\in C\,:\,u(z)=1\rbrace.
\end{align}
Then:
\begin{enumerate}[label=(\alph*)]
    \item $u\in\intr C^\star$, and $u$ is an order unit of $(V^\star,C^\star)$;
    \item $K$ is nonempty, convex and compact, and every $z\in C\setminus\lbrace 0\rbrace$ can be written in exactly one way as $z=tz'$ with $t>0$ and $z'\in K$; thus $K$ is a compact base of $C$, and $C=\operatorname{cone}(K)=\lbrace tz'\,:\,t\geq 0,\ z'\in K\rbrace$;
    \item the restriction map
    \begin{align}
        r\,:\,V^\star\rightarrow A(K),\qquad f\mapsto f\vert_K,
    \end{align}
    is a linear isomorphism with $r(C^\star)=A(K)^+$ and $r(u)=1_K$; moreover,
    \begin{align}
        \iota\,:\,V\rightarrow A(K)^\star,\qquad \iota(z)\coloneqq\bigl[h\mapsto (r^{-1}h)(z)\bigr],
    \end{align}
    is a linear isomorphism which extends the evaluation embedding of $K$ into $A(K)^\star$ and maps $C$ onto $A(K)^{\star+}$;
    \item if $C_2\subseteq V_2$, $u_2$ and $K_2$ are a second triple of this kind, with maps $r_2$ and $\iota_2$ as in (c), then $\iota\otimes\iota_2$ maps
    \begin{align}
        \lbrace z\in C\tmax C_2\,:\,(u\otimes u_2)(z)=1\rbrace
    \end{align}
    onto the maximal tensor product $K\tmax K_2$ of the state spaces $K$ and $K_2$.
\end{enumerate}
In the terminology of \autoref{sec:notation_and_preliminaries}, (b) and (c) say that $K$ is a state space with ambient space $A(K)^\star\cong V$, cone $A(K)^{\star+}\cong C$ and unit effect $1_K\cong u$.
\end{lemma}
\begin{proof}
Fix a norm $\lVert\cdot\rVert$ on $V$. Since $C$ is closed and $C\neq\lbrace 0\rbrace$ (it spans $V\neq\lbrace 0\rbrace$), the set $\lbrace z\in C\,:\,\lVert z\rVert=1\rbrace$ is compact and nonempty, and $u$ is positive on it. Hence $u$ attains a minimum $\delta>0$ on this set, and by homogeneity
\begin{align}\label{eqn:bases_of_proper_cones_bound}
    u(z)\geq\delta\,\lVert z\rVert\qquad\text{for all } z\in C.
\end{align}

For (a), let $f\in V^\star$ and let $\lVert f\rVert$ be its dual norm. By \eqref{eqn:bases_of_proper_cones_bound},
\begin{align}
    \lvert f(z)\rvert\leq\lVert f\rVert\,\lVert z\rVert\leq\delta^{-1}\lVert f\rVert\,u(z)\qquad\text{for all } z\in C,
\end{align}
that is, $\delta^{-1}\lVert f\rVert\,u\pm f\in C^\star$. Consequently $f\leq\delta^{-1}\lVert f\rVert\,u$ in the order induced by $C^\star$, so $u$ is an order unit, and $u+f\in C^\star$ whenever $\lVert f\rVert\leq\delta$, so $u\in\intr C^\star$.

For (b), the set $K$ is convex and closed, and it is bounded by \eqref{eqn:bases_of_proper_cones_bound}. Every $z\in C\setminus\lbrace 0\rbrace$ satisfies $z=u(z)\,z'$ with $z'\coloneqq z/u(z)\in K$; in particular $K\neq\emptyset$. If $z=tz'$ with $t>0$ and $z'\in K$, then $t=u(z)$, which proves uniqueness. Since $K$ is convex, $\lbrace tz'\,:\,t\geq 0,\ z'\in K\rbrace$ is the smallest cone containing $K$, and by the above it equals $C$.

For (c), if $f\vert_K=0$, then $f$ vanishes on $\operatorname{cone}(K)=C$ and hence on $V$, because $C$ is generating; so $r$ is injective. Since $\operatorname{aff}(K)$ is contained in the hyperplane $u^{-1}(1)$, we have
\begin{align}
    \dim A(K)=\dim\operatorname{aff}(K)+1\leq\dim V
\end{align}
(\autoref{sec:preliminaries_on_GPTs}), so $r$ is bijective. A functional is nonnegative on $K$ if and only if it is nonnegative on $\operatorname{cone}(K)=C$, whence $r(C^\star)=A(K)^+$, and $r(u)=1_K$ by the definition of $K$.

Finally, $\iota$ is the transpose of the isomorphism $r^{-1}$, composed with $V\cong V^{\star\star}$, and hence an isomorphism. It sends $x\in K$ to the evaluation $h\mapsto h(x)$, and $\iota(z)$ lies in $A(K)^{\star+}$ if and only if $f(z)\geq 0$ for all $f\in r^{-1}(A(K)^+)=C^\star$, that is, if and only if $z\in C^{\star\star}=C$.

For (d), the definition of $\iota$ and $\iota_2$ gives
\begin{align}
    (h\otimes h_2)\bigl((\iota\otimes\iota_2)(z)\bigr)=(r^{-1}h\otimes r_2^{-1}h_2)(z)
\end{align}
for all $h\in A(K)$, $h_2\in A(K_2)$ and $z\in V\otimes V_2$. The dual cone of $A(K)^{\star+}$ is $A(K)^+$, and $r^{-1}(A(K)^+)=C^\star$ by (c); the same holds for $K_2$. Hence $(\iota\otimes\iota_2)(z)$ lies in $A(K)^{\star+}\tmax A(K_2)^{\star+}$ if and only if $z\in C\tmax C_2$. Moreover, $r^{-1}1_K=u$ and $r_2^{-1}1_{K_2}=u_2$, so that
\begin{align}
    (1_K\otimes 1_{K_2})\bigl((\iota\otimes\iota_2)(z)\bigr)=(u\otimes u_2)(z).
\end{align}
Since $\iota\otimes\iota_2$ is bijective, the claim follows from the definition of $K\tmax K_2$ in \autoref{sec:notation_and_preliminaries}.
\end{proof}

\begin{lemma}[Strictly positive functionals of a state space]\label{lem:strictly_positive_functionals}
Let $K$ be a state space.
\begin{enumerate}[label=(\alph*)]
    \item $\phi(1_K)>0$ for every $\phi\in A(K)^{\star+}\setminus\lbrace 0\rbrace$.
    \item If $\omega\in\operatorname{relint}K$, then $f(\omega)>0$ for every $f\in A(K)^{+}\setminus\lbrace 0\rbrace$.
\end{enumerate}
Consequently,
\begin{align}
    1_K\in\intr A(K)^{+}\qquad\text{and}\qquad\operatorname{relint}K=K\cap\intr A(K)^{\star+}.
\end{align}
\end{lemma}
\begin{proof}
For (a), let $\phi\in A(K)^{\star+}$ with $\phi(1_K)=0$. For $f\in A(K)$ and $\mu\coloneqq\max_{K}\lvert f\rvert$ we have $\mu 1_{K}\pm f\in A(K)^+$, hence $\pm\phi(f)\leq\mu\,\phi(1_{K})=0$. Thus $\phi=0$.

For (b), let $f\in A(K)^+$ with $f(\omega)=0$, and let $\rho\in K$. The points $\omega-\varepsilon(\rho-\omega)$, $\varepsilon>0$, lie in $\operatorname{aff}(K)$ and converge to $\omega$ as $\varepsilon\to 0$; since $\omega\in\operatorname{relint}K$, there is therefore an $\varepsilon>0$ such that
\begin{align}
    \rho'\coloneqq\omega-\varepsilon(\rho-\omega)\in K.
\end{align}
Then $\omega=(1+\varepsilon)^{-1}(\rho'+\varepsilon\rho)$ is a convex combination of $\rho'$ and $\rho$, so
\begin{align}
    0=(1+\varepsilon)f(\omega)=f(\rho')+\varepsilon f(\rho),
\end{align}
and $f(\rho)=0$ because both summands are nonnegative. Thus $f=0$.

For the consequences, recall from \autoref{sec:notation_and_preliminaries} that $A(K)^{+}$ and $A(K)^{\star+}$ are proper cones, each the dual cone of the other, and that $K=\lbrace\phi\in A(K)^{\star+}\,:\,\phi(1_K)=1\rbrace$. By (a) and \autoref{lem:bases_of_proper_cones}(a), applied to $C=A(K)^{\star+}$ and $u=1_K$, we have $1_K\in\intr A(K)^{+}$. By (b) and \autoref{lem:bases_of_proper_cones}(a), applied to $C=A(K)^{+}$ and the evaluation at $\omega$, every $\omega\in\operatorname{relint}K$ lies in $\intr A(K)^{\star+}$.

Conversely, let $\omega\in K\cap\intr A(K)^{\star+}$, and let $B\subseteq A(K)^{\star+}$ be a ball around $\omega$. Since $\phi(1_K)=1$ for all $\phi\in\operatorname{aff}(K)$, we get
\begin{align}
    B\cap\operatorname{aff}(K)\subseteq\lbrace\phi\in A(K)^{\star+}\,:\,\phi(1_K)=1\rbrace=K,
\end{align}
i.e.\ $\omega\in\operatorname{relint}K$.
\end{proof}

\begin{lemma}[Maximal tensor products with a polysimplex]\label{lem:polysimplex_tensors}
Let $k,g\in\mathbb{N}$, and let $CS_{k,g}$, $CS^+_{k,g}$, $CS^1_{k,g}$, $1_{CS^1_{k,g}}$ and $\epsilon^{(j)}_i$ be as in \autoref{sec:games_setup}. Let $D$ be a proper cone in a vector space $U\neq\lbrace 0\rbrace$, let $w\in U^\star$ satisfy $w(v)>0$ for all $v\in D\setminus\lbrace 0\rbrace$, and let
\begin{align}
    K_D\coloneqq\lbrace v\in D\,:\,w(v)=1\rbrace
\end{align}
be the corresponding state space (\autoref{lem:bases_of_proper_cones}). For $\zeta\in CS_{k,g}\otimes U$ call
\begin{align}
    \zeta_{i\vert j}\coloneqq(\epsilon^{(j)}_i\otimes\operatorname{id}_U)(\zeta)\in U,\qquad i\in[k],\ j\in[g],
\end{align}
the \emph{entries} of $\zeta$. Then:
\begin{enumerate}[label=(\alph*)]
    \item for every $j\in[g]$,
    \begin{align}
        \sum_{i\in[k]}\zeta_{i\vert j}=(1_{CS^1_{k,g}}\otimes\operatorname{id}_U)(\zeta),
    \end{align}
    and $\zeta\mapsto(\zeta_{i\vert j})_{i,j}$ is a linear bijection from $CS_{k,g}\otimes U$ onto the space of all families $(v_{i\vert j})_{i\in[k],\,j\in[g]}$ in $U$ whose sums $\sum_{i\in[k]}v_{i\vert j}$ do not depend on $j$;
    \item $\zeta\in CS^+_{k,g}\tmax D$ if and only if $\zeta_{i\vert j}\in D$ for all $i\in[k]$ and $j\in[g]$;
    \item $CS^+_{k,g}\tmax D$ is a proper cone in $CS_{k,g}\otimes U\neq\lbrace 0\rbrace$, the functional $1_{CS^1_{k,g}}\otimes w$ is strictly positive on its nonzero elements and is an order unit in the interior of its dual cone, and, under the identifications of \autoref{lem:bases_of_proper_cones}(c),
    \begin{align}
        CS^1_{k,g}\tmax K_D=\bigl\lbrace\zeta\in CS^+_{k,g}\tmax D\,:\,(1_{CS^1_{k,g}}\otimes w)(\zeta)=1\bigr\rbrace
    \end{align}
    is a nonempty compact base of $CS^+_{k,g}\tmax D$ and a state space with ambient space $CS_{k,g}\otimes U$, cone $CS^+_{k,g}\tmax D$ and unit effect $1_{CS^1_{k,g}}\otimes w$;
    \item $\zeta\mapsto(\zeta_{i\vert j})_{i,j}$ restricts to an affine bijection, with affine inverse, from $CS^1_{k,g}\tmax K_D$ onto the set of all families $(v_{i\vert j})_{i\in[k],\,j\in[g]}$ of elements of $D$ whose sums $\sum_{i\in[k]}v_{i\vert j}$ do not depend on $j$ and lie in $K_D$.
\end{enumerate}
The same statements hold with the two tensor factors in the opposite order.
\end{lemma}
\begin{proof}
Recall from \autoref{sec:games_setup} that $CS^+_{k,g}$ is the cone of entrywise nonnegative matrices in $CS_{k,g}$. It is proper: it is closed and pointed, being the intersection of the nonnegative orthant with the subspace $CS_{k,g}$, and generating, since $CS_{k,g}=\operatorname{span}(CS^1_{k,g})$. Moreover, $1_{CS^1_{k,g}}$ is strictly positive on its nonzero elements, and
\begin{align}
    CS^1_{k,g}=\lbrace M\in CS^+_{k,g}\,:\,1_{CS^1_{k,g}}(M)=1\rbrace:
\end{align}
a nonnegative matrix in $CS_{k,g}$ with common column sum $c$ vanishes if $c=0$ and is $c$ times a column-stochastic matrix otherwise, and $1_{CS^1_{k,g}}$ returns $c$.

The first assertion of (a) follows from \eqref{eqn:games_readout_relation}. For the second, let $(e_l)_l$ be a basis of $U$ with dual basis $(e^l)_l$. Every $\zeta\in CS_{k,g}\otimes U$ has a unique representation
\begin{align}
    \zeta=\sum_l M_l\otimes e_l\qquad\text{with } M_l\in CS_{k,g},
\end{align}
and then $e^l(\zeta_{i\vert j})=(M_l)_{i,j}$; hence the entries determine $\zeta$. Conversely, let $(v_{i\vert j})_{i,j}$ be a family in $U$ whose sums $\sum_i v_{i\vert j}$ do not depend on $j$. Then the column sums $e^l\bigl(\sum_i v_{i\vert j}\bigr)$ of the matrices
\begin{align}
    M_l\coloneqq\bigl(e^l(v_{i\vert j})\bigr)_{i,j}
\end{align}
do not depend on $j$, i.e.\ $M_l\in CS_{k,g}$, and $\zeta\coloneqq\sum_l M_l\otimes e_l$ has the entries $\sum_l e^l(v_{i\vert j})\,e_l=v_{i\vert j}$.

For (b), observe that for $f\in U^\star$ the matrix $(\operatorname{id}\otimes f)(\zeta)\in CS_{k,g}$ has the entries $f(\zeta_{i\vert j})$. If $\zeta\in CS^+_{k,g}\tmax D$, then
\begin{align}
    f(\zeta_{i\vert j})=(\epsilon^{(j)}_i\otimes f)(\zeta)\geq 0\qquad\text{for all } f\in D^\star,
\end{align}
because $\epsilon^{(j)}_i\in(CS^+_{k,g})^\star$; hence $\zeta_{i\vert j}\in D^{\star\star}=D$. Conversely, if all entries of $\zeta$ lie in $D$ and $f\in D^\star$, then $(\operatorname{id}\otimes f)(\zeta)$ is a nonnegative matrix in $CS_{k,g}$, i.e.\ an element of $CS^+_{k,g}$, so that
\begin{align}
    (h\otimes f)(\zeta)=h\bigl((\operatorname{id}\otimes f)(\zeta)\bigr)\geq 0\qquad\text{for every } h\in(CS^+_{k,g})^\star.
\end{align}

For (c), the set $CS^+_{k,g}\tmax D$ is a closed convex cone, being an intersection of closed half-spaces. It is generating: for $M\in CS^+_{k,g}$ and $v\in D$ it contains the product tensor $M\otimes v$, because
\begin{align}
    (h\otimes f)(M\otimes v)=h(M)\,f(v)\geq 0\qquad\text{for all } h\in(CS^+_{k,g})^\star \text{ and } f\in D^\star,
\end{align}
and these product tensors span $CS_{k,g}\otimes U$, since $CS^+_{k,g}$ and $D$ are generating. It is pointed: if $\zeta$ and $-\zeta$ lie in $CS^+_{k,g}\tmax D$, then $\zeta_{i\vert j}$ and $-\zeta_{i\vert j}$ lie in the pointed cone $D$ by (b), so all entries of $\zeta$ vanish and $\zeta=0$ by (a).

Next, let $\zeta\in CS^+_{k,g}\tmax D$. By (a),
\begin{align}
    (1_{CS^1_{k,g}}\otimes w)(\zeta)=\sum_{i\in[k]}w(\zeta_{i\vert j})\qquad\text{for every } j\in[g],
\end{align}
and by (b) all summands are nonnegative. If this number is zero, then $w(\zeta_{i\vert j})=0$ and hence $\zeta_{i\vert j}=0$ for all $i$ and $j$, so that $\zeta=0$ by (a).

Thus \autoref{lem:bases_of_proper_cones} applies to the cone $CS^+_{k,g}\tmax D$ and the functional $1_{CS^1_{k,g}}\otimes w$; its parts (a)--(c) give the remaining assertions for the set on the right-hand side of the identity displayed in (c). This identity itself is \autoref{lem:bases_of_proper_cones}(d) for the cone $CS^+_{k,g}$ with the functional $1_{CS^1_{k,g}}$, whose base is $CS^1_{k,g}$, and the cone $D$ with the functional $w$, whose base is $K_D$.

For (d), recall from (a) that $(1_{CS^1_{k,g}}\otimes w)(\zeta)=w\bigl(\sum_{i}\zeta_{i\vert j}\bigr)$ for every $j$. By (a), (b) and the identity in (c), a tensor $\zeta\in CS_{k,g}\otimes U$ therefore lies in $CS^1_{k,g}\tmax K_D$ if and only if its entries lie in $D$ and their sums $\sum_i\zeta_{i\vert j}$, which do not depend on $j$ and lie in $D$, lie in $K_D$; and every family of this kind is the family of entries of exactly one $\zeta$. The resulting bijection and its inverse are restrictions of linear maps and hence affine.

The same arguments apply with the two tensor factors in the opposite order.
\end{proof}

\section{Comparison to Polarization Hierarchy}\label{sec:Comparison_to_Polarization_Hierarchy}

In this section we consider the special class of local constraints from \autoref{sec:comparison_DPS_PLG}, but without inequality constraints, i.e. we would like to describe the set of elements
\begin{align}\label{eq:feasible_set}
    \mathcal{F}\coloneqq \{(x_A,x_B) \ : \ f_A(x_A) = 0, \ f_B(x_B) = 0\} \subseteq K_A \times K_B
\end{align}
for two affine functions $f_A:K_A \to U_A$ and $f_B:K_B \to U_B$. In recent works in quantum information theory, this has turned out to be a particularly useful, yet somewhat intricate, case (e.g.\ \cite{Berta2021,zeiss2025approximatingfixedsizequantum,kossmann2025symmetric}). In particular, \cite[Sec.~II]{plavala2025polarization} points out (see also \autoref{sec:comparison_DPS_PLG}) that replacing the product problem
\begin{align}
    \begin{split}
        \inf_{(x_A,x_B) \in K_A \times K_B} \ & p(x_A,x_B) \\
        \text{s.t.}\quad & f_A(x_A) = 0 \quad \text{and} \quad f_B(x_B)= 0
    \end{split}
\end{align}
by the averaged-constraint problem
\begin{align}
    \begin{split}
        \inf_{x_{AB} \in K_A \dot{\otimes} K_B} \ & \langle P, x_{AB}\rangle \\
        \text{s.t.}\quad & \hat{F}_{\hookrightarrow}(x_{AB})=0
    \end{split}
\end{align}
can change the optimal value. The latter imposes the constraints only after averaging. In \cite[Sec.~2.4]{kossmann2025symmetric}, a concrete construction is given that illustrates the difference to the problem of optimizing over the convex hull of product states satisfying the constraints
\begin{align}
    \Sigma(A:B) \coloneqq \operatorname{conv}\{x_A \otimes x_B \in K_A \dot{\otimes} K_B \ : \ f_A(x_A) = 0, \ f_B(x_B)= 0\},
\end{align}
namely,
\begin{align}\label{eq:appendix_sigma_opt}
\begin{split}
    \inf \   & \langle P,x_{AB}\rangle \\
    \text{s.t.}\quad & x_{AB} \in \Sigma(A:B).
\end{split}
\end{align}
Let $E_A=A(K_A)^\star$ and $E_B=A(K_B)^\star$, and let $F_A:E_A\to U_A$ and $F_B:E_B\to U_B$ be the linear extensions of $f_A$ and $f_B$. The averaged local constraints are encoded by the ambient linear map
\begin{align}\label{eq:appendix_constraint_function}
    \begin{split}
        \hat F_{\hookrightarrow}:&\ E_A\otimes E_B\longrightarrow U_A\oplus U_B,\\
        \hat F_{\hookrightarrow}(x_A\otimes x_B)
        =&\ 1_{K_B}(x_B)F_A(x_A)\oplus1_{K_A}(x_A)F_B(x_B).
    \end{split}
\end{align}
We now specialize the corrected convergence argument of \cite[Thms.\ 5.1 and 5.4, Rem.\ 5.2]{zeiss2026commentpolarization} to local constraints: the identity-polarization hierarchy converges to the value of \autoref{eq:appendix_sigma_opt}, and the limiting optimal state admits a decomposition whose components satisfy the local constraints.
\begin{proposition}\label{prop:polarization_forall_i}
    Assume $\mathcal F\ne\varnothing$ and take $\Pi=\operatorname{id}$ in the polarization hierarchy of \cite[eq.\ (28)]{plavala2025polarization}, with constraint function \autoref{eq:appendix_constraint_function}. A diagonal subsequence of its optimizers has fixed-order marginals converging to a bicompatible family $\{y_n\}_{n\in\mathbb N}$, whose first member is optimal for \autoref{eq:appendix_sigma_opt} and has a decomposition
    \begin{align}
        y_1 = \sum_i p_i y_A^i \otimes y_B^i
    \end{align}
    with $(y_A^i,y_B^i) \in \mathcal{F}$ for all $i \in I$ and $\vert I \vert <\infty$.
\end{proposition}
\begin{proof}
Let $p^*$ denote the value of \autoref{eq:appendix_sigma_opt}. At every level $n\geq2$, the feasible set of \autoref{eqn:Polarization_hierarchy} is a closed subset of a compact normalized tensor body and is nonempty because every pair in $\mathcal F$ gives a feasible product-power extension. Hence an optimizer $z^{(n)}$ exists. Reduction preserves feasibility and the objective, so the optimal values are nondecreasing and converge to a limit $L\leq p^*$; a finite lower bound follows from compactness of $K_A\tmax K_B$.

For each fixed $k$, let $R_{n,k}$ retain $k$ copies of each system. The marginals $R_{n,k}z^{(n)}$ lie in a compact set. Successively extracting subsequences for $k=1,2,\ldots$ and taking a diagonal subsequence gives indices $n_j\to\infty$ such that
\begin{align}
    R_{n_j,k}z^{(n_j)}\longrightarrow y_k\qquad\text{for every fixed }k.
\end{align}
Positivity, normalization and the independent permutation symmetries pass to the limit. For $k>\ell$, continuity and $R_{k,\ell}R_{n,k}=R_{n,\ell}$ give $R_{k,\ell}y_k=y_\ell$. Thus $(y_k)$ is bicompatible. The polarization equality passes to $y_2$, and the objective to $y_1$, with $\langle P,y_1\rangle=L$.

By \cite[Lem.\ 2]{plavala2025polarization}, there is a Borel probability measure $\mu$ on $K_A\times K_B$ such that
\begin{align}
    y_k=\int x_A^{\otimes k}\otimes x_B^{\otimes k}\,d\mu(x_A,x_B).
\end{align}
For $\Pi=\operatorname{id}$, the limiting polarization equality is therefore
\begin{align}
    0=(\hat F_{\hookrightarrow}\otimes\hat F_{\hookrightarrow})\mathcal S_2(y_2)
    =\int\hat F_{\hookrightarrow}(x_A\otimes x_B)^{\otimes2}\,d\mu.
\end{align}
Choose inner products on $U_A$ and $U_B$ and contract with the induced inner product on their direct sum. Since the states are normalized, \autoref{eq:appendix_constraint_function} yields
\begin{align}\label{eq:proof_polarization_trick}
    0=\int\bigl(\lVert f_A(x_A)\rVert^2+\lVert f_B(x_B)\rVert^2\bigr)\,d\mu.
\end{align}
The nonnegative integrand vanishes almost everywhere, so $\mu(\mathcal F)=1$. The set of feasible products is compact; its convex hull is compact in finite dimension and contains $y_1$. Carath\'eodory's theorem therefore gives a finite decomposition into feasible products. In particular, $\langle P,y_1\rangle\geq p^*$. Together with $\langle P,y_1\rangle=L\leq p^*$, this proves optimality and convergence of the level values.
\end{proof}
Thus, with the identity polarization map, the limiting optimal state admits a decomposition in which every product component satisfies both local constraints.

\bibliographystyle{spmpsci}
\bibliography{references}

@article{Frenkel_2023,
   title={Integral formula for quantum relative entropy implies data processing inequality},
   volume={7},
   ISSN={2521-327X},
   DOI={10.22331/q-2023-09-07-1102},
   journal={Quantum},
   publisher={Verein zur Forderung des Open Access Publizierens in den Quantenwissenschaften},
   author={Frenkel, Péter E.},
   year={2023},
   month=sep, pages={1102} }

@article{Pl_vala_2023,
   title={General probabilistic theories: An introduction},
   volume={1033},
   ISSN={0370-1573},
   DOI={10.1016/j.physrep.2023.09.001},
   journal={Physics Reports},
   publisher={Elsevier BV},
   author={Plávala, Martin},
   year={2023},
   month=sep, pages={1–64} }

@article{Jen_ov__2024,
   title={Recoverability of quantum channels via hypothesis testing},
   volume={114},
   ISSN={1573-0530},
   DOI={10.1007/s11005-024-01775-2},
   number={1},
   journal={Letters in Mathematical Physics},
   publisher={Springer Science and Business Media LLC},
   author={Jenčová, Anna},
   year={2024},
   month=feb }

@misc{gouveia2025faces,
  title={Faces of homogeneous cones and applications to homogeneous chordality},
  author={Gouveia, Jo{\~a}o and Ito, Masaru and Louren{\c{c}}o, Bruno F},
  eprint={2501.09581},
  archivePrefix={arXiv},
  year={2025},
  doi={10.48550/arXiv.2501.09581},
  url={https://arxiv.org/abs/2501.09581}
}

@article{Umegaki1962RelativeEntropy,
  author    = {Umegaki, Hisaharu},
  title     = {Conditional expectation in an operator algebra. IV. Entropy and information},
  journal   = {Kodai Mathematical Seminar Reports},
  year      = {1962},
  volume    = {14},
  number    = {2},
  pages     = {59--85},
  doi       = {10.2996/kmj/1138844604},
  url       = {https://projecteuclid.org/journals/kodai-mathematical-journal/volume-14/issue-2/Conditional-expectation-in-an-operator-algebra-IV-Entropy-and-information/10.2996/kmj/1138844604.full},
  mrnumber  = {0153328},
  zbmath    = {0116.13002}
}

@article{Hiai1991,
  title = {The proper formula for relative entropy and its asymptotics in quantum probability},
  volume = {143},
  ISSN = {1432-0916},
  DOI = {10.1007/bf02100287},
  number = {1},
  journal = {Communications in Mathematical Physics},
  publisher = {Springer Science and Business Media LLC},
  author = {Hiai,  Fumio and Petz,  Dénes},
  year = {1991},
  month = dec,
  pages = {99–114}
}

@article{Berta_2011,
   title={The Quantum Reverse Shannon Theorem Based on One-Shot Information Theory},
   volume={306},
   ISSN={1432-0916},
   DOI={10.1007/s00220-011-1309-7},
   number={3},
   journal={Communications in Mathematical Physics},
   publisher={Springer Science and Business Media LLC},
   author={Berta, Mario and Christandl, Matthias and Renner, Renato},
   year={2011},
   month=aug, pages={579–615} }

@article{Brando2017,
  title = {Quantum {de Finetti} Theorems Under Local Measurements with Applications},
  volume = {353},
  ISSN = {1432-0916},
  DOI = {10.1007/s00220-017-2880-3},
  number = {2},
  journal = {Communications in Mathematical Physics},
  publisher = {Springer Science and Business Media LLC},
  author = {Brandão,  Fernando G. S. L. and Harrow,  Aram W.},
  year = {2017},
  month = apr,
  pages = {469–506}
}

@article{Weis2014,
author = {Stephan Weis},
title = {{Information topologies on non-commutative state spaces}},
journal = {Journal of convex analysis},
pages = {339--399},
year = {2014},
volume = {21},
number = {2},
issn = {0944-6532},
}

@article{christandl2007one,
  title = {One-and-a-Half Quantum {de Finetti} Theorems},
  volume = {273},
  ISSN = {1432-0916},
  DOI = {10.1007/s00220-007-0189-3},
  number = {2},
  journal = {Communications in Mathematical Physics},
  publisher = {Springer Science and Business Media LLC},
  author = {Christandl,  Matthias and K\"{o}nig,  Robert and Mitchison,  Graeme and Renner,  Renato},
  year = {2007},
  month = mar,
  pages = {473–498}
}

@inbook{Fuchs2004,
  title = {Unknown Quantum States and Operations, a Bayesian View},
  ISBN = {9783540444817},
  ISSN = {1616-6361},
  DOI = {10.1007/978-3-540-44481-7_5},
  booktitle = {Quantum State Estimation},
  publisher = {Springer Berlin Heidelberg},
  author = {Fuchs,  Christopher A. and Schack,  R\"{u}diger},
  year = {2004},
  month = aug,
  pages = {147–187}
}

@article{Aubrun_2021,
   title={Entangleability of cones},
   volume={31},
   ISSN={1420-8970},
   DOI={10.1007/s00039-021-00565-5},
   number={2},
   journal={Geometric and Functional Analysis},
   publisher={Springer Science and Business Media LLC},
   author={Aubrun, Guillaume and Lami, Ludovico and Palazuelos, Carlos and Plávala, Martin},
   year={2021},
   month=apr, pages={181–205} }

@article{Aubrun_2024,
   title={Monogamy of entanglement between cones},
   volume={391},
   ISSN={1432-1807},
   DOI={10.1007/s00208-024-02935-4},
   number={1},
   journal={Mathematische Annalen},
   publisher={Springer Science and Business Media LLC},
   author={Aubrun, Guillaume and Müller-Hermes, Alexander and Plávala, Martin},
   year = {2025},
   month=aug, pages={1591–1609} }

@misc{kossmann2025symmetric,
      title={On approximate quantum error correction for symmetric noise}, 
      author={Gereon Ko{\ss}mann and Julius A. Zeiss and Omar Fawzi and Mario Berta},
      year={2025},
      eprint={2507.12326},
      archivePrefix={arXiv},
      primaryClass={quant-ph},
      doi={10.48550/arXiv.2507.12326},
      url={https://arxiv.org/abs/2507.12326}, 
}

@misc{zeiss2025approximatingfixedsizequantum,
      title={Approximating fixed size quantum correlations in polynomial time}, 
      author={Julius A. Zeiss and Gereon Ko{\ss}mann and Omar Fawzi and Mario Berta},
      year={2025},
      eprint={2507.12302},
      archivePrefix={arXiv},
      primaryClass={quant-ph},
      doi={10.48550/arXiv.2507.12302},
      url={https://arxiv.org/abs/2507.12302}, 
}

@article{Tavakoli_2024,
   title={Semidefinite programming relaxations for quantum correlations},
   volume={96},
   ISSN={1539-0756},
   DOI={10.1103/revmodphys.96.045006},
   number={4},
   journal={Reviews of Modern Physics},
   publisher={American Physical Society (APS)},
   author={Tavakoli, Armin and Pozas-Kerstjens, Alejandro and Brown, Peter and Araújo, Mateus},
   year={2024},
   month=dec }

@article{Berta2021,
  title = {Semidefinite programming hierarchies for constrained bilinear optimization},
  volume = {194},
  ISSN = {1436-4646},
  DOI = {10.1007/s10107-021-01650-1},
  number = {1–2},
  journal = {Mathematical Programming},
  publisher = {Springer Science and Business Media LLC},
  author = {Berta,  Mario and Borderi,  Francesco and Fawzi,  Omar and Scholz,  Volkher B.},
  year = {2022},
  month = apr,
  pages = {781–829}
}

@article{ohst2024characterising,
  title={Characterising memory in quantum channel discrimination via constrained separability problems},
  author={Ohst, Ties-Albrecht and Zhang, Shijun and Nguyen, Hai Chau and Pl{\'a}vala, Martin and Quintino, Marco T{\'u}lio},
  journal={Quantum}, volume={10}, pages={1988}, year={2026},
  doi={10.22331/q-2026-01-28-1988}
}

@article{Fawzi2021,
  author  = {Fawzi, Hamza},
  title   = {The Set of Separable States has no Finite Semidefinite
             Representation Except in Dimension {$3\times 2$}},
  journal = {Communications in Mathematical Physics},
  volume  = {386},
  pages   = {1319--1335},
  year    = {2021},
  doi     = {10.1007/s00220-021-04163-2}
}

@book{Grtner2012,
  title = {Approximation Algorithms and Semidefinite Programming},
  ISBN = {9783642220159},
  DOI = {10.1007/978-3-642-22015-9},
  publisher = {Springer Berlin Heidelberg},
  author = {G\"{a}rtner,  Bernd and Matousek,  Jiri},
  year = {2012}
}

@inproceedings{THOMAS2019,
  title = {Spectrahedral Lifts Of Convex Sets},
  DOI = {10.1142/9789813272880_0202},
  booktitle = {Proceedings of the International Congress of Mathematicians (ICM 2018)},
  publisher = {WORLD SCIENTIFIC},
  author = {Thomas,  Rekha R.},
  year = {2019},
  month = may,
  pages = {3819–3842}
}

@article{plavala2025polarization,
  title={The polarization hierarchy for polynomial optimization over convex bodies, with applications to nonnegative matrix rank},
  author={Pl{\'a}vala, Martin and Ligthart, Laurens T and Gross, David},
  journal={Linear Algebra and its Applications},
  volume = {723},
  pages = {15-32},
  year={2025},
  doi={10.1016/j.laa.2025.05.019}
}

@misc{zeiss2026commentpolarization,
  title = {Comment on ``The polarization hierarchy for polynomial optimization over convex bodies, with applications to nonnegative matrix rank''},
  author = {Zeiss, Julius A.},
  year = {2026},
  eprint = {2610.24542},
  archivePrefix = {arXiv},
  primaryClass = {math.OC},
  doi = {10.48550/arXiv.2610.24542},
  url = {https://arxiv.org/abs/2610.24542}
}

@article{jee2020quasi,
  doi = {10.4230/LIPICS.ICALP.2021.82},
  url = {https://drops.dagstuhl.de/entities/document/10.4230/LIPIcs.ICALP.2021.82},
  author = {Jee,  Hyejung H. and Sparaciari,  Carlo and Fawzi,  Omar and Berta,  Mario},
  language = {en},
  title = {Quasi-Polynomial Time Algorithms for Free Quantum Games in Bounded Dimension},
  journal = {LIPIcs,  Volume 198,  ICALP 2021},
  volume = {198},
  pages = {82:1--82:20},
  publisher = {Schloss Dagstuhl – Leibniz-Zentrum f\"{u}r Informatik},
  year = {2021},
  copyright = {Creative Commons Attribution 4.0 International license}
}

@book{Hug2020,
  title = {Lectures on Convex Geometry},
  ISBN = {9783030501808},
  ISSN = {2197-5612},
  DOI = {10.1007/978-3-030-50180-8},
  journal = {Graduate Texts in Mathematics},
  publisher = {Springer International Publishing},
  author = {Hug,  Daniel and Weil,  Wolfgang},
  year = {2020}
}

@article{Doherty_2004,
	doi = {10.1103/physreva.69.022308},
	year = 2004,
	month = {2},
  
	publisher = {American Physical Society ({APS})},
  
	volume = {69},
  
	number = {2},
  
	author = {Andrew C. Doherty and Pablo A. Parrilo and Federico M. Spedalieri},
  
	title = {Complete family of separability criteria},
  
	journal = {Physical Review A}
}

@article{Sherali1990,
  title = {A Hierarchy of Relaxations between the Continuous and Convex Hull Representations for Zero-One Programming Problems},
  volume = {3},
  ISSN = {1095-7146},
  DOI = {10.1137/0403036},
  number = {3},
  journal = {SIAM Journal on Discrete Mathematics},
  publisher = {Society for Industrial & Applied Mathematics (SIAM)},
  author = {Sherali,  Hanif D. and Adams,  Warren P.},
  year = {1990},
  month = aug,
  pages = {411–430}
}

@article{Sherali1992,
  title = {A global optimization algorithm for polynomial programming problems using a Reformulation-Linearization Technique},
  volume = {2},
  ISSN = {1573-2916},
  DOI = {10.1007/bf00121304},
  number = {1},
  journal = {Journal of Global Optimization},
  publisher = {Springer Science and Business Media LLC},
  author = {Sherali,  Hanif D. and Tuncbilek,  Cihan H.},
  year = {1992},
  pages = {101–112}
}

@article{Lasserre2001,
  title = {Global Optimization with Polynomials and the Problem of Moments},
  volume = {11},
  ISSN = {1095-7189},
  DOI = {10.1137/s1052623400366802},
  number = {3},
  journal = {SIAM Journal on Optimization},
  publisher = {Society for Industrial & Applied Mathematics (SIAM)},
  author = {Lasserre,  Jean B.},
  year = {2001},
  month = jan,
  pages = {796–817}
}

@article{deKlerk2011,
  title = {On the Lasserre Hierarchy of Semidefinite Programming Relaxations of Convex Polynomial Optimization Problems},
  volume = {21},
  ISSN = {1095-7189},
  DOI = {10.1137/100814147},
  number = {3},
  journal = {SIAM Journal on Optimization},
  publisher = {Society for Industrial & Applied Mathematics (SIAM)},
  author = {de Klerk,  Etienne and Laurent,  Monique},
  year = {2011},
  month = jul,
  pages = {824–832}
}

@article{Nie2013,
  title = {Optimality conditions and finite convergence of Lasserre’s hierarchy},
  volume = {146},
  ISSN = {1436-4646},
  DOI = {10.1007/s10107-013-0680-x},
  number = {1–2},
  journal = {Mathematical Programming},
  publisher = {Springer Science and Business Media LLC},
  author = {Nie,  Jiawang},
  year = {2013},
  month = may,
  pages = {97–121}
}

@article{Pironio2010,
  title = {Convergent Relaxations of Polynomial Optimization Problems with Noncommuting Variables},
  volume = {20},
  ISSN = {1095-7189},
  DOI = {10.1137/090760155},
  number = {5},
  journal = {SIAM Journal on Optimization},
  publisher = {Society for Industrial & Applied Mathematics (SIAM)},
  author = {Pironio,  S. and Navascués,  M. and Acín,  A.},
  year = {2010},
  month = jan,
  pages = {2157–2180}
}

@inbook{nesterov2000squared,
  title = {Squared Functional Systems and Optimization Problems},
  ISBN = {9781475732160},
  ISSN = {1384-6485},
  DOI = {10.1007/978-1-4757-3216-0_17},
  booktitle = {High Performance Optimization},
  publisher = {Springer US},
  author = {Nesterov,  Yurii},
  year = {2000},
  pages = {405–440}
}

@article{Klep_2023,
   title={State polynomials: positivity, optimization and nonlinear Bell inequalities},
   volume={207},
   ISSN={1436-4646},
   DOI={10.1007/s10107-023-02024-5},
   number={1–2},
   journal={Mathematical Programming},
   publisher={Springer Science and Business Media LLC},
   author={Klep, Igor and Magron, Victor and Volčič, Jurij and Wang, Jie},
   year={2023},
   month=nov, pages={645–691} }

@article{Ligthart2023,
  title = {The inflation hierarchy and the polarization hierarchy are complete for the quantum bilocal scenario},
  volume = {64},
  ISSN = {1089-7658},
  DOI = {10.1063/5.0143792},
  number = {7},
  journal = {Journal of Mathematical Physics},
  publisher = {AIP Publishing},
  author = {Ligthart,  Laurens T. and Gross,  David},
  year = {2023},
  month = jul 
}

@article{Ligthart2023_inflation,
  title = {A Convergent Inflation Hierarchy for Quantum Causal Structures},
  volume = {401},
  ISSN = {1432-0916},
  DOI = {10.1007/s00220-023-04697-7},
  number = {3},
  journal = {Communications in Mathematical Physics},
  publisher = {Springer Science and Business Media LLC},
  author = {Ligthart,  Laurens T. and Gachechiladze,  Mariami and Gross,  David},
  year = {2023},
  month = jun,
  pages = {2673–2714}
}

@article{Barrett2009,
  title = {The {de Finetti} theorem for test spaces},
  volume = {11},
  ISSN = {1367-2630},
  DOI = {10.1088/1367-2630/11/3/033024},
  number = {3},
  journal = {New Journal of Physics},
  publisher = {IOP Publishing},
  author = {Barrett,  Jonathan and Leifer,  Matthew},
  year = {2009},
  month = mar,
  pages = {033024}
}

@article{gharibian2009strongnphardnessquantumseparability,
  title = {Strong NP-hardness of the quantum separability problem},
  volume = {10},
  ISSN = {1533-7146},
  DOI = {10.26421/qic10.3-4-11},
  number = {3 and 4},
  journal = {Quantum Information and Computation},
  publisher = {Rinton Press},
  author = {Gharibian,  S.},
  year = {2010},
  month = mar,
  pages = {343–360}
}

@article{ioannou2007computationalcomplexityquantumseparability,
  title = {Computational complexity of the quantum separability problem},
  volume = {7},
  ISSN = {1533-7146},
  DOI = {10.26421/qic7.4-5},
  number = {4},
  journal = {Quantum Information and Computation},
  publisher = {Rinton Press},
  author = {Ioannou,  L.M.},
  year = {2007},
  month = may,
  pages = {336–370}
}

@book{Ryan2002,
  title = {Introduction to Tensor Products of Banach Spaces},
  ISBN = {9781447139034},
  ISSN = {1439-7382},
  DOI = {10.1007/978-1-4471-3903-4},
  journal = {Springer Monographs in Mathematics},
  publisher = {Springer London},
  author = {Ryan,  Raymond A.},
  year = {2002}
}

@article{Puchaa2015,
  title = {Exploring boundaries of quantum convex structures: Special role of unitary processes},
  volume = {92},
  ISSN = {1094-1622},
  DOI = {10.1103/physreva.92.012304},
  number = {1},
  journal = {Physical Review A},
  publisher = {American Physical Society (APS)},
  author = {Puchała,  Zbigniew and Jenčová,  Anna and Sedlák,  Michal and Ziman,  Mário},
  year = {2015},
  month = jul 
}

@article{kossmann2025optimisingrelativeentropysemidefinite,
  title={Optimising the relative entropy under semidefinite constraints},
  author={Ko{\ss}mann, Gereon and Schwonnek, Ren{\'e}},
  journal={npj Quantum Information},
  volume={12}, pages={23}, year={2026},
  doi={10.1038/s41534-026-01184-4}
}

@article{Jenov2022,
  title = {Assemblages and steering in general probabilistic theories},
  volume = {55},
  ISSN = {1751-8121},
  DOI = {10.1088/1751-8121/ac97ce},
  number = {43},
  journal = {Journal of Physics A: Mathematical and Theoretical},
  publisher = {IOP Publishing},
  author = {Jenčová,  Anna},
  year = {2022},
  month = oct,
  pages = {434001}
}

@article{achenbach2025factorizationmultimetersunifiedview,
  title={Factorization of multimeters: a unified view on nonclassical quantum phenomena},
  author={Achenbach, Tim and Bluhm, Andreas and Lepp{\"a}j{\"a}rvi, Leevi and Nechita, Ion and Pl{\'a}vala, Martin},
  journal={Letters in Mathematical Physics}, volume={116}, pages={56}, year={2026},
  doi={10.1007/s11005-026-02088-2}
}

@article{Gisin1989StochasticQD,
  title={Stochastic quantum dynamics and relativity},
  author={Nicolas Gisin},
  journal={Helvetica Physica Acta},
  year={1989},
  volume={62},
  pages={363-371},
  url={https://api.semanticscholar.org/CorpusID:115540179}
}

@article{Hughston1993,
  title = {A complete classification of quantum ensembles having a given density matrix},
  volume = {183},
  ISSN = {0375-9601},
  DOI = {10.1016/0375-9601(93)90880-9},
  number = {1},
  journal = {Physics Letters A},
  publisher = {Elsevier BV},
  author = {Hughston,  Lane P. and Jozsa,  Richard and Wootters,  William K.},
  year = {1993},
  month = nov,
  pages = {14–18}
}

@article{Barnum2013,
  title = {Ensemble Steering,  Weak Self-Duality,  and the Structure of Probabilistic Theories},
  volume = {43},
  ISSN = {1572-9516},
  DOI = {10.1007/s10701-013-9752-2},
  number = {12},
  journal = {Foundations of Physics},
  publisher = {Springer Science and Business Media LLC},
  author = {Barnum,  Howard and Gaebler,  Carl Philipp and Wilce,  Alexander},
  year = {2013},
  month = oct,
  pages = {1411–1427}
}

@article{Jenov2018,
  title = {Incompatible measurements in a class of general probabilistic theories},
  volume = {98},
  ISSN = {2469-9934},
  DOI = {10.1103/physreva.98.012133},
  number = {1},
  journal = {Physical Review A},
  publisher = {American Physical Society (APS)},
  author = {Jenčová,  Anna},
  year = {2018},
  month = jul 
}

@article{Uola2020,
  title = {Quantum steering},
  volume = {92},
  ISSN = {1539-0756},
  DOI = {10.1103/revmodphys.92.015001},
  number = {1},
  journal = {Reviews of Modern Physics},
  publisher = {American Physical Society (APS)},
  author = {Uola,  Roope and Costa,  Ana C. S. and Nguyen,  H. Chau and G\"{u}hne,  Otfried},
  year = {2020},
  month = mar 
}

@article{Stevens2014,
  title = {Steering,  incompatibility,  and Bell-inequality violations in a class of probabilistic theories},
  volume = {89},
  ISSN = {1094-1622},
  DOI = {10.1103/physreva.89.022123},
  number = {2},
  journal = {Physical Review A},
  publisher = {American Physical Society (APS)},
  author = {Stevens,  Neil and Busch,  Paul},
  year = {2014},
  month = feb 
}

@article{kimura2010distinguishability,
  title={Distinguishability measures and entropies for general probabilistic theories},
  author={Kimura, Gen and Nuida, Koji and Imai, Hideki},
  journal={Reports on Mathematical Physics},
  volume={66},
  number={2},
  pages={175--206},
  year={2010},
  doi={10.1016/S0034-4877(10)00025-X}
}

@article{kimura2016entropies,
  title={Entropies in general probabilistic theories and their application to the Holevo bound},
  author={Kimura, Gen and Ishiguro, Junji and Fukui, Makoto},
  journal={Physical Review A},
  volume={94},
  number={4},
  pages={042113},
  year={2016},
  doi={10.1103/PhysRevA.94.042113}
}

@article{takakura2019entropy,
  title={Entropy of mixing exists only for classical and quantum-like theories among the regular polygon theories},
  author={Takakura, Ryo},
  journal={Journal of Physics A: Mathematical and Theoretical},
  volume={52},
  number={46},
  pages={465302},
  year={2019},
  doi={10.1088/1751-8121/ab4a2e}
}

@article{Janotta_2011,
   title={Limits on nonlocal correlations from the structure of the local state space},
   volume={13},
   ISSN={1367-2630},
   DOI={10.1088/1367-2630/13/6/063024},
   number={6},
   journal={New Journal of Physics},
   publisher={IOP Publishing},
   author={Janotta, Peter and Gogolin, Christian and Barrett, Jonathan and Brunner, Nicolas},
   year={2011},
   pages={063024}
}

@article{Huangfu_2018,
   title={Parallelizing the dual revised simplex method},
   volume={10},
   ISSN={1867-2957},
   DOI={10.1007/s12532-017-0130-5},
   number={1},
   journal={Mathematical Programming Computation},
   author={Huangfu, Qi and Hall, J. A. Julian},
   year={2018},
   pages={119--142}
}

@article{Clauser1969,
  title = {Proposed Experiment to Test Local Hidden-Variable Theories},
  author = {Clauser, John F. and Horne, Michael A. and Shimony, Abner and Holt, Richard A.},
  journal = {Physical Review Letters},
  volume = {23},
  number = {15},
  pages = {880--884},
  year = {1969},
  DOI = {10.1103/PhysRevLett.23.880}
}

@article{Popescu1994,
  title = {Quantum nonlocality as an axiom},
  author = {Popescu, Sandu and Rohrlich, Daniel},
  journal = {Foundations of Physics},
  volume = {24},
  number = {3},
  pages = {379--385},
  year = {1994},
  DOI = {10.1007/BF02058098}
}

@article{Christandl2009,
  title = {Finite {de Finetti} theorem for conditional probability distributions describing physical theories},
  author = {Christandl, Matthias and Toner, Ben},
  journal = {Journal of Mathematical Physics},
  volume = {50},
  number = {4},
  pages = {042104},
  year = {2009},
  eprint = {0712.0916},
  archivePrefix = {arXiv},
  doi = {10.1063/1.3114986}
}

@misc{JuliusAZeissGithub_GPT_Games,
  author  = {Julius A. Zeiss},
  title   = {Numerics for "Finite {de Finetti} for convex bodies and polynomial optimization"},
  year    = {2026},
  url     = {https://github.com/JuliusAZeiss/FINITE-DE-FINETTI-FOR-CONVEX-BODIES-AND-POLYNOMIAL-OPTIMIZATION-},
  urldate = {2026-08-03},
}
\end{document}